\documentclass[11pt]{article}
\usepackage[hscale=0.8,vscale=0.8]{geometry}
\usepackage{amsmath}
\usepackage{amsfonts}
\usepackage{latexsym}
\usepackage{graphicx}
\usepackage{caption}
\usepackage{subcaption}
\usepackage{multicol}
\usepackage{multirow}
\usepackage{amssymb}
\usepackage{mathtools}
\usepackage{hhline}
\usepackage{amsthm}
\usepackage{float}
\usepackage{cite}
\usepackage{color}
\usepackage{chngcntr}
\usepackage{tikz}
\usepackage{xcolor}
\usepackage{array}
\usepackage{comment}
\usepackage{geometry}
\usepackage{bm}
\usepackage{epsfig}

\graphicspath{{./figs/}}

\numberwithin{equation}{section}
\allowdisplaybreaks

\allowdisplaybreaks
\newtheorem{lemma}{Lemma}[section]
\newtheorem{theorem}{Theorem}[section]
\newtheorem{corollary}{Corollary}[section]
\newtheorem{remark}{Remark}[section]

\newcommand{\ds}{\displaystyle}

\newcommand{\vertiii}[1]{{\left\vert\kern-0.25ex\left\vert\kern-0.25ex\left\vert #1 
    \right\vert\kern-0.25ex\right\vert\kern-0.25ex\right\vert}}
\newcommand{\bgamma}{{\boldsymbol\gamma}}

\newcommand{\bLambda}{{\boldsymbol\Lambda}}

\newcommand{\bbeta}{{\boldsymbol\eta}}
\newcommand{\bsi}{{\boldsymbol\sigma}}

\newcommand{\bphi}{{\boldsymbol\phi}}

\newcommand{\bvarphi}{{\boldsymbol\varphi}}
\newcommand{\bpsi}{{\boldsymbol\psi}}

\newcommand{\btau}{{\boldsymbol\tau}}

\newcommand{\bchi}{{\boldsymbol\chi}}

\newcommand{\btheta}{{\boldsymbol\theta}}

\newcommand{\brho}{{\boldsymbol\rho}}
\newcommand{\bv}{{\mathbf{v}}}
\newcommand{\bw}{{\mathbf{w}}}
\newcommand{\f}{\mathbf{f}}

\newcommand{\bp}{\mathbf{p}}
\newcommand{\bq}{\mathbf{q}}
\newcommand{\br}{\mathbf{r}}
\newcommand{\bu}{\mathbf{u}}

\newcommand{\bt}{{\mathbf{t}}}
\newcommand{\bn}{{\mathbf{n}}}

\def\bs{\mathbf{s}}
\newcommand{\0}{{\mathbf{0}}}

\def\bD{\mathbf{D}}

\def\bF{\mathbf{F}}
\def\bG{\mathbf{G}}
\def\bK{\mathbf{K}}
\def\bI{\mathbf{I}}

\def\bV{\mathbf{V}}

\def\bQ{\mathbf{Q}}
\def\bS{\mathbf{S}}

\def\bZ{\mathbf{Z}}

\def\bx{\mathbf{x}}
\newcommand{\bL}{\mathbf{L}}
\newcommand\bH{\mathbf{H}}
\newcommand\bbN{\mathbb{N}}
\newcommand\bbM{\mathbb{M}}

\newcommand\bbQ{\mathbb{Q}}
\newcommand\bbR{\mathbb{R}}
\newcommand\bbS{\mathbb{S}}
\newcommand\bbH{\mathbb{H}}
\newcommand\bbX{\mathbb{X}}

\newcommand\bbL{\mathbb{L}}

\newcommand{\bbZ}{\mathbb{Z}}

\newcommand{\cA}{\mathcal{A}}
\newcommand{\cB}{\mathcal{B}}

\newcommand{\cD}{\mathcal{D}}
\newcommand{\cE}{\mathcal{E}}

\newcommand{\cL}{\mathcal{L}}
\newcommand{\cM}{\mathcal{M}}
\newcommand{\cN}{\mathcal{N}}
\newcommand{\cO}{\mathcal{O}}

\newcommand{\cR}{\mathcal{R}}

\def\H{\mathrm{H}}
\def\L{\mathrm{L}}

\def\V{\mathrm{V}}

\def\rP{\mathrm{P}}
\def\BDM{\mathrm{BDM}}
\def\RT{\mathrm{RT}}
\def\W{\mathrm{W}}
\def\rt{\mathrm{t}}

\def\BJS{\mathtt{BJS}}

\def\tr{\mathrm{tr}\,}

\def\div{\mathrm{div}}
\def\dist{\mathrm{dist}\,}

\def\pil{\left<}
\def\pir{\right>}

\def\sk{\mathrm{sk}}

\def\RT{\mathrm{RT}}

\def\qin{{\quad\hbox{in}\quad}}
\def\qon{{\quad\hbox{on}\quad}}
\def\qan{{\quad\hbox{and}\quad}}

\def\wt{\widetilde}
\def\wh{\widehat}
\def\ov{\overline}
\newcommand{\fh}{{fh}}
\newcommand{\ph}{{ph}}
\newcommand{\sh}{{sh}}
\title{A mixed elasticity formulation for fluid--poroelastic structure interaction}

\author{
  {\sc Tongtong Li}\thanks{Department of Mathematics, University of Pittsburgh, Pittsburgh, PA 15260, USA, email: {\tt \{tol24@pitt.edu, yotov@math.pitt.edu\} }. Supported in part by NSF
    grant DMS 1818775.}	
\quad
    {\sc Ivan Yotov}\footnotemark[1]~}

\date{\today}

\begin{document}

\maketitle

\begin{abstract} 
We develop a mixed finite element method for the coupled problem
arising in the interaction between a free fluid governed by the Stokes
equations and flow in deformable porous medium modeled by the Biot
system of poroelasticity. Mass conservation, balance of stress, and
the Beavers--Joseph--Saffman condition are imposed on the interface. We
consider a fully mixed Biot formulation based on a weakly symmetric
stress-displacement-rotation elasticity system and Darcy
velocity-pressure flow formulation. A velocity-pressure formulation is
used for the Stokes equations. The interface conditions are
incorporated through the introduction of the traces of the structure
velocity and the Darcy pressure as Lagrange multipliers. Existence and
uniqueness of a solution are established for the continuous weak
formulation. Stability and error estimates are derived for the
semi-discrete continuous-in-time mixed finite element
approximation. Numerical experiments are presented to verify
the theoretical results and illustrate the robustness
of the method with respect to the physical parameters.
\end{abstract}

\section{Introduction}
In this paper we develop a new mixed elasticity formulation for the
quasi-static Stokes--Biot problem that models the interaction between a
free fluid and flow in deformable porous medium. This coupled physical
phenomenon is referred to as fluid--poroelastic structure interaction
(FPSI). There has been an increased interest in this problem in recent
years, due to its wide range of applications in petroleum engineering,
hydrology, environmental sciences, and biomedical engineering, such
as predicting and controlling processes arising in gas and oil
extraction from naturally or hydraulically fractured reservoirs,
cleanup of groundwater flow in deformable aquifers, designing
industrial filters, and modeling blood-vessel interactions in blood
flows. The free fluid is modeled by the Stokes equations, while the
flow in the deformable porous media is modeled by the Biot system of
poroelasticity \cite{Biot1941}. The Biot system couples an elasticity
equation for the deformation of the elastic porous matrix with a Darcy
flow model for the mass conservation of the fluid in the pores. The
Stokes and Biot regions are coupled via interface conditions enforcing
continuity of normal flux, the Beavers--Joseph--Saffman (BJS) slip with
friction condition for the tangential velocity, balance of forces, and
continuity of normal stress. The FPSI system exhibits features of both
coupled Stokes--Darcy flows \cite{dmq2002, ry2005, vwy2014, lsy2003,
  galvis2007, gmo2009, ervin2009} and fluid--structure interaction (FSI)
\cite{galdi2010fundamental,bazilevs2013computational,bungartz2006fluid,richter2017fluid}, both of which have been extensively studied.
In applications of the Stokes--Biot model to flow in fractured poroelastic media, the use of the Stokes model in the fractures provides a more accurate alternative to the traditional Darcy model \cite{Martin2005}, which becomes inadequate for faster flow and higher porosity.

The first mathematical analysis of the Stokes--Biot system can be found
in \cite{Showalter2005}, where a fully dynamic system is considered
and well-posedness is shown by rewriting it as a parabolic system. A
numerical study was presented in \cite{bqq2009}, using the
Navier-Stokes equations to model the free fluid flow. The authors
develop a variational multiscale finite element method and propose
both monolithic and iterative partitioned methods for the solution of
the coupled system. A non-iterative operator splitting scheme is
developed in \cite{byz2015} for an arterial flow model that includes a
thin elastic membrane separating the two regions, using a pressure
formulation for the flow in the poroelastic region. In
\cite{byzz2015,BYZZ-MSA}, a mixed Darcy model is considered in the
Biot system and the Nitsche's interior penalty method is used to impose weakly
the continuity of normal flux. A Lagrange multiplier formulation for
imposing the normal flux continuity is developed in
\cite{AKYZ-LNCC,akyz2018}. A decoupling algorithm based on solving an
optimization problem is developed in \cite{Cesm-etal-optim}.  A
dimensionally reduced Brinkman--Biot model for flow through fractures
in poroelastic media is developed and analyzed in
\cite{Buk-Yot-Zun-fracture}. The well-posedness of the fully dynamic
coupled Navier-Stokes/Biot model using a pressure Darcy formulation is
established in \cite{cesm2017}. A finite element method for this
formulation is developed in \cite{Cesm-Chid}.  A nonlinear Stokes--Biot
model for non-Newtonian fluids and its finite element approximation are
considered in \cite{aeny2019}, where the first well-posedness analysis
of the quasi-static Stokes--Biot system is presented. Coupling of the
Stokes--Biot system with transport is studied in \cite{fpsi-transport}.
A second order in time decoupling scheme for a nonlinear Stokes--Biot
model is developed in \cite{Kunwar-etal}. Recent works study various
discretization schemes for the Stokes--Biot system, including a coupled
discontinuous Galerkin -- mixed finite element method \cite{Wen-He}, a
staggered finite element method \cite{Bergkamp-etal} and
non-conforming finite element method \cite{wilfrid2020nonconforming}.

To the best of our knowledge, all of the previous works consider
displacement-based discretizations of the elasticity equation in the
Biot system. In this paper we develop a mixed finite element
discretization of the quasi-static Stokes--Biot system using a mixed
elasticity formulation with a weakly symmetric poroelastic stress. The
advantages of mixed finite element methods for elasticity include
locking-free behavior, robustness with respect to the physical
parameters, local momentum conservation, and accurate stress
approximations with continuous normal components across element edges
or faces. Here we consider a three-field
stress--displacement--rotation elasticity formulation. This
formulation allows for mixed finite element methods with reduced
number of degrees of freedom, see e.g. \cite{arnold2007mixed,arnold2015}.
It is also the basis for the multipoint stress mixed finite element
method \cite{msmfe-simpl,msmfe-quads}, where stress and rotation can
be locally eliminated, resulting in a positive definite cell-centered
scheme for the displacement. We consider a mixed velocity--pressure
Darcy formulation, resulting in a five-field Biot formulation, which
was proposed in \cite{lee2016} and studied further in
\cite{msfmfe-Biot}, where a multipoint stress-flux mixed finite
element method is developed. We note that our analysis can be easily
extended to the strongly symmetric mixed elasticity formulation, which
leads to the four-field mixed Biot formulation developed in
\cite{Yi-Biot-mixed}. Finally, for the Stokes equations we consider
the classical velocity--pressure formulation. The weak formulation for
the resulting Stokes--Biot system has not been studied in the
literature.  One main difference from the previous works with
displacement-based elasticity formulations \cite{akyz2018, aeny2019} is
that the normal component of the poroelastic stress appears explicitly
in the interface terms. Correspondingly, we introduce a Lagrange multiplier
with a physical meaning of structure velocity that is used to impose
weakly the balance of force and the BJS condition. In addition, a
Darcy pressure Lagrange multiplier is used to impose weakly the
continuity of normal flux.

Since the weak formulation of the Stokes--Biot system considered in
this paper is new, we first show that it has a unique solution. This
is done by casting it in the form of a degenerate evolution saddle
point system and employing results from classical semigroup theory for
differential equations with monotone operators \cite{showalter}.  We
then present a semi-discrete continuous-in-time formulation, which is
based on employing stable mixed finite element spaces for the Stokes,
Darcy, and elasticity equations on grids that may be non-matching
along the interface, as well as suitable choices for the Lagrange
multiplier finite element spaces. Well-posedness of the semidiscrete
formulation is established with a similar argument to the continuous
case, using discrete inf-sup conditions for the divergence and
interface bilinear forms. Stability and optimal order error estimates
are then derived for all variables in their natural space-time
norms. We emphasize that the estimates hold uniformly in the limit of
the storativity coefficient $s_0$ going to zero, which is a locking
regime for non-mixed elasticity discretizations for the Biot
system. In addition, our results are robust with respect to
$a_{\min}$, the lower bound for the compliance tensor $A$, which
relates to another locking phenomena in poroelasticity called Poisson
locking \cite{ Yi-Biot-locking}. Furthermore, we do not use Gronwall's inequality
in the stability bound, thus obtaining long-time stability for our
method. We present several computational experiments for a fully discrete finite
element method designed to verify the convergence theory, illustrate
the behavior of the method for a problem modeling an interaction between
surface and subsurface hydrological systems, and study the robustness
of the method with respect to the physical parameters. In particular,
the numerical experiments illustrate the locking-free properties
of the mixed finite element method for the Stokes--Biot system.

The rest of the paper is organized as follows. In
Section~\ref{sec:model} we present the mathematical
model. Section~\ref{sec:weak} is devoted to the continuous weak
formulation. Well-posedness of the continuous formulation is proved in
Section~\ref{sec:well-posed}, where existence and uniqueness of
solution are established. The semidiscrete continuous-in-time
approximation is introduced in
Section~\ref{sec:semi-discrete}. Stability and error analyses are
performed in Sections \ref{sec:stability} and \ref{sec:error},
respectively. Numerical experiments are presented in
Section~\ref{sec:numerical}, followed by conclusions in
Section~\ref{sec:conclusions}. 

We end this section by fixing some notation. Let $\mathbb{M}$, $\mathbb{S}$ and $\mathbb{N}$
denote the sets of $n \times n$ matrices, $n \times n$ symmetric matrices and $n \times n$ skew-symmetric matrices,
respectively. For a domain $\cO \subset
\bbR^n$, we make use of the usual notation for Lebesgue spaces
$\L^p(\cO)$, Sobolev spaces $\W^{k,p}(\cO)$, and Hilbert spaces
$\H^k(\cO)$. The corresponding norms are denoted by $\| \cdot
\|_{\L^p(\cO)}$, $\| \cdot \|_{\W^{k,p}(\cO)}$ and $\| \cdot
\|_{\H^k(\cO)}$. For a generic scalar space Z, we denote by $\bZ$ and
$\bbZ$ the corresponding vector and tensor counterparts,
respectively. The $\L^2(\cO)$ inner product is denoted by $(\cdot,
\cdot)_{\cO}$ for scalar, vector and tensor valued functions. For a
section of the boundary $S \subset \partial \cO$, we write $\langle
\cdot, \cdot \rangle_{S}$ for the $\L^2(S)$ inner product or duality
pairing. We will also use the Hilbert space
\begin{equation*}
  \bH(\div;\cO):=\left\{\bv\in \bL^2(\cO): \, \nabla \cdot \bv\in \L^2(\cO)\right\},
\end{equation*}
endowed with the norm $\|\bv\|^2_{\bH(\div;\cO)} := \|\bv\|_{\bL^2(\cO)}^2 + \|\nabla\cdot \bv\|_{\L^2(\cO)}^2$, as well as its tensor-valued counterpart $\bbH(\div;\cO)$ consisting of matrices with rows in $\bH(\div;\cO)$. The latter is equipped with the norm $\|\btau\|^2_{\bbH(\div;\cO)} := \|\btau\|_{\bbL^2(\cO)}^2 + \|\nabla\cdot \btau\|_{\bL^2(\cO)}^2$.  
Given a separable Banach space $\V$ endowed with the norm $\|\cdot\|_{\V}$, we let $\L^{p}(0,T;\V)$ be the space of functions $f : (0,T)\to \V$ that are Bochner measurable and such that $\|f\|_{\L^{p}(0,T;\V)} < \infty$, with
\begin{equation*}
\|f\|^{p}_{\L^{p}(0,T;\V)} := \int^T_0 \|f(t)\|^{p}_{\V} \, dt,\quad
\|f\|_{\L^\infty(0,T;\V)} \,:=\mathop{\mathrm{ess\,sup}} \limits_{t\in [0,T]} \|f(t)\|_{\V}.
\end{equation*}
We employ $\0$ to denote the null vector or tensor, and use $C$ and $c$, with or without subscripts, bars, tildes or hats, to denote generic constants independent of the discretization parameters, which may take different values at different places.

%%%%%%%%%%%%%%%%%%%%%%%%%%%%%%%%%%%%%%%%%%%%%%%%%%%%%%%%%%%%%%
%%%%%%%%%%%%%%%%%%%%%%%%%%%%%%%%%%%%%%%%%%%%%%%%%%%%%%%%%%%%%%
%%%%%%%%%%%%%%%%%%%%%%%%%%%%%%%%%%%%%%%%%%%%%%%%%%%%%%%%%%%%%%
\section{Stokes--Biot model problem}\label{sec:model}

\begin{comment}
\begin{center}
\begin{tikzpicture}
  \draw (0,0) rectangle (2,2);
  \draw (1,1) node
                       {poroe};
  \draw (0,0) rectangle (2,4);
  \draw (1,3) node
                       {fluid};
  \draw[red, thick]
  (0,2) -- (0,0) -- (2,0) -- (2,2);
  \draw (2.5,1) node
                       {{\color{red}{$\Gamma_p$}}};
  \draw[blue, thick]
  (0,2) -- (2,2);
  \draw (2.5,2) node
                       {{\color{blue}{$\Gamma_{fp}$}}};
  \draw[black!60!green, thick]
  (0,2) -- (0,4) -- (2,4) -- (2,2);
  \draw (2.5,3) node
                       {{\color{black!60!green}{$\Gamma_f$}}};
\end{tikzpicture}
\end{center}
\end{comment}

Let $\Omega \subseteq \bbR^n$, $n=2$ or $3$, be a connected domain
that consists of two non-overlapping regions, the fluid part $\Omega_f$
and the poroelastic part $\Omega_p$. Let
$\Gamma_f=\partial\Omega_f\cap\partial\Omega$,
$\Gamma_{fp}=\partial\Omega_f\cap\partial\Omega_p$, and
$\Gamma_p=\partial\Omega_p\cap\partial\Omega$.

The free fluid in $\Omega_f$ is governed by the Stokes equations
\begin{subequations}\label{eq:model1}
\begin{gather}
-\nabla\cdot \bsi_f=\f_f, \quad 
\nabla\cdot \bu_f=q_f \qin \Omega_f\times(0,T], \label{eq: model problem 1a} \\[1ex]
\bu_f=\0 \qon \Gamma_f\times(0,T], \label{eq: model problem 1b} 
\end{gather}
\end{subequations}
where $T>0$ is the final time, $\bu_f$ is the fluid velocity, $p_f$ is
the fluid pressure, and $\ds \bsi_f=-p_f \bI + 2\mu \bD(\bu_f)$ is the
stress tensor. Here $\ds \bD(\bu_f)=\frac{1}{2}(\nabla \bu_f + \nabla
\bu_f^\rt)$ is the deformation rate tensor and $\mu>0$ is the fluid
viscosity. In addition, $\f_f$ is a fluid body force and $q_f$ is an external
source or sink term.

The poroelastic region is governed by the quasi-static Biot system \cite{Biot1941}
\begin{subequations}\label{eq:model2}
\begin{gather}
\ds -\nabla\cdot \bsi_p=\f_p, \quad 
\mu \bK^{-1}\bu_p + \nabla p_p=\0, \quad
\frac{\partial}{\partial t}(s_0 p_p+\alpha \nabla \cdot \bbeta_p)+\nabla \cdot \bu_p=q_p
\qin \Omega_p\times(0,T], \label{eq: model problem 2a} \\
\ds p_p=0 \qon \Gamma_p^{D_p}\times(0,T], \quad
\bu_p\cdot \bn_p=0 \qon \Gamma_p^{N_v}\times(0,T], \label{eq: model problem 2b} \\[1ex]
  \ds \bbeta_p=\0 \qon \Gamma_p^{D_d} \times(0,T], \quad \bsi_p \bn_p=0 \qon \Gamma_p^{N_s}\times(0,T].
      \label{eq: model problem 2c} 
\end{gather}
\end{subequations}
Here $\bu_p$ is the Darcy velocity, $p_p$ is the Darcy 
pressure, $\bbeta_p$ is the displacement, and $\bsi_p $ is the
poroelastic stress tensor, with
\begin{equation}\label{eq: def sigmap}
\bsi_p=\bsi_e - \alpha p_p\bI, \qquad
A\,\bsi_e= \bD(\bbeta_p),
\end{equation}
where $\bsi_e$ is the elastic stress tensor and $A:\bbS \to \bbM$ is the compliance tensor,
which is a uniformly symmetric and positive definite operator satisfying for some constants
$0<a_{\min} \leq a_{\max}$,
\begin{equation}\label{eq: def compliance tensor A 2}
  \forall \, \btau \in \bbS, \quad
  a_{\min}\, \btau : \btau \leq A\, \btau : \btau \leq a_{\max} \, \btau : \btau \ \ \forall \,
  \bx \in \Omega_p.
\end{equation}
In the isotropic case, $\bsi_e=\lambda_p(\nabla\cdot\bbeta_p)\bI+2\mu_p\bD(\bbeta_p)$,
where $0<\lambda_{\min} \leq \lambda_p(\textbf{x}) \leq \lambda_{\max}$ and
$0<\mu_{\min} \leq \mu_p(\textbf{x}) \leq \mu_{\max}$ are the Lam\'e parameters. In this case,
\begin{equation}\label{eq: def compliance tensor A 1}
  A(\btau) = \frac{1}{2\mu_p} \left(\btau - \frac{\lambda_p}{2\mu_p + n\, \lambda_p}
  \tr(\btau)\bI \right),
  \quad A^{-1}(\btau) = 2\,\mu_p\,\btau + \lambda_p\,\tr(\btau)\,\bI,
\end{equation}
with $a_{\min}=1/(2\mu_{\max} + n\, \lambda_{\max})$ and
$a_{\max}=1/(2\mu_{\min})$. We extend the definition of $A$ on $\bbM$ such that it is a positive constant multiple of the identity map on $\bbN$ as in \cite{lee2016}. In addition,
$\bK$ is the symmetric and uniformly positive definite rock permeability
tensor satisfying for some constants $0<k_{\min} \leq k_{\max}$,
\begin{equation}\label{eq: def permeability tensor K}
  \forall\, \bw\in\bbR^n, \quad k_{\min}\,\bw\cdot\bw \leq
  (\bK\bw)\cdot\bw \leq k_{\max}\,\bw\cdot\bw \quad \forall\, \bx\in\Omega_p.
\end{equation}
Finally, $s_0 > 0$ is the storativity coefficient, $0 < \alpha \leq 1$
is the Biot-Willis constant, $\f_p$ is a structure body force, and
$q_p$ is a source or sink term. For the boundary conditions we have
$\Gamma_p=\Gamma_p^{D_p} \cup \Gamma_p^{N_v}$ and $\Gamma_p
=\Gamma_p^{D_d}\cup \Gamma_p^{N_s}$. To avoid technical non-uniqueness
issues, we assume that $\vert\Gamma_p^{D_p}\vert$,
$\vert\Gamma_p^{D_d}\vert>0$. Furthermore,
to simplify the characterization of the normal trace spaces
on $\Gamma_{fp}$, we assume that
$\Gamma_p^{D_p}$ and $\Gamma_p^{D_d}$ are not adjacent to the
interface $\Gamma_{fp}$, i.e.  $\dist(\Gamma_p^{D_p},\Gamma_{fp})\geq
d_1 >0$ and $\dist(\Gamma_p^{D_d},\Gamma_{fp})\geq d_2 >0$.

The Stokes and Biot equations are coupled through interface conditions
on the fluid--poroelastic structure interface $\Gamma_{fp}$
\cite{bqq2009, Showalter2005}. They are mass conservation, balance of
normal components of the stresses, conservation of momentum and the
BJS condition \cite{bj1967,saff1971} modeling slip with friction:
\begin{subequations}\label{eq:model3}
\begin{gather}
\ds \bu_f \cdot \bn_f + (\frac{\partial}{\partial t}\bbeta_p+\bu_p)\cdot \bn_p=0, \quad
-(\bsi_f \bn_f)\cdot \bn_f=p_p \qon \Gamma_{fp}\times(0,T], \label{eq: model problem 3a} \\[1ex]
\ds \bsi_f \bn_f+\bsi_p \bn_p=0, \quad 
(-\bsi_f \bn_f)\cdot \bt_{f,j}=\mu \alpha_{\BJS}\sqrt{\bK_j^{-1}}
(\bu_f-\frac{\partial}{\partial t}\bbeta_p)\cdot \bt_{f,j} \qon \Gamma_{fp}\times(0,T],
  \label{eq: model problem 3b}
\end{gather}
\end{subequations}
where $\bn_f$ and $\bn_p$ are the outward unit normal vectors to
$\partial \Omega_f$ and $\partial \Omega_p$ respectively, 
$\bt_{f,j}$, $1 \leq j \leq n-1$ is an orthonormal system of 
tangent vectors on $\Gamma_{fp}$, $\bK_j=(\bK\bt_{f,j})\cdot
\bt_{f,j}$, and $\alpha_{\BJS} \geq 0$ is a friction coefficient.

Finally, the above system of equations is complemented by the initial
condition $p_p(\bx, 0)=p_{p,0}(\bx)$. Compatible initial data for the rest of the variables can be constructed from $p_{p,0}$ in a way that all
equations in the system \eqref{eq:model1}--\eqref{eq:model3}, except
for the unsteady conservation of mass equation in \eqref{eq: model
  problem 2a}, hold at $t = 0$. This will be established in
Lemma~\ref{lem: initial condition} below.  We will consider a weak
formulation with a time-differentiated elasticity equation and
compatible initial data $(\bsi_{p,0},p_{p,0})$.

%%%%%%%%%%%%%%%%%%%%%%%%%%%%%%%%%%%%%%%%%%%%%%%%%%%%%%%%%%%%%%
%%%%%%%%%%%%%%%%%%%%%%%%%%%%%%%%%%%%%%%%%%%%%%%%%%%%%%%%%%%%%%
%%%%%%%%%%%%%%%%%%%%%%%%%%%%%%%%%%%%%%%%%%%%%%%%%%%%%%%%%%%%%%
\section{Weak formulation}\label{sec:weak}
We define the fluid velocity space and fluid pressure space as the Hilbert spaces
\begin{equation*}
\bV_f:=\{\bv_f \in \bH^1(\Omega_f): \bv_f=\0 \qon \Gamma_f \},
\qquad
\W_f:=\L^2(\Omega_f),
\end{equation*}
respectively, endowed with the corresponding standard norms
\begin{equation*}
\Vert \bv_f \Vert_{\bV_f}:= \Vert \bv_f \Vert_{\bH^1(\Omega_f)},
\qquad
\Vert w_f \Vert_{\W_f}:= \Vert w_f \Vert_{\L^2(\Omega_f)}.
\end{equation*}
%The two variables involved in the equations are fluid velocity $\bu_f$ and fluid pressure $p_f$, with test functions to be $\bv_f$ and $w_f$ respectively. \\
%
For the structure region, we introduce a new variable, the structure
velocity $\ds \bu_s := \partial_t\bbeta_p$, using the notation
$\partial_t:= \frac{\partial}{\partial t}$. We will develop a formulation
that uses $\bu_s$ instead of $\bbeta_p$, which is better suitable for analysis. To impose the symmetry
condition on $\bsi_p$ weakly, we introduce the rotation operator
$\brho_p := \frac{1}{2}(\nabla \bbeta_p - \nabla \bbeta_p^\rt)$.
In the weak formulation we will use its time derivative 
$\bgamma_p := \partial_t \brho_p =\frac{1}{2}(\nabla \bu_s-\nabla \bu_s^\rt)$. We
introduce the Hilbert spaces
\begin{gather*}
\bV_p:=\{\bv_p \in \bH(\div;\Omega_p): \ \bv_p\cdot \bn_p=0\qon \Gamma_p^{N_v} \},
\qquad
\W_p:=\L^2(\Omega_p), \\[1ex]
\bbX_p:=\{\btau_p \in \bbH(\div;\Omega_p,\mathbb{M}): \ \btau_p \bn_p=0\qon \Gamma_p^{N_s} \}, \\[1ex]
\bV_s:= \bL^2(\Omega_p),
\qquad
\bbQ_p=\bbL^2(\Omega_p,\mathbb{N}),
\end{gather*}
endowed with the standard norms, respectively,
\begin{gather*}
\Vert \bv_p \Vert_{\bV_p}:= \Vert \bv_p \Vert_{\bH(\div;\Omega_p)},
\qquad
\Vert w_p \Vert_{\W_p}:= \Vert w_p \Vert_{\L^2(\Omega_p)}, \\[1ex]
\Vert \btau_p \Vert_{\bbX_p}:= \Vert \btau_p \Vert_{\bbH(\div;\Omega_p)},
\qquad
\Vert \bv_s \Vert_{\bV_s}:= \Vert \bv_s \Vert_{\bL^2(\Omega_p)},
\qquad
\Vert \bchi_p \Vert_{\bbQ_p}:=\Vert \bchi_p \Vert_{\bbL^2(\Omega_p)}.
\end{gather*}
We further introduce two Lagrange multipliers:
$$ \lambda:=-(\bsi_f\bn_f)\cdot \bn_f=p_p , \qan \btheta:= \bu_s \qon
\Gamma_{fp}.
$$
The first one is standard in Stokes--Darcy and Stokes--Biot models
with a mixed Darcy formulation and it is used to impose weakly
continuity of flux, cf. the first equation in \eqref{eq: model
  problem 3a}. The second one is needed in the mixed elasticity
formulation, since the trace of $\bu_s$ on $\Gamma_{fp}$ is not well
defined for $\bu_s \in \bL^2(\Omega_p)$. It will be used to impose
weakly the continuity of normal stress condition $\bsi_f
\bn_f\cdot\bn_f = \bsi_p \bn_p\cdot \bn_p$ and the BJS condition,
cf. \eqref{eq: model problem 3b}. For the Lagrange multiplier spaces
we need $\Lambda_p=(\bV_p \cdot \bn_p)'$ and
$\bLambda_s=(\bbX_p\bn_p)'$. According to the normal trace theorem,
since $\ds \bv_p \in \bV_p \subset \bH(\div; \Omega_p)$, then $\bv_p
\cdot \bn_p \in \H^{-1/2}(\partial \Omega_p)$. It is shown in
\cite{galvis2007} that if $\bv_p \cdot \bn_p = 0 $ on $\partial
\Omega_p \backslash \Gamma_{fp}$, then $\bv_p \cdot \bn_p \in
\H^{-1/2}(\Gamma_{fp})$.  In our case, since $\bv_p \cdot \bn_p = 0$
on $\Gamma_p^{N_v}$ and $\dist(\Gamma_p^{D_p},\Gamma_{fp})\geq d_1 >0
$, the argument can be modified as follows. For any $\xi \in
\H^{1/2}(\Gamma_{fp})$, let $E_1 \xi$ be a continuous extension to
$\H^{1/2}(\Gamma_{fp} \cup \Gamma_p^{N_v})$ such that $E_1 \xi =0$ on
$\partial(\Gamma_{fp} \cup \Gamma_p^{N_v})$, then let $E_2 (E_1 \xi) \in \H^{1/2}(\partial\Omega)$
be a continuous extension of $E_1 \xi$ such that
$E_2 (E_1 \xi)=0$ on $\Gamma_p^{D_p}$. We then have
\begin{equation*}
\langle \bv_p \cdot \bn_p, \xi \rangle_{\Gamma_{fp}}
= \langle \bv_p \cdot \bn_p, E_1\xi \rangle_{\Gamma_{fp} \cup \Gamma_p^{N_v}}
=\langle \bv_p \cdot \bn_p, E_2(E_1\xi) \rangle_{\partial \Omega_p}
\end{equation*}
and 
\begin{equation}\label{trace-vel}
\langle \bv_p \cdot \bn_p, \xi \rangle_{\Gamma_{fp}}
\leq 
\Vert \bv_p \cdot \bn_p \Vert_{\H^{-1/2}(\partial \Omega_p)} \Vert E_2(E_1\xi) \Vert_{\H^{1/2}(\partial \Omega_p)}
\leq 
C \Vert \bv_p \Vert_{\bH(\div; \Omega_p)} \Vert \xi \Vert_{\H^{1/2}(\Gamma_{fp})}.
\end{equation}
Similarly, for any $\bphi \in \bH^{1/2}(\Gamma_{fp})$,
\begin{equation}\label{trace-stress}
\langle \bsi_p \bn_p, \bphi \rangle_{\Gamma_{fp}} \leq 
C \Vert \bsi_p \Vert_{\bbH(\div; \Omega_p)} \Vert \bphi \Vert_{\bH^{1/2}(\Gamma_{fp})}.
\end{equation}
Thus we can take
\begin{equation*}
\Lambda_p := \H^{1/2}(\Gamma_{fp}), 
\quad
\bLambda_s := \bH^{1/2}(\Gamma_{fp})
\end{equation*}
with norms
\begin{equation}\label{H-1/2 norm}
\Vert \xi \Vert_{\Lambda_p} := \Vert \xi \Vert_{\H^{1/2}(\Gamma_{fp})},
\quad
\Vert \bphi \Vert_{\bLambda_s} := \Vert \bphi \Vert_{\bH^{1/2}(\Gamma_{fp})}.
\end{equation}

We now proceed with the derivation of the variational
formulation of \eqref{eq:model1}--\eqref{eq:model3}.
We test the first equation in \eqref{eq: model problem 1a} with an arbitrary
$\bv_f \in \bV_f$, integrate by parts, and combine with the BJS
interface condition in \eqref{eq: model problem 3b}. We test the third
equation in \eqref{eq: model problem 2a} by $w_p \in \W_p$ and
make use of \eqref{eq: def sigmap} and the fact that
$$
\nabla \cdot
\bbeta_p = \tr (\bD(\bbeta_p)) = \tr(A\bsi_e)=\tr A(\bsi_p+\alpha
\,p_p\, \bI),
$$
as well as $\tr(\btau)w=\btau:(w\bI) \ \forall \, \btau \in \bbM, w\in \bbR$.
In addition, \eqref{eq: def sigmap} gives
$$
A(\bsi_p + \alpha p_p \bI) = \nabla \bbeta_p - \brho_p.
$$
In the weak formulation we will use its time differentiated version
$$
\partial_t A(\bsi_p + \alpha p_p \bI)=\nabla \bu_s - \bgamma_p,
$$
which is tested by $\btau_p \in \bbX_p$. Finally, we impose the
remaining equations weakly, as well as the symmetry of $\bsi_p$ and
the interface conditions \eqref{eq:model3}, obtaining the following mixed variational formulation:
Given
\begin{equation*}
\f_f: [0,T] \rightarrow \bV_f',\quad 
\f_p: [0,T] \rightarrow \bV_s',\quad 
q_f: [0,T] \rightarrow \W_f',\quad 
q_p: [0,T] \rightarrow \W_p'
\end{equation*}
and $(\bsi_{p,0},p_{p,0}) \in \bbX_p \times \W_p$, find $(\bu_f, p_f,
\bsi_p, \bu_s, \bgamma_p,
\bu_p, p_p,
\lambda, \btheta): [0,T] \rightarrow \bV_f \times \W_f
\times \bbX_p \times \bV_s \times \bbQ_p
\times \bV_p \times \W_p
\times \Lambda_p \times \bLambda_s $ such that $(\bsi_p(0),p_p(0)) = (\bsi_{p,0},p_{p,0})$ and,
for a.e. $t\in (0,T)$ and for all $\bv_f \in \bV_f$, $w_f \in \W_f$,
$\btau_p \in \bbX_p$, $\bv_s \in \bV_s$, $\bchi_p \in \bbQ_p$,
$\bv_p \in \bV_p$, $w_p \in \W_p$,
$\xi \in \Lambda_p$, and $\bphi \in
\bLambda_s$,
\begin{subequations}\label{weak-form}
\begin{align}
&\ds (2\mu \bD(\bu_f),\bD(\bv_f))_{\Omega_f}
-(\nabla\cdot \bv_f, p_f)_{\Omega_f} 
+\langle \bv_f\cdot \bn_f,  \lambda \rangle_{\Gamma_{fp}}  \nonumber \\[1ex]
& \qquad +\sum_{j=1}^{n-1}{\langle \mu \alpha_{\BJS}\sqrt{\bK_j^{-1}}(\bu_f-\btheta)\cdot \bt_{f,j},
  \bv_f \cdot \bt_{f,j}\rangle_{\Gamma_{fp}}}
=(\f_f,\bv_f)_{\Omega_f}, \label{weak-form-1}  \\[1ex]
&\ds (\nabla\cdot \bu_f, w_f)_{\Omega_f}
=(q_f, w_f)_{\Omega_f}, \label{weak-form-2} \\[1ex]
&\ds (\partial_t A(\bsi_p+\alpha  p_p\bI),\btau_p)_{\Omega_p}
+(\nabla \cdot\btau_p,\bu_s)_{\Omega_p}
+(\btau_p,\bgamma_p)_{\Omega_p}
-\langle \btau_p \bn_p, \btheta \rangle_{\Gamma_{fp}}
=0, \label{weak-form-3} \\[1ex]
&\ds (\nabla\cdot \bsi_p, \bv_s)_{\Omega_p} = -(\f_p, \bv_s)_{\Omega_p}, \label{weak-form-4} \\[1ex]
&\ds (\bsi_p,\bchi_p)_{\Omega_p}  =0, \label{weak-form-5} \\[1ex]
&\ds (\mu \bK^{-1}\bu_p,\bv_p)_{\Omega_p}
-(\nabla \cdot \bv_p,p_p)_{\Omega_p}
+\langle \bv_p\cdot \bn_p, \lambda \rangle_{\Gamma_{fp}}
=0, \label{weak-form-6} \\[1ex]
&\ds (s_0 \partial_t p_p, w_p)_{\Omega_p}
+\alpha( \partial_t A(\bsi_p+\alpha p_p \bI), w_p \bI)_{\Omega_p}
+(\nabla \cdot \bu_p, w_p)_{\Omega_p}
=(q_p, w_p)_{\Omega_p}, \label{weak-form-7} \\[1ex]
&\ds \langle \bu_f \cdot \bn_f  + \btheta\cdot \bn_p  + \bu_p\cdot \bn_p ,\xi \rangle_{\Gamma_{fp}} =0,
\label{weak-form-8} \\[1ex]
&\ds \langle \bphi\cdot \bn_p, \lambda \rangle_{\Gamma_{fp}}
-\sum_{j=1}^{n-1}{\langle \mu \alpha_{\BJS}\sqrt{\bK_j^{-1}}(\bu_f-\btheta)
  \cdot \bt_{f,j}, \bphi \cdot \bt_{f,j}\rangle_{\Gamma_{fp}}}
+\langle \bsi_p \bn_p, \bphi \rangle_{\Gamma_{fp}} =0. \label{weak-form-9}
\end{align}
\end{subequations}
In the above, \eqref{weak-form-1}--\eqref{weak-form-2} are the Stokes equations,
\eqref{weak-form-3}--\eqref{weak-form-5} are the elasticity equations,
\eqref{weak-form-6}--\eqref{weak-form-7} are the Darcy equations,
and \eqref{weak-form-8}--\eqref{weak-form-9} enforce weakly the interface conditions.

\begin{remark}\label{rem:time-diff}
The time differentiated equation \eqref{weak-form-3} allows us to
eliminate the displacement variable $\bbeta_p$ and obtain a
formulation that uses only $\bu_s$. As part of the analysis we will
construct suitable initial data such that, by integrating
\eqref{weak-form-3} in time, we can recover the original equation
\begin{equation}\label{non-diff-eq}
(A(\bsi_p+\alpha  p_p\bI),\btau_p)_{\Omega_p}
+(\nabla \cdot\btau_p,\bbeta_p)_{\Omega_p}
+(\btau_p,\brho_p)_{\Omega_p}
-\langle \btau_p \bn_p, \bpsi \rangle_{\Gamma_{fp}} = 0,
\end{equation}
where $\bpsi:= \bbeta_p|_{\Gamma_{fp}}$.
\end{remark}

In order to obtain a structure suitable for analysis,
we combine the equations for the variables with coercive bilinear forms,
$\bu_f$, $\bu_p$, $\bsi_p$, and $p_p$, together with $\btheta$, which is coupled with them
via the continuity of flux and BJS conditions. We further combine the rest of the equations.
Introducing the bilinear forms
\begin{gather*}
a_f(\bu_f, \bv_f):=(2\mu \bD(\bu_f), \bD(\bv_f))_{\Omega_f},
\quad 
a_p(\bu_p, \bv_p):=(\mu \bK^{-1}\bu_p, \bv_p)_{\Omega_p},
\quad 
a_p^p(p_p, w_p):=(s_0 p_p, w_p)_{\Omega_p}, \\[1ex]
b_\star(\bv_\star, w_\star):=-(\nabla \cdot \bv_\star, w_\star)_{\Omega_\star}, \ \star \in \{f,p\},
\quad
b_s(\btau_p, \bv_s):=(\nabla \cdot \btau_p, \bv_s)_{\Omega_p}, \\[1ex]
b_n^p(\btau_p, \bphi):= \langle \btau_p \bn_p, \bphi \rangle_{\Gamma_{fp}},
\quad
b_{\sk}(\btau_p, \bchi_p):=(\btau_p, \bchi_p)_{\Omega_p}, \\[1ex]
a_e(\bsi_p, p_p; \btau_p, w_p) := (A(\bsi_p+\alpha p_p \bI), \btau_p + \alpha w_p \bI)_{\Omega_p}, \\[1ex]
a_{\BJS}(\bu_f, \btheta;\bv_f, \bphi) :=
\sum_{j=1}^{n-1}{\langle\mu \alpha_{\BJS}\sqrt{\bK_j^{-1}}(\bu_f-\btheta)
  \cdot \bt_{f,j}, (\bv_f-\bphi) \cdot \bt_{f,j}\rangle_{\Gamma_{fp}}},\\[1ex]
b_{\Gamma}(\bv_f, \bv_p, \bphi ;\xi) := 
\langle \bv_f \cdot \bn_f + \bphi \cdot \bn_p + \bv_p \cdot \bn_p ,\xi \rangle_{\Gamma_{fp}},
\end{gather*}
the system \eqref{weak-form} can be written as follows:
\begin{equation}\label{eq: continuous formulation 2}
\begin{array}{l}  
 a_f(\bu_f,\bv_f)
+a_p(\bu_p,\bv_p)
+a_{\BJS}(\bu_f,\btheta;\bv_f,\bphi)
+b_n^p(\bsi_p, \bphi)
+b_p(\bv_p,p_p)+b_f(\bv_f,p_f) \\[1ex]
\quad +b_s(\btau_p,\bu_s)
+b_{\sk}(\btau_p, \bgamma_p)
+b_{\Gamma}(\bv_f,\bv_p,\bphi;\lambda)
+a_p^p(\partial_tp_p, w_p)
+ a_e(\partial_t \bsi_p, \partial_t p_p; \btau_p, w_p)
\\[1ex]
\quad -b_n^p(\btau_p, \btheta)
-b_p(\bu_p, w_p)= (\f_f,\bv_f)
+(q_p,w_p)_{\Omega_p}, \\[1ex]
-b_f(\bu_f, w_f)
-b_s(\bsi_p,\bv_s)
-b_{\sk}(\bsi_p, \bchi_p)
-b_{\Gamma}(\bu_f, \bu_p, \btheta ;\xi) 
=(q_f,w_f)_{\Omega_f}
+(\f_p,\bv_s).
\end{array}
\end{equation}
We group the spaces and test functions as:
\begin{gather*}
\ds \bQ := \bV_f \times \bLambda_s \times \bV_p \times \bbX_p \times \W_p,\quad
\ds \bS := \W_f\times \bV_s\times \bbQ_p\times \Lambda_p, \\[1ex]
\ds \bp := (\bu_f, \btheta, \bu_p, \bsi_p, p_p)\in \bQ,\quad 
\br := (p_f, \bu_s, \bgamma_p, \lambda)\in \bS,\\[1ex] 
\ds \bq := (\bv_f, \bphi, \bv_p, \btau_p, w_p)\in \bQ,\quad 
\bs := (w_f, \bv_s, \bchi_p, \xi)\in \bS,
\end{gather*}
where the spaces $\bQ$ and $\bS$ are endowed with the norms, respectively,
\begin{gather*}
\ds \|\bq\|_{\bQ}  =  \|\bv_f\|_{\bV_f}
+\|\bphi\|_{\bLambda_s}+ \|\bv_p\|_{\bV_p}+\|\btau_p\|_{\bbX_p}+\|w_p\|_{\W_p},\\[1ex]
\ds \|\bs\|_{\bS}  = \|w_f\|_{\W_f}+\|\bv_s\|_{\bV_s}+\|\bchi_p\|_{\bbQ_p}+\|\xi\|_{\Lambda_p}. 
\end{gather*}
Hence, we can write \eqref{eq: continuous formulation 2}
in an operator notation as a degenerate evolution problem in a mixed form:
\begin{equation}\label{eq: continuous formulation 3}
    \arraycolsep=1.7pt
  \begin{array}{rcl}
\ds \partial_t\,\cE_1\,\bp(t)
+ \cA\,\bp(t) + \cB'\,\br(t) & = & \bF(t) \qin \bQ', \\[1ex]
\ds  - \cB\,\bp(t)  & = & \bG(t) \qin \bS'.
\end{array}
\end{equation}
The operators $\cA\,:\,\bQ\to \bQ'$, $\cB\,:\, \bQ\to \bS'$
and the functionals $\bF(t) \in \bQ'$, $\bG(t) \in \bS'$ are defined as follows:
\begin{equation}\label{defn-A-B}
\cA = \left(\begin{array}{ccccc}
\ds A_f+A_{\BJS}^{f}  & (A_{\BJS}^{fs})' & 0 & 0 & 0\\ 
\ds A_{\BJS}^{fs} & A_{\BJS}^{s} & 0 & (B_n^p)' & 0 \\
\ds 0 & 0 & A_p & 0 & B_p'\\ 
\ds 0 & -B_n^p & 0 & 0 & 0\\ 
0 & 0 & -B_p & 0 & 0
\end{array}\right),\quad
\cB = \left(\begin{array}{ccccc}
\ds B_f & 0 & 0 & 0 & 0 \\ 
\ds 0 & 0 & 0 & B_s & 0 \\ 
\ds 0 & 0 & 0 & B_{\sk} & 0  \\  
\ds B_{\Gamma}^f & B_{\Gamma}^s & B_{\Gamma}^p & 0 & 0
\end{array}\right),
\end{equation}
\begin{equation*}
\bF(t) = \left(\begin{array}{c}
\f_f \\ 
0 \\ 
0 \\ 
\0 \\
q_p
\end{array}\right),\quad 
\bG(t) = \left(\begin{array}{c}
q_f \\ 
\f_p \\ 
0 \\ 
0
\end{array}\right),
\end{equation*}
where 
\begin{gather*}
\ds (A_f \bu_f,\bv_f) = a_f(\bu_f, \bv_f), \qquad
(A_p \bu_p,\bv_p) = a_p(\bu_p, \bv_p), \\[1ex]
\ds (B_p \bu_p,w_p) = b_p(\bu_p, w_p), \qquad
(B_n^p \bsi_p, \bphi) = b_n^p(\bsi_p, \bphi), \\[1ex]
\ds (A_{\BJS}^f \bu_f,\bv_f) = a_{\BJS}(\bu_f, \0; \bv_f, \0), \quad
(A_{\BJS}^{fs} \bu_f,\bphi) = a_{\BJS}(\bu_f, \0; \0, \bphi), \quad
(A_{\BJS}^s \btheta, \bphi) = a_{\BJS}(\0, \btheta; \0, \bphi), \\[1ex]
\ds (B_f \bu_f, w_f) = b_f(\bu_f, w_f), \quad
(B_s \bsi_p, \bv_s) = b_s(\bsi_p, \bv_s), \quad
(B_{\sk} \bsi_p, \bchi_p) = b_{\sk}(\bsi_p, \bchi_s),\\[1ex]
\ds (B_\Gamma^f \bu_f, \xi) = b_\Gamma(\bu_f, \0, \0; \xi), \quad
(B_\Gamma^s \btheta, \xi) = b_\Gamma(\0, \0, \btheta; \xi), \quad
(B_\Gamma^p \bu_p, \xi) = b_\Gamma(\0, \bu_p, \0; \xi).
\end{gather*}
The operator $\cE_1 : \bQ \to \bQ'$ is given by:
\begin{equation*}
\cE_1 = \left(\begin{array}{ccccc}
0 & 0 & 0 & 0 & 0 \\
0 & 0 & 0 & 0 & 0 \\
0 & 0 & 0 & 0 & 0 \\
0 & 0 & 0 & A_e^s & A_e^{sp} \\
0 & 0 & 0 & (A_e^{sp})' &  A_p^p+A_e^p
\end{array}\right),
\end{equation*}
where
\begin{gather*}
\ds (A_e^s \bsi_p, \btau_p) = a_e(\bsi_p, 0;  \btau_p,0), \qquad
(A_e^{sp} \bsi_p, w_p) = a_e(\bsi_p, 0;  \0,w_p), \\[1ex]
\ds (A_e^p p_p, w_p) = a_e(\0, p_p;  \0, w_p), \qquad
(A_p^p p_p, w_p) = a_p^p(p_p, w_p).
\end{gather*}

%%%%%%%%%%%%%%%%%%%%%%%%%%%%%%%%%%%%%%%%%%%%%%%%%%%%%%%%%%%%%%
%%%%%%%%%%%%%%%%%%%%%%%%%%%%%%%%%%%%%%%%%%%%%%%%%%%%%%%%%%%%%%
%%%%%%%%%%%%%%%%%%%%%%%%%%%%%%%%%%%%%%%%%%%%%%%%%%%%%%%%%%%%%%
\section{Well-posedness of the weak formulation}\label{sec:well-posed}

\subsection{Preliminaries}
We start with exploring important properties of the operators introduced in the previous section.
\begin{lemma}\label{lem: continuous monotonicity}
The linear operators $\cA$ and $\cE_1$ are continuous and monotone.
\end{lemma}
\begin{proof}
Continuity follows from the Cauchy-Schwarz inequality and
the trace inequalities \eqref{trace-vel}--\eqref{trace-stress}. In particular,
\begin{equation}\label{eq: continuous continuity 0}
  \begin{array}{c}
\ds a_f(\bu_f,\bv_f)
\leq 2\mu \| \bu_f \|_{\bV_f} \| \bv_f \|_{\bV_f},
\qquad
a_p(\bu_p,\bv_p)
\leq \mu k_{\min}^{-1}\| \bu_p \|_{\bL^2(\Omega_p)} \| \bv_p \|_{\bL^2(\Omega_p)} , \\[2ex]
\ds a_{\BJS}(\bu_f,\btheta;\bv_f,\bphi)
\leq \mu \alpha_{\BJS} k_{\min}^{-1/2}
\vert \bu_f - \btheta \vert_{a_{\BJS}}\vert \bv_f - \bphi \vert_{a_{\BJS}} \\[1ex]
\ds \le C(\|\bu_f\|_{\bV_f} + \|\btheta\|_{\bL^2(\Gamma_{fp})})
(\|\bv_f\|_{\bV_f} + \|\bphi\|_{\bL^2(\Gamma_{fp})}), \\[1ex]
\ds b_n^p(\btau_p, \bphi)
\leq C\|\btau_p \|_{\bbX_p} \|\bphi \|_{\bLambda_s},
\qquad b_p(\bv_p,w_p) 
\leq  \|\bv_p\|_{\bV_p}\|w_p\|_{\W_p},
  \end{array}
  \end{equation}
where, for $\bv_f \in \bV_f$, $\bphi \in \bLambda_f$,
$ \vert \bv_f - \bphi\vert^2_{a_\BJS}
:= \sum^{n-1}_{j=1} \pil(\bv_f - \bphi)
\cdot\bt_{f,j},(\bv_f - \bphi)\cdot\bt_{f,j}\pir_{\Gamma_{fp}}$, and we have used the trace
inequality, for a domain $\cO$ and $S \subset \partial \cO$,
\begin{equation}\label{trace}
  \|\varphi\|_{\H^{1/2}(S)} \le C \|\varphi\|_{\H^1(\cO)} \quad \forall \, \varphi \in \H^1(\cO).
\end{equation}
Thus we have
\begin{align}
(\cA \bp, \bq) & = a_f(\bu_f,\bv_f) 
+a_p(\bu_p,\bv_p)
+a_{\BJS}(\bu_f,\btheta;\bv_f,\bphi)
+b_n^p(\bsi_p, \bphi)
-b_n^p(\btau_p, \theta) 
+b_p(\bv_p,p_p)-b_p(\bu_p,w_p) \nonumber \\[1ex]
& \leq C\| \bp \|_{\bQ} \| \bq \|_{\bQ} \label{eq: continuous continuity 1}
\end{align}
and
\begin{equation}\label{eq: continuous continuity 2}
(\cE_1 \bp, \bq)
= (s_0 p_p,w_p)_{\Omega_p} 
+ (A(\bsi_p + \alpha p_p \bI), \btau_p + \alpha w_p \bI)_{\Omega_p} \leq
C\|\bp\|_{\bQ}\| \bq \|_{\bQ}.
\end{equation}
Therefore $\cA$ and $\cE_1$ are continuous. The monotonicity of $\cA$
follows from 
\begin{equation}\label{eq: continuous monotonicity 1}
\begin{array}{c}
a_f(\bv_f,\bv_f)
= 2\mu \| \bD(\bv_f) \|^2_{\bbL^2(\Omega_f)}
\geq 2\mu C_K^2 \Vert \bv_{f} \Vert ^2_{\bH^1(\Omega_f)}, \\[2ex]
a_p(\bv_p,\bv_p) = \mu \Vert \bK^{-1/2} \bv_{p} \Vert^2_{\bL^2(\Omega_p)}
\geq \mu k_{\max}^{-1} \| \bv_p \|^2_{\bL^2(\Omega_p)} ,  \\[2ex]
\ds a_{\BJS}(\bv_f,\bphi;\bv_f,\bphi)
\geq \mu \alpha_{\BJS} k_{\max}^{-1/2}\vert \bv_f - \bphi\vert^2_{a_\BJS},
\end{array}
\end{equation}
where we used Korn's inequality $\| \bD(\bv_f) \| \ge C_K \Vert \bv_{f} \Vert_{\bH^1(\Omega_f)}$
in the first bound.
The monotonicity of $\cE_1$ follows from 
\begin{equation}\label{eq: continuous monotonicity 3}
\ds (\cE_1 \bq, \bq)\,= \,
s_0 \Vert w_p \Vert^2_{\L^2(\Omega_p)} + \Vert A^{1/2}\,(\btau_p + \alpha \, w_p \, \bI) \Vert^2_{\bbL^2(\Omega_p)}.
\end{equation}
\end{proof}

\begin{lemma}\label{lem: continuous inf-sup}
  The linear operator $\cB$ is continuous. Furthermore,
  there exist positive constants $\beta_1$, $\beta_2$, and $\beta_3$ such that
\begin{align}
& \ds \beta_1 (\|\bv_s\|_{\bV_s}+\|\bchi_p\|_{\bbQ_p})
\leq 
\underset{\btau_p \in \bbX_p \, \text{s.t.} \btau_p \bn_p=\0 \, \text{on} \, \Gamma_{fp}}{\sup}
\frac{b_s(\btau_p,\bv_s)+b_{\sk}(\btau_p,\bchi_p)}{\|\btau_p\|_{\bbX_p}},
\qquad \forall\, \bv_s \in \bV_s, \,\bchi_p \in \bbQ_p, 
\label{eq: continuous inf-sup 2} \\[1ex]
& \ds \beta_2(\|w_f\|_{\W_f} + \|w_p\|_{\W_p} + \|\xi\|_{\Lambda_p}) \leq 
\underset{(\bv_f, \bv_p) \in \bV_f\times\bV_p}{\sup}
\frac{b_f(\bv_f,w_f)+b_p(\bv_p,w_p)+
b_{\Gamma}(\bv_f, \bv_p, \0; \xi)}{\|(\bv_f, \bv_p)\|_{\bV_f\times\bV_p}}, \nonumber \\[1ex]
& \qquad\qquad
\forall\, w_f \in \W_f, \,w_p \in \W_p, \ \text{and} \ \xi \in \Lambda_p,
\label{eq: continuous inf-sup 3} \\[1ex]
&  \beta_3\|\bphi\|_{\bLambda_s} \leq   \underset{\btau_p \in \bbX_p \, \text{s.t.}
    \nabla \cdot \btau_p=0}{\sup} \frac{b_n^p(\btau_p, \bphi) }{\|\btau_p\|_{\bbX_p}},
  \qquad \forall \, \bphi \in \bLambda_s. \label{eq: continuous inf-sup 4}
\end{align}
\end{lemma}
\begin{proof}
The definition \eqref{defn-A-B} of $\cB$ implies
\begin{align}\label{eq: continuous continuity 3}
&(\cB \bq, \bs)
= b_f(\bv_f, w_f)
+b_s(\btau_p,\bv_s)
+b_{\sk}(\btau_p,\bchi_p)
+b_{\Gamma}(\bv_f,\bv_p,\bphi;\xi) \nonumber \\[1ex]
& \quad \leq
\|\nabla \cdot \bv_f\|_{\L^2(\Omega_f)}\|w_f\|_{\L^2(\Omega_f)}
+\|\nabla \cdot \btau_p\|_{\bL^2(\Omega_p)}\|\bv_s\|_{\bL^2(\Omega_p)}
+\|\btau_p\|_{\bbL^2(\Omega_p)}\|\bchi_p\|_{\bbL^2(\Omega_p)}\nonumber \\[1ex]
&\qquad +C\|\bv_f \|_{\bH^1(\Omega_f)}\|\xi\|_{\L^2(\Gamma_{fp})}
+C\|\bv_p\|_{\bH(\div;\Omega_p)}\|\xi\|_{\H^{1/2}(\Gamma_{fp})}
+ \|\bphi\|_{\bL^2(\Gamma_{fp})}\|\xi\|_{\L^2(\Gamma_{fp})} \nonumber \\[1ex]
&\quad \leq
C\|\bq\|_{\bQ}\| \bs \|_{\bS},
\end{align}
so $\cB$ is continuous. Next, inf-sup condition \eqref{eq:
  continuous inf-sup 2} follows from \cite[Section~2.4.3]{gatica2014}.  We
note that the restriction $\btau_p\bn_p=\0$ on $\Gamma_{fp}$ allows
us to eliminate the term $b_n^p(\btau_p, \btheta)$ when applying
this inf-sup condition, see \eqref{eq: inf-sup app epsilon} below.
Inf-sup condition \eqref{eq: continuous inf-sup 3} follows from
a modification of the argument in Lemmas 3.1 and 3.2 in
\cite{ervin2009} to account for $\vert \Gamma_p^D \vert >0$.
Finally, \eqref{eq: continuous inf-sup 4}
can be proved using the argument in \cite[Lemma 4.2]{gatica2014}.
\end{proof}

\subsection{Existence and uniqueness of a solution}

We will establish existence of a solution to the weak formulation \eqref{eq: continuous formulation 3}
using the following key result. 

\begin{theorem}\label{thm: Showalter}
\textnormal{\cite[Theorem~IV.6.1(b)]{showalter}}
Let the linear, symmetric and monotone operator $\cN$ be given for the real vector space $E$ to its algebraic dual $E^*$, and let $E_b'$ be the Hilbert space which is the dual of $E$ with the seminorm
$$ \vert x \vert_b = (\cN x(x))^{1/2} \qquad x\in E.$$
Let $\cM \subset E\times E_b'$ be a relation with domain $\cD=\{ x\in E: \cM(x) \neq \emptyset \}$.
Assume that $\cM$ is monotone and $Rg(\cN + \cM) = E'_b$.
Then, for each $u_0\in \cD$ and for each $f\in \W^{1,1}(0,T;E'_b)$, there is a solution $u$ of
\begin{equation}\label{show}
\frac{d}{dt}\big(\cN\,u(t)\big) + \cM\big(u(t)\big) \ni f(t) \quad a.e. \ 0 < t < T,
\end{equation}
with
$$
\cN\,u\in \W^{1,\infty}(0,T;E'_b),\quad u(t)\in \cD,\quad \mbox{ for a.e. }\, 0\leq t \leq T,
\qan \cN\,u(0) = \cN\,u_0.
$$
\end{theorem}

We cast \eqref{eq: continuous formulation 3} in the form \eqref{show} by setting
\begin{equation}\label{eq: definition N and M}
E = \bQ \times \bS, \quad u = \left(\begin{array}{c} \bp \\ \br \end{array} \right), \quad
  \cN = \left( \begin{array}{cc}
\ds \cE_1 &  \0\\
\ds \0    &  \0
\end{array} \right),
\quad
\cM = \left( \begin{array}{cc}
\ds \cA &  \cB'\\
\ds -\cB   &  \0
\end{array} \right), \quad
f = \left(\begin{array}{c} \bF \\ \bG \end{array} \right).
\end{equation}
The seminorm induced by the operator $\cE_1$ is $|\bq|_{\cE_1}^2 :=
s_0 \Vert w_p \Vert^2_{\L^2(\Omega_p)} + \Vert A^{1/2}\,(\btau_p +
\alpha \, w_p \, \bI) \Vert^2_{\bbL^2(\Omega_p)}$, cf. \eqref{eq:
  continuous monotonicity 3}. Since $s_0 > 0$, it is equivalent to
$\|\btau_p\|_{\bbL^2(\Omega_p)}^2 + \|w_p\|_{\L^2(\Omega_p)}^2$.
We denote by $\bbX_{p,2}$ and $\W_{p,2}$
the closures of the spaces $\bbX_p$ and $\W_p$, respectively, with
respect to the norms $\Vert \btau_p \Vert_{\bbX_{p,2}} :=
\|\btau_p\|_{\bbL^2(\Omega_p)}$ and $\Vert w_p \Vert_{\W_{p,2}} :=
\|w_p\|_{\L^2(\Omega_p)}$. Then the Hilbert space $E_b'$ in
Theorem~\ref{thm: Showalter} in our case is
\begin{equation}\label{defn-Eb}
E_b':= \bQ_{2,0}' \times \bS_{2,0}', \ \
\text{where} \ \ \bQ_{2,0}' := \0 \times \0 \times \0 \times \bbX_{p,2}' \times \W_{p,2}', \quad
\bS_{2,0}' := \0 \times \0 \times \0 \times \0.
\end{equation}
We further define $\cD := \{ (\bp,\br) \in \bQ \times \bS: \ \cM (\bp,\br) \in E_b' \}$.

\begin{remark}
The above definition of the space $E_b'$ and the corresponding domain
$\cD$ implies that, in order to apply Theorem~\ref{thm: Showalter} for
our problem \eqref{eq: continuous formulation 3}, we need to restrict
$\f_f = \0$, $q_f = 0$, and $\f_p = \0$. To avoid this restriction we
will employ a translation argument \cite{Showalter-SIAMMA} to reduce the existence for
\eqref{eq: continuous formulation 3} to existence for the following
initial-value problem: Given initial data $(\wh\bp_0,\wh\br_0) \in
\cD$ and source terms $(\wh{g}_{\btau_p},\wh{g}_{w_p}): (0,T) \to \bbX_{p,2}'
\times \W_{p,2}'$, find $(\bp, \br): [0,T] \to \bQ \times \bS$ such
that $(\bsi_p(0), p_p(0))=(\wh\bsi_{p,0}, \wh p_{p,0})$ and, for
a.e. $t \in (0,T)$,
\begin{equation}\label{eq: reduced formulation 1}
    \arraycolsep=1.7pt
  \begin{array}{rcl}  
\ds \partial_t\,\cE_1\,\bp(t)
+ \cA\,\bp(t) + \cB'\,\br(t) & = & \wh{\bF}(t) \qin \bQ_{2,0}', \\[1ex]
\ds - \cB\,\bp(t) & = & \0 \quad \qin \bS_{2,0}',
\end{array}
\end{equation}
where $\wh{\bF}(t)=(\0, 0, 0, \wh{g}_{\btau_p},\wh{g}_{w_p})^{\rt}$.
\end{remark}

In order to apply Theorem~\ref{thm: Showalter} for problem \eqref{eq: reduced formulation 1}, we need
to 1) establish the required properties of the operators $\cN$ and $\cM$, 2) prove the range condition
$Rg(\cN + \cM) = E'_b$, and 3) construct compatible initial data $(\wh\bp_0,\wh\br_0) \in \cD$. We
proceed with a sequence of lemmas establishing these results.

\begin{lemma}\label{lem:oper-prop}
  The linear operator $\cN$ defined in \eqref{eq: definition N and M} is continuous,
  symmetric, and monotone. The linear operator $\cM$ defined in \eqref{eq: definition N and M}
  is continuous and monotone.
\end{lemma}
\begin{proof}
The stated properties follow easily from the properties of the operators $\cE_1$, $\cA$, and $\cB$
established in Lemma~\ref{lem: continuous monotonicity} and Lemma~\ref{lem: continuous inf-sup}.
\end{proof}  

Next, we establish the range condition $Rg(\cN + \cM) = E'_b$, which is done by solving the related
resolvent system. In fact, we will show a stronger result by considering a resolvent system
where all source terms may be non-zero. This stronger result will be used in the translation argument
for proving existence of the original problem \eqref{eq: continuous formulation 3}. In particular,
consider the following resolvent
system: Given
$\wh{g}_{\bv_f} \in \bV_f'$,
$\wh{g}_{w_f} \in \W_f'$,
$\wh{g}_{\btau_p} \in \bbX_{p,2}'$,
$\wh{g}_{\bv_s} \in \bV_s'$,
$\wh{g}_{\bchi_p} \in \bbQ_p'$, 
$\wh{g}_{\bv_p} \in \bV_p'$,
$\wh{g}_{w_p} \in \W_{p,2}'$,
$\wh{g}_\xi \in \Lambda_p'$, and
$\wh{g}_{\bphi} \in\bLambda_s'$,
find $(\bu_f, p_f,
\bsi_p, \bu_s, \bgamma_p,
\bu_p, p_p,
\lambda, \btheta) \in \bV_f \times \W_f
\times \bbX_p \times \bV_s \times \bbQ_p
\times \bV_p \times \W_p
\times \Lambda_p \times \bLambda_s$ such that
for all $\bv_f \in \bV_f$, $w_f \in \W_f$,
$\btau_p \in \bbX_p$, $\bv_s \in \bV_s$, $\bchi_p \in \bbQ_p$,
$\bv_p \in \bV_p$, $w_p \in \W_p$,
$\xi \in \Lambda_p$, and $\bphi \in
\bLambda_s$,
\begin{equation}\label{eq: resolvent formulation 1}
\begin{array}{l}  
a_f(\bu_f,\bv_f)
+a_p(\bu_p,\bv_p)
+a_{\BJS}(\bu_f,\btheta;\bv_f,\bphi)
+b_n^p(\bsi_p, \bphi)
+b_p(\bv_p,p_p)+b_f(\bv_f,p_f)
+b_s(\btau_p,\bu_s) \\[1ex]
\quad
+b_{\sk}(\btau_p, \bgamma_p)
+b_{\Gamma}(\bv_f,\bv_p,\bphi;\lambda)
+a_p^p(p_p, w_p)
+a_e(\bsi_p, p_p; \btau_p, w_p) 
-b_n^p(\btau_p, \btheta)
-b_p(\bu_p, w_p)  \\[1ex]
\qquad
= 
( \wh{g}_{\bv_f}, \bv_f)_{\Omega_f}+
( \wh{g}_\bphi, \bphi)_{\Omega_p}+
( \wh{g}_{\bv_p}, \bv_p)_{\Omega_p}+
( \wh{g}_{\btau_p}, \btau_p)_{\Omega_p}+
(\wh{g}_{w_p},w_p)_{\Omega_p}, \\[1ex]
 -b_f(\bu_f, w_f)
-b_s(\bsi_p,\bv_s)
-b_{\sk}(\bsi_p,\bchi_p)
-b_{\Gamma}(\bu_f, \bu_p, \btheta ;\xi) \\[1ex]
\qquad 
= 
(\wh{g}_{w_f},w_f)_{\Omega_f}+
( \wh{g}_{\bv_s}, \bv_s)_{\Omega_p}+
( \wh{g}_{\bchi_p}, \bchi_p)_{\Omega_p}+
( \wh{g}_\xi, \xi)_{\Omega_p}.
\end{array}
\end{equation}
Letting
$$
\bQ_2=\bV_f \times \bLambda_s \times \bV_p \times \bbX_{p,2} \times \W_{p,2},
$$
the resolvent system \eqref{eq: resolvent formulation 1} can be written in an operator form as
\begin{equation}\label{eq: resolvent formulation 2}
  \arraycolsep=1.7pt
  \begin{array}{rcl}  
(\cE_1 + \cA) \bp + \cB'\br & = & \wh{\bF}  \qin  \bQ_2', \\[1ex]
    - \cB \bp & = & \wh{\bG} \qin \bS'.
\end{array}
\end{equation}    
where $\wh{\bF} \in \bQ_2'$ and $\wh{\bG} \in \bS'$ are the functionals on the right hand
side of \eqref{eq: resolvent formulation 1}.

To prove the solvability of this resolvent system, we use a
regularization technique, following the approach in
\cite{Showalter-SIAMMA,aeny2019}.  To that end, we introduce
operators that will be used to regularize the problem. Let $R_{\bu_p}:
\bV_p \to \bV_p'$, $R_{\bsi_p}: \bbX_p \to \bbX_p'$, $R_{p_p}: \W_p
\to \W_p'$, $L_{p_f}: \W_f \to \W_f'$, $L_{\bu_s}: \bV_s \to \bV_s'$,
and $L_{\bgamma_p}: \bbQ_p \to \bbQ_p'$ be defined as follows:
\begin{align*}
& (R_{\bu_p}\bu_p,\bv_p) = r_{\bu_p}(\bu_p, \bv_p)
:= (\nabla \cdot \bu_p, \nabla \cdot \bv_p)_{\Omega_p}, \\[1ex]
& (R_{\bsi_p}\bsi_p,\btau_p) = r_{\bsi_p}(\bsi_p, \btau_p)
:= (\bsi_p, \btau_p)_{\Omega_p} + (\nabla \cdot \bsi_p, \nabla \cdot \btau_p)_{\Omega_p}, \\[1ex]
& (R_{p_p}p_p, w_p) = r_{p_p}(p_p,w_p) := (p_p, w_p)_{\Omega_p}, \quad
(L_{p_f}p_f,w_f) = l_{p_f}(p_f,w_f) := (p_f, w_f)_{\Omega_f}, \\[1ex]
& (L_{\bu_s}\bu_s, \bv_s) = l_{\bu_s}(\bu_s, \bv_s) := (\bu_s, \bv_s)_{\Omega_p}, \quad
(L_{\bgamma_p}\bgamma_p, \bchi_P) = l_{\bgamma_p}(\bgamma_p, \bchi_p)
:= (\bgamma_p, \bchi_p)_{\Omega_p}.
\end{align*}
The following operator properties follow immediately from the above definitions.
\begin{lemma}\label{lem:R-L}
The operators $R_{\bu_p}$, $R_{\bsi_p}$, $R_{p_p}$, $L_{p_f}$, $L_{\bu_s}$, and $L_{\bgamma_p}$
are continuous and monotone.
\end{lemma}

For the regularization of the Lagrange multipliers, let 
$\psi(\lambda)\in \H^1(\Omega_p)$ be the weak solution of
\begin{gather*}
- \nabla \cdot \nabla \psi(\lambda) = 0 \qin \Omega_p, \\[1ex]
\psi(\lambda)  = \lambda \qon \Gamma_{fp}, \quad
\nabla \psi(\lambda) \cdot \bn_p  = 0 \qon \Gamma_p. 
\end{gather*}
Elliptic regularity and the trace inequality \eqref{trace} imply that
there exist positive constants $c$ and $C$ such that
\begin{equation}\label{equiv-1}
c\Vert \psi(\lambda) \Vert_{\H^1(\Omega_p)} \leq \Vert \lambda \Vert_{\H^{1/2}(\Gamma_{fp})}
\leq C \Vert \psi(\lambda) \Vert_{\H^1(\Omega_p)}.
\end{equation}
We define $L_\lambda: \Lambda_p \to \Lambda_p'$ as
\begin{equation}\label{eq: def Llambda}
(L_\lambda\lambda,\xi)= l_{\lambda}(\lambda,\xi):=(\nabla \psi(\lambda), \nabla \psi(\xi))_{\Omega_p}.
\end{equation}
Similarly, let 
$\bvarphi(\btheta)\in \bH^1(\Omega_p)$ be the weak solution of
\begin{gather*}
- \nabla \cdot \nabla \bvarphi(\btheta)  = \0 \qin \Omega_p, \\[1ex]
\bvarphi(\btheta)  = \btheta \qon \Gamma_{fp}, \quad 
\nabla \bvarphi(\btheta) \cdot \bn_p  = 0 \qon \Gamma_p, 
\end{gather*}
satisfying
\begin{equation}\label{equiv-2}
  c\Vert \bvarphi(\btheta) \Vert_{\H^1(\Omega_p)} \leq \Vert \btheta \Vert_{\H^{1/2}(\Gamma_{fp})}
\leq C \Vert \bvarphi(\btheta)  \Vert_{\H^1(\Omega_p)}.
\end{equation}
Let $R_{\btheta}: \bLambda_s \to \bLambda_s'$ be defined as
\begin{equation}\label{eq: def Rbtheta}
  (R_\btheta\btheta,\bphi)=r_{\btheta}(\btheta, \bphi) :=
  (\nabla \bvarphi(\btheta), \nabla \bvarphi(\bphi))_{\Omega_p}.
\end{equation}

\begin{lemma}\label{lem: R-operators 2}
The operators $L_\lambda$ and $R_\btheta$ are continuous and coercive. 
\end{lemma}
\begin{proof}
It follows from \eqref{equiv-1} and \eqref{equiv-2} that 
there exist positive constants $c$ and $C$ such that
\begin{equation}\label{R-L-bounds}
  \begin{array}{lll}  
(L_{\lambda}\lambda,\xi) \leq C \Vert \lambda \Vert_{\H^{1/2}(\Gamma_{fp})}
\Vert \xi \Vert_{\H^{1/2}(\Gamma_{fp})}, 
& \ (L_{\lambda}\lambda,\lambda) \geq  c \Vert \lambda \Vert^2_{\H^{1/2}(\Gamma_{fp})},
& \ \forall \,\lambda, \xi \in \Lambda_p,\\[1ex]
(R_{\btheta}\btheta,\bphi) \leq C \Vert \btheta \Vert_{\bH^{1/2}(\Gamma_{fp})}
\Vert \bphi \Vert_{\bH^{1/2}(\Gamma_{fp})}, 
& \ (R_{\btheta}\btheta,\btheta) \geq c \Vert \btheta \Vert^2_{\bH^{1/2}(\Gamma_{fp})},
& \ \forall \,\btheta, \bphi \in \bLambda_s.
  \end{array}
  \end{equation}
\end{proof}

\begin{lemma}\label{lem: well-posedness 1}
  For every $\wh{\bF} \in \bQ_2'$ and $\wh{\bG} \in \bS'$, there exists a
  solution of the resolvent system \eqref{eq: resolvent formulation 2}.
\end{lemma}
\begin{proof}
Define the operators $\cR: \bQ \to \bQ_2'$ and $\cL: \bS \to \bS'$ such that,
for any $\bp = (\bu_f, \btheta, \bu_p, \bsi_p, p_p)$, $\bq = (\bv_f,
\bphi, \bv_p, \btau_p, w_p) \in \bQ$ and $\br=(p_f, \bu_s, \bgamma_p,
\lambda)$, $\bs=(w_f, \bv_s, \bchi_p, \xi) \in \bS$,
\begin{align*}
&(\cR\bp,\bq)  := (R_{\bu_p}\bu_p,\bv_p) + (R_{\bsi_p}\bsi_p,\btau_p)
+ (R_{p_p}p_p,w_p) + (R_{\btheta}\btheta,\bphi), \\[1ex]
&(\cL\br,\bs)  := (L_{p_f}p_f,w_f) + (L_{\bu_s} \bu_s,\bv_s) + (L_{\bgamma_p} \bgamma_p,\bchi_p)
+ (L_\lambda\lambda,\xi).
\end{align*}
For $\epsilon > 0$, consider a regularization of \eqref{eq: resolvent formulation 1}: Given $\wh{\bF}=(\wh{g}_{\bv_f},
\wh{g}_{\bphi}, \wh{g}_{\bv_p}, \wh{g}_{\btau_p}, \wh{g}_{w_p}) \in
\bQ_2'$ and $\wh{\bG}=(\wh{g}_{w_f}, \wh{g}_{\bv_s}, \wh{g}_{\bchi_p},
\wh{g}_{\xi}) \in \bS'$, find $\bp_\epsilon = (\bu_{f,\epsilon},
\btheta_\epsilon, \bu_{p,\epsilon}, \bsi_{p,\epsilon},
p_{p,\epsilon}) \in \bQ$ and $\br_\epsilon=(p_{f,\epsilon}, \bu_{s,\epsilon},
\bgamma_{p,\epsilon}, \lambda_\epsilon) \in \bS$ such that
\begin{equation}\label{eq: regularized formulation 1}
  \arraycolsep=1.7pt
  \begin{array}{rcl}  
(\epsilon \cR + \cE_1 + \cA) \bp_\epsilon + \cB'\br_\epsilon & = & \wh{\bF}  \qin  \bQ_2', \\[1ex]
    - \cB \bp_\epsilon + \epsilon \cL \br_\epsilon & = & \wh{\bG} \qin \bS'.
    \end{array}
\end{equation}
Let the operator $\cO : \bQ \times \bS \to \bQ_2' \times \bS'$ be defined as
\begin{equation*}
\cO 
\left(\begin{array}{c}
\bq\\ 
\bs
\end{array}\right)
= \left(\begin{array}{cc}
\epsilon \cR + \cE_1 + \cA & \ \cB' \\ 
- \cB  & \  \epsilon \cL 
\end{array}\right) \,
\left(\begin{array}{c}
\bq \\
\bs
\end{array}\right).
\end{equation*}
We have%
\begin{equation*}
\left(\cO 
\left(\begin{array}{c}
\bp\\ 
\br
\end{array}\right),
\left(\begin{array}{c}
\bq\\ 
\bs
\end{array}\right)\right)
= ((\epsilon \cR + \cE_1 + \cA)\bp,\bq)
+(\cB'\br,\bq)
-(\cB\bp,\bs)
+\epsilon (\cL \br,\bs).
\end{equation*}
Lemmas \ref{lem: continuous monotonicity}--\ref{lem: R-operators 2} imply that 
$\cO$ is continuous. Moreover, using the coercivity and monotonicity bounds
\eqref{eq: continuous monotonicity 1}, \eqref{eq: continuous monotonicity 3}, and
\eqref{R-L-bounds}, we have
\begin{align}\label{eq: O coercivity}
&\left(\cO 
\left(\begin{array}{c}
\bq\\ 
\bs
\end{array}\right),
\left(\begin{array}{c}
\bq\\ 
\bs
\end{array}\right)\right)
= ((\epsilon \cR + \cE_1 + \cA)\bq,\bq) + (\epsilon \cL \bs,\bs) \nonumber\\[1ex]
&\quad =\epsilon r_{\bu_p}(\bv_p, \bv_p)
+ \epsilon r_{\bsi_p}(\btau_p, \btau_p)
+ \epsilon r_{\btheta}(\bphi, \bphi)
+ \epsilon r_{p_p}(w_p, w_p)
+ a_p(\bv_p, \bv_p) \nonumber\\[1ex]
& \qquad
+(A(\btau_p+\alpha w_p \bI), \btau_p+\alpha w_p \bI)
+ (s_0 w_p, w_p)
+ a_f(\bv_f, \bv_f)
+ a_{\BJS}(\bv_f, \bphi;\bv_f, \bphi)  \nonumber\\[1ex]
& \qquad
+ \epsilon l_{p_f}(w_f, w_f)
+ \epsilon l_{\bu_s}(\bv_s, \bv_s)
+ \epsilon l_{\bgamma_p}(\bchi_p, \bchi_p)
+ \epsilon l_\lambda(\xi, \xi)\nonumber\\[1ex]
&\quad \geq C(
\epsilon \Vert \nabla \cdot \bv_p \Vert^2_{\L^2(\Omega_p)}
+ \epsilon \Vert \btau_p \Vert^2_{\bbL^2(\Omega_p)}
+ \epsilon \Vert \nabla \cdot \btau_p \Vert^2_{\bL^2(\Omega_p)}
+ \epsilon \Vert \bphi \Vert^2_{\bH^{1/2}(\Gamma_{fp})}
+ \epsilon \Vert w_p \Vert^2_{\L^2(\Omega_p)} \nonumber\\[1ex]
&\qquad
+ \Vert \bv_p \Vert^2_{\bL^2(\Omega_p)}
+ \Vert A^{1/2}(\btau_p+\alpha w_p \bI) \Vert^2_{\bbL^2(\Omega_p)} 
+ s_0\Vert w_p \Vert^2_{\L^2(\Omega_p)}
+ \Vert \bD(\bv_f) \Vert^2_{\bbL^2(\Omega_f)} \nonumber\\[1ex]
&\qquad
+ \vert \bv_f - \bphi \vert^2_{a_{\BJS}}
+ \epsilon \Vert w_f \Vert^2_{\L^2(\Omega_p)}
+ \epsilon \Vert \bv_s \Vert^2_{\bL^2(\Omega_p)} 
+ \epsilon \Vert \bchi_p \Vert^2_{\bbL^2(\Omega_p)}
+ \epsilon \Vert \xi \Vert^2_{\H^{1/2}(\Gamma_{fp})}),
\end{align}
which implies that $\cO$ is coercive. Thus, an application of the
Lax-Milgram theorem establishes the existence of a unique solution
$(\bp_\epsilon, \br_\epsilon)\in \bQ \times \bS$ of \eqref{eq:
  regularized formulation 1}. Now, from \eqref{eq: regularized
  formulation 1} and \eqref{eq: O coercivity} we obtain
\begin{align}
&\epsilon \Vert \nabla \cdot \bu_{p,\epsilon} \Vert^2_{\L^2(\Omega_p)}
+ \epsilon \Vert \nabla \cdot \bsi_{p,\epsilon} \Vert^2_{\bL^2(\Omega_p)}
+ \epsilon \Vert \btheta_\epsilon \Vert^2_{\bH^{1/2}(\Gamma_{fp})}
+ \epsilon \Vert \bsi_{p,\epsilon} \Vert^2_{\bbL^2(\Omega_p)}
+ \epsilon \Vert p_{p,\epsilon} \Vert^2_{\L^2(\Omega_p)}
+ \Vert \bu_{p,\epsilon} \Vert^2_{\bL^2(\Omega_p)} \nonumber \\[1ex]
& \qquad
+ \Vert A^{1/2}(\bsi_{p,\epsilon}+\alpha p_{p,\epsilon}\bI) \Vert^2_{\bbL^2(\Omega_p)}
+ s_0 \Vert p_{p,\epsilon} \Vert^2_{\L^2(\Omega_p)}
+ \Vert \bu_{f,\epsilon} \Vert^2_{\bH^1(\Omega_f)}
+ \vert \bu_{f,\epsilon} - \btheta_\epsilon \vert^2_{a_{\BJS}}
+ \epsilon \Vert p_{f,\epsilon} \Vert^2_{\L^2(\Omega_p)}  \nonumber \\[1ex]
& \qquad
+ \epsilon \Vert \bu_{s,\epsilon} \Vert^2_{\bL^2(\Omega_p)}
+ \epsilon \Vert \bgamma_{p,\epsilon} \Vert^2_{\bbL^2(\Omega_p)}
+ \epsilon \Vert \lambda_\epsilon \Vert^2_{\H^{1/2}(\Gamma_{fp})} \nonumber \\[1ex]
&\quad \leq C(
\Vert \wh{g}_{\bv_f} \Vert_{\bL^2(\Omega_f)}\Vert \bu_{f,\epsilon} \Vert_{\bL^2(\Omega_f)}
+ \Vert \wh{g}_{\bphi} \Vert_{\bL^2(\Omega_p)} \Vert \btheta_{\epsilon} \Vert_{\bL^2(\Omega_p)}
+ \Vert \wh{g}_{\bv_p} \Vert_{\bL^2(\Omega_p)}\Vert \bu_{p,\epsilon} \Vert_{\bL^2(\Omega_p)}  \nonumber \\[1ex]
& \qquad
+ \Vert \wh{g}_{\btau_p} \Vert_{\bbL^2(\Omega_p)}\Vert \bsi_{p,\epsilon} \Vert_{\bbL^2(\Omega_p)}  
+ \Vert \wh{g}_{w_p} \Vert_{\L^2(\Omega_p)}\Vert p_{p,\epsilon} \Vert_{\L^2(\Omega_p)} 
+ \Vert \wh{g}_{w_f} \Vert_{\L^2(\Omega_f)}\Vert p_{f,\epsilon} \Vert_{\L^2(\Omega_f)}  \nonumber \\[1ex]
& \qquad
+ \Vert \wh{g}_{\bv_s} \Vert_{\bL^2(\Omega_p)}\Vert \bu_{s,\epsilon} \Vert_{\bL^2(\Omega_p)}
+ \Vert \wh{g}_{\bchi_p} \Vert_{\bbL^2(\Omega_p)}\Vert \bgamma_{p,\epsilon} \Vert_{\bbL^2(\Omega_p)}
+ \Vert \wh{g}_{\xi} \Vert_{\L^2(\Omega_p)}\Vert \lambda_{\epsilon} \Vert_{\L^2(\Omega_p)}),
\end{align}
which implies that $\Vert \bu_{p,\epsilon} \Vert_{\bL^2(\Omega_p)}$,
$\Vert A^{1/2}(\bsi_{p,\epsilon}+\alpha p_{p,\epsilon}\bI)
\Vert_{\bbL^2(\Omega_p)}$ and $\Vert \bu_{f,\epsilon}
\Vert_{\bH^1(\Omega_f)}$ are bounded independently of $\epsilon$.
Next, from \eqref{eq: regularized formulation 1} we have
\begin{align}\label{eq: regularized formulation 2}
&(A (\bsi_{p,\epsilon}+\alpha p_{p, \epsilon} \bI), \btau_p)_{\Omega_p}
+ \epsilon (\bsi_{p,\epsilon}, \btau_p)_{\Omega_p}
+ \epsilon (\nabla \cdot \bsi_{p,\epsilon}, 
\nabla \cdot \btau_p)_{\Omega_p} \nonumber \\[1ex]
& \quad - b_n^p(\btau_p, \btheta_\epsilon)
+ b_s(\btau_p, \bu_{s,\epsilon})
+ b_{\sk}(\btau_p, \bgamma_{p,\epsilon})
 = (\wh{g}_{\tau_p}, \btau_p)_{\Omega_p}.
\end{align}
Applying the inf-sup condition \eqref{eq: continuous inf-sup 2} results in
\begin{align}\label{eq: inf-sup app epsilon}
& \ds \|\bu_{s,\epsilon}\|_{\bL^2(\Omega_p)}
+\|\bgamma_{p,\epsilon}\|_{\bbL^2(\Omega_p)}
\leq C
\underset{\btau_p \in \bbX_p\, \text{s.t.} \btau_p \bn_p=\0 \, \text{on} \, \Gamma_{fp}}{\sup}
\frac{b_s(\btau_p,\bu_{s,\epsilon})
+b_{\sk}(\btau_p,\bgamma_{p,\epsilon})}{\|\btau_p\|_{\bbX_p}}  \nonumber \\[1ex]
& \ds  
\quad = C\underset{\btau_p \in \bbX_p\, \text{s.t.} \btau_p \bn_p=\0 \, \text{on} \, \Gamma_{fp}}{\sup}
\bigg(\frac{
- (A (\bsi_{p,\epsilon}+\alpha p_{p, \epsilon} \bI), \btau_p)_{\Omega_p}
- \epsilon (\bsi_{p,\epsilon}, \btau_p)_{\Omega_p}
- \epsilon (\nabla \cdot \bsi_{p,\epsilon}, 
\nabla \cdot \btau_p)_{\Omega_p})_{\Omega_p}}
{\|\btau_p\|_{\bbX_p}} \nonumber \\[1ex]
& \qquad \qquad \qquad \qquad \qquad \qquad +\frac{
 b_n^p(\btau_p, \btheta_\epsilon )
+ (\wh{g}_{\btau_p}, \btau_p)_{\Omega_p}}
{\|\btau_p\|_{\bbX_p}} \bigg) \nonumber \\[1ex]
& \ds
\quad \leq C
(\Vert A(\bsi_{p,\epsilon}+\alpha p_{p,\epsilon}\bI) \Vert_{\bbL^2(\Omega_p)}
+ \epsilon \Vert \bsi_{p,\epsilon} \Vert_{\bbL^2(\Omega_p)}
+ \epsilon \Vert \nabla \cdot \bsi_{p, \epsilon} \Vert_{\bL^2(\Omega_p)}
+ \Vert \wh{g}_{\btau_p} \Vert_{\bbL^2(\Omega_p)}),
\end{align}
where the term $b_n^p(\btau_p, \btheta_\epsilon )$ vanishes due to the
restriction $\btau_p \bn_p = \0$ on $\Gamma_{fp}$. Also,
applying the inf-sup condition \eqref{eq: continuous inf-sup 4} and using \eqref{eq:
  regularized formulation 2}, we obtain
\begin{align}\label{eq: inf-sup bound}
& \ds \Vert \btheta_{\epsilon} \Vert_{\bH^{1/2}(\Gamma_{fp})} 
 \leq C
\underset{\btau_p \in \bbX_p \, \text{s.t.}\, \nabla \cdot \btau_p=0}{\sup} \frac{
b_n^p(\btau_p, \btheta_\epsilon)
}{\|\btau_p\|_{\bbX_p}}  \nonumber \\[1ex]
& \quad 
\ds = C
\underset{\btau_p \in \bbX_p \, \text{s.t.}\, \nabla \cdot \btau_p=0}{\sup}
% (
\frac{
 (A (\bsi_{p,\epsilon}+\alpha p_{p, \epsilon} \bI), \btau_p)_{\Omega_p}
+ \epsilon  (\bsi_{p,\epsilon}, \btau_p)_{\Omega_p}
+ b_{\sk}(\btau_p, \bgamma_{p,\epsilon})
- (\wh{g}_{\btau_p}, \btau_p)_{\Omega_p}
}
{\Vert \btau_p \Vert_{\bbX_p}}
\nonumber \\[1ex]
&\ds  \quad 
\leq C
(\Vert A(\bsi_{p,\epsilon}+\alpha p_{p,\epsilon} \bI) \Vert_{\bbL^2(\Omega_p)}
+ \epsilon \Vert \bsi_{p,\epsilon} \Vert_{\bbL^2(\Omega_p)}
+ \Vert \bgamma_{p,\epsilon} \Vert_{\bbL^2(\Omega_f)}
+ \Vert \wh{g}_{\btau_p} \Vert_{\bbL^2(\Omega_p)}).
\end{align}
Bounds \eqref{eq: inf-sup app epsilon} and
\eqref{eq: inf-sup bound} imply that 
$\Vert \bu_{s,\epsilon} \Vert_{\bL^2(\Omega_p)}$ $\Vert
\bgamma_{p,\epsilon} \Vert_{\bbL^2(\Omega_p)}$, and
$\Vert \btheta_\epsilon \Vert_{\bH^{1/2}(\Gamma_{fp})}$
are bounded
independently of $\epsilon$. In addition, \eqref{eq: regularized formulation 1} gives
\begin{align}
& a_p(\bu_{p,\epsilon},\bv_p)
+\epsilon(\nabla \cdot \bu_{p,\epsilon}, \nabla \cdot \bv_p)_{\Omega_p}
+ b_p(\bv_p, p_{p,\epsilon})
+ \langle \bv_p \cdot \bn_p, \lambda_\epsilon \rangle_{\Gamma_{fp}}
+ a_f(\bu_{f,\epsilon}, \bv_f) 
\nonumber \\[1ex]
& \quad
+ a_{\BJS}(\bu_{f,\epsilon}, \btheta_\epsilon; \bv_f, \0)
+ b_f(\bv_f, p_{f,\epsilon})
+ \langle \bv_f \cdot \bn_f, \lambda_\epsilon \rangle_{\Gamma_{fp}}
= 0,
\end{align}
so applying the inf-sup condition \eqref{eq: continuous inf-sup 3}, we obtain
\begin{align}\label{pf-eps-inf-sup}
& \ds
\Vert p_{f,\epsilon} \Vert_{\L^2(\Omega_f)} 
+ \Vert p_{p,\epsilon} \Vert_{\L^2(\Omega_p)} 
+ \Vert \lambda_\epsilon \Vert_{\H^{1/2}(\Gamma_{fp})} \nonumber \\[1ex]
& \ds \quad
 \leq C
\underset{(\bv_f, \bv_p, \0) \in \bV_f\times\bV_p\times\bLambda_s}{\sup}
\frac{
b_f(\bv_f,p_{f,\epsilon})+b_p(\bv_p,p_{p,\epsilon})+
b_{\Gamma}(\bv_f, \bv_p, \0; \lambda_\epsilon)}
{\|(\bv_f, \bv_p, \0)\|_{\bV_f\times\bV_p\times\bLambda_s}} \nonumber \\[1ex]
& \ds  \quad
= C
\underset{(\bv_f, \bv_p, \0) \in \bV_f\times\bV_p\times\bLambda_s}{\sup}
\frac{-a_p(\bu_{p,\epsilon},\bv_p)
-\epsilon(\nabla \cdot \bu_{p,\epsilon}, \nabla \cdot \bv_p)
-a_f(\bu_{f,\epsilon},\bv_f)
-a_{\BJS}(\bu_{f,\epsilon},\btheta_\epsilon;\bv_f, \0)
}
{\|(\bv_f, \bv_p, \0)\|_{\bV_f\times\bV_p\times\bLambda_s}}\nonumber  \\[1ex]
& \ds \quad 
\leq C
(\Vert \bu_{p,\epsilon} \Vert_{\bL^2(\Omega_p)}
+ \epsilon \Vert \nabla \cdot \bu_{p, \epsilon} \Vert_{\L^2(\Omega_p)}
+ \Vert \bu_{f,\epsilon} \Vert_{\bH^1(\Omega_f)}
+ \vert \bu_{f,\epsilon} - \btheta_\epsilon \vert_{a_{\BJS}}).
\end{align}
Therefore we have that $\Vert p_{f,\epsilon} \Vert_{\L^2(\Omega_f)}$,
$\Vert p_{p,\epsilon} \Vert_{\L^2(\Omega_p)}$ and $\Vert
\lambda_\epsilon \Vert_{\H^{1/2}(\Gamma_{fp})}$ are also bounded
independently of $\epsilon$.

Since $\nabla \cdot \bbX_p = \bV_s$, by taking $\bv_s = \nabla \cdot \bsi_{p,\epsilon}$ in \eqref{eq: regularized formulation 1}, we have
\begin{equation}
\Vert \nabla \cdot \bsi_{p, \epsilon} \Vert_{\bL^2(\Omega_p)} \leq 
\epsilon \Vert \bu_{s, \epsilon} \Vert_{\bL^2(\Omega_p)} + \Vert \wh{g}_{\bv_s} \Vert_{\bL^2(\Omega_p)}, 
\end{equation}
which implies that $\Vert \nabla \cdot \bsi_{p, \epsilon}
\Vert_{\bL^2(\Omega_p)}$ is bounded independently of $\epsilon$. Since
$\Vert A^{1/2}(\bsi_{p,\epsilon}+\alpha p_{p,\epsilon}\bI)
\Vert_{\bbL^2(\Omega_p)}$, $\Vert p_{p,\epsilon}
\Vert_{\L^2(\Omega_p)}$ and $\Vert \nabla \cdot \bsi_{p,\epsilon}
\Vert_{\bL^2(\Omega_p)}$ are all bounded independently of $\epsilon$,
the same holds for 
$\Vert \bsi_{p,\epsilon} \Vert_{\bbH(\div,\Omega_p)}$.
Finally, since $\nabla \cdot
\bV_p = \W_p$, by taking $w_p = \nabla \cdot \bu_{p,\epsilon}$ in
\eqref{eq: regularized formulation 1}, we have
\begin{equation}
\Vert \nabla \cdot \bu_{p,\epsilon} \Vert_{\L^2(\Omega_p)} \leq C(\Vert \bsi_{p,\epsilon} \Vert_{\bbL^2(\Omega_p)}
+ (s_0+\epsilon) \Vert p_{p,\epsilon} \Vert_{\L^2(\Omega_p)}
+ \Vert \wh{g}_{w_p} \Vert_{\L^2(\Omega_p)}),
\end{equation}
so $\Vert \nabla \cdot \bu_{p, \epsilon} \Vert_{\L^2(\Omega_p)}$, and therefore
$\Vert \bu_{p,\epsilon} \Vert_{\bV_p}$ is bounded
independently of $\epsilon$. Thus we conclude that all the variables
are bounded independently of $\epsilon$.

Since $\bQ$ and $\bS$ are reflexive Banach spaces, as $\epsilon \to 0$
we can extract weakly convergent subsequences
$\{\bp_{\epsilon,n}\}_{n=1}^\infty$ and
$\{\br_{\epsilon,n}\}_{n=1}^\infty$ such that $\bp_{\epsilon,n} \to
\bp$ in $\bQ$, $\br_{\epsilon,n} \to \br$ in $\bS$. Taking the limit in
\eqref{eq: regularized formulation 1}, we obtain that $(\bp,\br)$ is a solution to
\eqref{eq: resolvent formulation 2}.
\end{proof}

\begin{lemma}\label{lem: range condition}
For $\cN$, $\cM$ and $E_b'$ defined in \eqref{eq: definition N and M} and \eqref{defn-Eb},
it holds that $Rg(\cN + \cM)=E_b'$, that is, given $f \in E_b'$, there exists
$v \in \cD$ such that $(\cN + \cM)v=f$.
\end{lemma}
\begin{proof}
Given any $\wh{g}_{\btau_p} \in \bbX_{p,2}'$ and $\wh{g}_{w_p} \in
\W_{p,2}'$, according to Lemma~\ref{lem: well-posedness 1}, there exist
$(\bp,\br) \in \bQ\times\bS$ such that
\begin{equation*}
  \arraycolsep=1.7pt
  \begin{array}{rcl}    
(\cE_1 + \cA) \bp + \cB'\br & = & \wh{\bF}  \qin  \bQ_{2,0}', \\[1ex]
    - \cB \bp & = & \0 \qin \bS_{2,0}',
  \end{array}
\end{equation*}
where $\wh{\bF}  = (\0,0,0,\wh{g}_{\btau_p},\wh{g}_{w_p})^\rt \in \bQ_{2,0}'$, implying
the range condition.
\end{proof}

We are now ready to establish existence for the auxiliary initial value problem
\eqref{eq: reduced formulation 1}, assuming compatible initial data.

\begin{theorem}\label{thm: well-posed reduced problem}
For each compatible initial data $(\wh\bp_0, \wh\br_0)\in \cD$
and each $(\wh{g}_{\btau_p},\wh{g}_{w_p}) \in \W^{1,1}(0,T;\bbX_{p,2}') \times \W^{1,1}(0,T; \W_{p,2}')$, there exists a solution to \eqref{eq: reduced formulation 1} with
$(\bsi_p(0), p_p(0))=(\wh\bsi_{p,0}, \wh p_{p,0})$ and $(\bu_f, p_f,\bsi_p, \bu_s, \bgamma_p, \linebreak \bu_p, p_p,\lambda, \btheta)
  \in \L^\infty(0,T;\bV_f)\times
\L^\infty(0,T;\W_f)\times
\W^{1,\infty}(0,T; \bbL^2(\Omega_p))\cap \L^{\infty}(0,T; \bbX_p)\times
\L^\infty(0,T;\bV_s)\times
\L^\infty(0,T;\bbQ_p)\times
\L^\infty(0,T;\bV_p)\times
\W^{1,\infty}(0,T;\W_p)  \times
\L^\infty(0,T;\Lambda_p)\times
  \L^\infty(0,T;\bLambda_s)$.
\end{theorem}
\begin{proof}
  Using Lemma~\ref{lem:oper-prop} and Lemma~\ref{lem: range condition},
  we apply
Theorem \ref{thm: Showalter} with $E$, $\cN$ and $\cM$ defined in
\eqref{eq: definition N and M} to obtain existence of a solution to
\eqref{eq: reduced formulation 1} with $\bsi_p \in \W^{1,\infty}(0,T;
\bbL^2(\Omega_p))$ and $p_p\in \W^{1,\infty}(0,T;\W_p)$. From the equations
\eqref{eq: reduced formulation 1} and the inf-sup conditions in
Lemma~\ref{lem: continuous inf-sup} we can further deduce that 
$\bu_f \in \L^\infty(0,T;\bV_f)$, $p_f \in \L^\infty(0,T;\W_f)$, 
$\bsi_p \in \L^{\infty}(0,T; \bbX_p)$, $\bu_s \in \L^\infty(0,T;\bV_s)$,
$\bgamma_p \in \L^\infty(0,T;\bbQ_p)$, $\bu_p \in \L^\infty(0,T;\bV_p)$,
$\lambda \in \L^\infty(0,T;\Lambda_p)$, and $\btheta \in \L^\infty(0,T;\bLambda_s)$.
\end{proof}

We will employ Theorem~\ref{thm: well-posed reduced problem} to obtain
existence of a solution to our problem \eqref{eq: continuous
  formulation 2}. To that end, we first construct compatible initial
data $(\bp_0,\br_0)$.

\begin{lemma}\label{lem: initial condition}
Assume that the initial data $p_{p,0}\in \W_p \cap \H$, where
\begin{align}
  \ds \H:=\big\{ w_p \in \H^1(\Omega_p): \bK\nabla w_p \in \bH^1(\Omega_p),
  \ \bK\nabla w_p \cdot \bn_p = 0 \qon \Gamma_p^{N_v},  \ w_p=0 \qon \Gamma_p^{D_p} \big\}.
  \label{eq: initial condition H}
\end{align}
Then, there exist $\bp_0:=(\bu_{f,0}, \btheta_0, \bu_{p,0}, \bsi_{p,0}, p_{p,0}) \in \bQ$
and $\br_0:=(p_{f,0}, \bu_{s,0}, \bgamma_{p,0}, \lambda_0)\in \bS$ such that
\begin{equation}\label{init-data}
  \arraycolsep=1.7pt
  \begin{array}{rcl}    
\cA \bp_0 + \cB'\br_0 & = & \wh{\bF}_0  \qin  \bQ_{2}', \\[1ex]
    - \cB \bp_0 & = & \bG(0) \qin \bS',
  \end{array}
\end{equation}
where $\wh{\bF}_0  = (\f_f(0),0,0,\wh{g}_{\btau_p},\wh{g}_{w_p})^\rt \in \bQ_{2}'$, 
with suitable $\wh{g}_{\btau_p} \in \bbX_{p,2}'$ and $\wh{g}_{w_p} \in \W_{p,2}'$.
\end{lemma}
\begin{proof}
Our approach is to solve a sequence of well-defined subproblems, using
the previously obtained solutions as data to guarantee that we obtain
a solution of the coupled problem \eqref{init-data}. We proceed as follows.

1. Define $\ds \bu_{p,0}:=-\mu^{-1}\bK\nabla p_{p,0} \in \bH^1(\Omega_p)$,
with $p_{p,0}\in \W_p\cap \H$, cf. \eqref{eq: initial condition H}. It follows that
\begin{equation*}
\mu \bK^{-1} \bu_{p,0} = -\nabla p_{p,0},
\quad 
\nabla \cdot \bu_{p,0} = -\mu^{-1}\nabla \cdot (\bK \nabla p_{p,0}) \qin \Omega_p,
\quad
\bu_{p,0}\cdot\bn_p = 0 \qon \Gamma_p^{N_v}.
\end{equation*}
Next, define $\lambda_0=p_{p,0} \vert_{\Gamma_{fp}}\in
\Lambda_p$. Testing the first two equations above with $\bv_p \in \bV_p$ and
$w_p \in \W_p$, respectively, we obtain
\begin{equation}\label{u_p0}
  \begin{array}{ll}
    \ds a_p(\bu_{p,0}, \bv_p) + b_p(\bv_p, p_{p,0})
    + \langle \bv_p \cdot \bn_p, \lambda_0 \rangle_{\Gamma_{fp}} = 0,
    & \ \forall \, \bv_p \in \bV_p, \\[2ex]
    \ds -b_p(\bu_{p,0}, w_p) = -\mu^{-1}(\nabla \cdot (\bK \nabla p_{p,0}), w_p)_{\Omega_p},
& \ \forall\, w_p \in \W_p.
  \end{array}
\end{equation}

2. Define $(\bu_{f,0}, p_{f,0}) \in \bV_f \times \W_f$ such that
\begin{equation}\label{u_f0}
  \begin{array}{ll}
    \ds a_f(\bu_{f,0}, \bv_f) + b_f(\bv_f,p_{f,0}) 
  & \qquad\quad \\
  \ds = - \sum_{j=1}^{n-1}\langle \mu \alpha_{\BJS} \sqrt{\bK_j^{-1}} \bu_{p,0}
    \cdot \bt_{f,j}, \bv_f \cdot \bt_{f,j} \rangle_{\Gamma_{fp}}
    - \langle \bv_f \cdot \bn_f, \lambda_0 \rangle_{\Gamma_{fp}}  + (\f_f(0),\bv_f)_{\Omega_f},
& \ \forall \,\bv_f \in \bV_f, \\ [2ex]
\ds -b_f(\bu_{f,0}, w_f) = (q_f(0),w_f), & \ \forall \, w_f \in \W_f.
  \end{array}
\end{equation}
This is a well-posed problem, since it corresponds to the weak
solution of the Stokes system with mixed boundary conditions on $\Gamma_{fp}$.
Note that $\lambda_0$ and $\bu_{p,0}$ are data for this problem.

\medskip
3. Define $\ds (\bsi_{p,0}, \bbeta_{p,0}, \brho_{p,0}, \bpsi_0) \in \bbX_p
\times \bV_s \times \bbQ_p \times \bLambda_s $ such that
\begin{equation}\label{sigma_p0}
  \begin{array}{ll}
    \ds (A \bsi_{p,0}, \btau_p)_{\Omega_p}
    + b_s( \btau_p, \bbeta_{p,0})
    + b_{\sk}(\btau_p, \brho_{p,0})
    - b_n^p(\btau_p, \bpsi_0)
    = - (A \alpha p_{p,0} \bI,\btau_p)_{\Omega_p}, 
& \ \forall \, \btau_p \in \bbX_p, \\[2ex]
    \ds b_n^p(\bsi_{p,0}, \bphi) = \sum_{j=1}^{n-1}\langle \mu \alpha_{\BJS} \sqrt{\bK_j^{-1}} \bu_{p,0}
    \cdot \bt_{f,j}, \bphi \cdot \bt_{f,j} \rangle_{\Gamma_{fp}}
- \langle \bphi \cdot \bn_p, \lambda_0 \rangle_{\Gamma_{fp}},
& \ \forall \, \bphi \in \bLambda_s, \\[2ex]
\ds  - b_s ( \bsi_{p,0}, \bv_s ) = (\f_p(0),\bv_s)_{\Omega_p}, & \ \forall \, \bv_s \in \bV_s, \\[2ex]
\ds  - b_{\sk} ( \bsi_{p,0}, \bchi_p ) = 0, & \ \forall \, \bchi_p \in \bbQ_p.
  \end{array}
\end{equation}
This is a well-posed problem corresponding to the weak solution of
the mixed elasticity system with mixed boundary conditions on $\Gamma_{fp}$.
Note that $p_{p,0}$, $\bu_{p,0}$ and $\lambda_0$ are
data for this problem. Here $\bbeta_{p,0}$, $\brho_{p,0}$, and
$\bpsi_0$ are auxiliary variables that are not part of the constructed initial data.
However, they can be used to recover the variables $\bbeta_{p}$, $\brho_{p}$, and
$\bpsi$ that satisfy the non-differentiated equation \eqref{non-diff-eq}.

\medskip
4. Define $\btheta_0 \in \bLambda_s$ as
\begin{equation}\label{theta_0}
\btheta_0 =  \bu_{f,0} - \bu_{p,0} 
\qon \Gamma_{fp}, 
\end{equation}
where $\bu_{f,0}$ and $\bu_{p,0}$ are data obtained in the previous
steps. Note that \eqref{theta_0} implies that the BJS terms in
\eqref{u_f0} and \eqref{sigma_p0} can be rewritten with
$\bu_{p,0}\cdot \bt_{f,j}$ replaced by $(\bu_{f,0}-\btheta_0)\cdot \bt_{f,j}$ and
that \eqref{weak-form-8} holds for the initial data.

\medskip
5. Define $(\wh{\bsi}_{p,0}, \bu_{s,0}, \bgamma_{p,0}) \in \bbX_p \times \bV_s \times \bbQ_p$ such that
\begin{equation}\label{u_s0}
  \begin{array}{ll}
    \ds (A \wh\bsi_{p,0}, \btau_p)_{\Omega_p}  + b_s( \btau_p, \bu_{s,0})
    + b_{\sk}(\btau_p, \bgamma_{p,0}) = b_n^p(\btau_p, \btheta_0), 
& \ \forall \, \btau_p \in \bbX_p, \\[2ex]
\ds  -b_s ( \wh\bsi_{p,0}, \bv_s ) = 0, & \ \forall \, \bv_s \in \bV_s, \\[2ex]
\ds  - b_{\sk} ( \wh\bsi_{p,0}, \bchi_p ) = 0, & \ \forall \, \bchi_p \in \bbQ_p.
  \end{array}
\end{equation}
This is a well-posed problem, since it corresponds to the weak solution of the
mixed elasticity system with Dirichlet data $\btheta_0$ on $\Gamma_{fp}$. We note that
$\wh\bsi_{p,0}$ is an auxiliary variable not used in the initial data.

Combining \eqref{u_p0}--\eqref{u_s0}, we obtain
$(\bu_{f,0}, \btheta_0, \bu_{p,0}, \bsi_{p,0}, p_{p,0})\in \bQ$ and
$(p_{f,0}, \bu_{s,0}, \bgamma_{p,0}, \lambda_0) \in \bS$ satisfying \eqref{init-data} with
\begin{equation*}
  (\wh{g}_{\btau_p}, \btau_p)_{\Omega_p} = - (A(\wh{\bsi}_{p,0}),\btau_p)_{\Omega_p}, \quad
(\wh{g}_{w_p}, w_p)_{\Omega_p}
  = - b_p( \bu_{p,0}, w_p).
\end{equation*}
The above equations imply
\begin{equation*}
  \Vert \wh{g}_{\btau_p} \Vert_{\bbL^2(\Omega_p)} + \Vert \wh{g}_{w_p} \Vert_{\L^2(\Omega_p)}
  \leq C( 
  \Vert \wh{\bsi}_{p,0} \Vert_{\bbL^2(\Omega_p)} + \Vert \nabla \cdot \bu_{p,0} \Vert_{\L^2(\Omega_p)}),
\end{equation*}
hence $(\wh{g}_{\btau_p}, \wh{g}_{w_p})\in \bbX_{p,2}'\times \W_{p,2}'$, completing the proof.
\end{proof}

We are now ready to prove the main result of this section.

\begin{theorem}\label{thm: well-posedness continuous}
  For each compatible initial data $(\bp_0, \br_0)\in \cD$ constructed in
  Lemma~\ref{lem: initial condition} and each
\begin{equation*}
\f_f \in \W^{1,1}(0,T;\bV_f'),\quad 
\f_p\in \W^{1,1}(0,T;\bV_s'),\quad 
q_f\in \W^{1,1}(0,T;\W_f'),\quad 
q_p\in \W^{1,1}(0,T;\W_p'),
\end{equation*}
there exists a unique solution of \eqref{weak-form}
$(\bu_f, p_f,\bsi_p, \bu_s, \bgamma_p, \bu_p, p_p,\lambda, \btheta)
\in \L^\infty(0,T;\bV_f)\times
\L^\infty(0,T;\W_f)\times \linebreak 
\W^{1,\infty}(0,T; \bbL^2(\Omega_p))\cap \L^{\infty}(0,T; \bbX_p)\times
\L^\infty(0,T;\bV_s)\times
\L^\infty(0,T;\bbQ_p)\times
\L^\infty(0,T;\bV_p)\times
\W^{1,\infty}(0,T;\W_p)  \times
\L^\infty(0,T; \Lambda_p)\times
\L^\infty(0,T;\bLambda_s)$ with $(\bsi_p(0),p_p(0)) = (\bsi_{p,0},p_{p,0})$.
\end{theorem}
\begin{proof}
For each fixed time $t \in [0,T]$, Lemma~\ref{lem: well-posedness 1} implies that there exists a
solution to the resolvent system
\eqref{eq: resolvent formulation 2} with $\wh{\bF}=\bF(t)$ and
$\wh{\bG}=\bG(t)$ defined in \eqref{eq: continuous formulation 3}. In other words, there exist
$(\wt{\bp}(t), \wt{\br}(t))$ such that
\begin{equation}\label{tilde-resolvent}
    \arraycolsep=1.7pt
  \begin{array}{rcl}  
(\cE_1 + \cA) \,\wt{\bp}(t) + \cB'\,\wt{\br}(t) & = & \bF(t)  \qin  \bQ_2', \\[1ex]
- \cB \,\wt{\bp}(t) & = & \bG(t) \qin \bS'. 
\end{array}
\end{equation}
We look for a solution to \eqref{eq: continuous formulation 3} in the form
$\bp(t) = \wt{\bp}(t) + \wh\bp(t)$, $\br(t) = \wt{\br}(t) + \wh\br(t)$. Subtracting
\eqref{tilde-resolvent} from \eqref{eq: continuous formulation 3} leads to the reduced evolution
problem
\begin{equation}\label{reduced}
    \arraycolsep=1.7pt
  \begin{array}{rcl}    
\ds \partial_t \,\cE_1\,\wh{\bp}(t)
+ \cA\,\wh{\bp}(t) + \cB'\,\wh{\br}(t) & = &
\cE_1\,\wt{\bp}(t)- \partial_t\,\cE_1\,\wt{\bp}(t) \qin \bQ_{2,0}',\\[1ex]
\ds - \cB\,\wh{\bp}(t) & = & \0  \qin \bS_{2,0}',
  \end{array}
  \end{equation}
with initial condition $\wh{\bp}(0)=\bp_{0}-\wt{\bp}(0)$ and
$\wh{\br}(0)=\br_{0}-\wt{\br}(0)$. Subtracting \eqref{tilde-resolvent} at $t = 0$ from
\eqref{init-data} gives
\begin{equation*}
  \arraycolsep=1.7pt
  \begin{array}{rcl}    
\cA \, \wh\bp(0) + \cB' \, \wh\br(0) & = & \cE_1\wt\bp(0) + \wh{\bF}_0 - \bF(0)   \qin  \bQ_{2,0}', \\[1ex]
    - \cB \, \wh\bp(0) & = & \0 \qin \bS_{2,0}',
  \end{array}
\end{equation*}
We emphasize that in the above, $\wh{\bF}_0 - \bF(0) =
(\0,0,0,\wh{g}_{\btau_p},\wh{g}_{w_p} - q_p(0))^\rt \in \bQ_{2,0}'$. Therefore,
$\cM \left( \begin{array}{c} \wh\bp(0) \\ \wh\br(0) \end{array} \right) \in E_b'$, i.e.,
$(\wh\bp(0),\wh\br(0)) \in \cD$. Thus, the reduced evolution problem \eqref{reduced}
is in the form of \eqref{eq: reduced formulation 1}. According to Theorem \ref{thm:
  well-posed reduced problem}, it has a
solution, which establishes the existence of a solution to \eqref{weak-form}
with the stated regularity satisfying
$(\bsi_p(0),p_p(0)) = (\bsi_{p,0},p_{p,0})$.

We next show that the solution is unique. Since the problem is linear,
it is sufficient to prove that the problem with zero data has only the
zero solution. Taking $\bF = \bG = \0$ in \eqref{eq: continuous formulation 3}
and testing it with the solution $(\bp,\br)$ yields
\begin{align*}
&\frac{1}{2}\partial_t \left(\Vert A^{1/2}(\bsi_p +\alpha p_p \,\bI)\Vert^2_{\bbL^2(\Omega_p)} 
+ s_0 \Vert p_p \Vert^2_{\L^2(\Omega_p)} \right) + a_p(\bu_p , \bu_p)
+ a_f(\bu_f, \bu_f)
+ a_{\BJS}( \bu_f, \btheta; \bu_f, \btheta)=0.
\end{align*}
Integrating in time from $0$ to $t \in (0, T]$ and using that the initial data is zero,
as well as the coercivity of $a_p$ and $a_f$ and monotonicity of $a_{\BJS}$,
cf. \eqref{eq: continuous monotonicity 1}, we conclude that $\bsi_p = \0$, $p_p = 0$,
$\bu_p = \0$, and $\bu_f = 0$. Then the inf-sup conditions
\eqref{eq: continuous inf-sup 2}--\eqref{eq: continuous inf-sup 4} imply that
$\bu_s = \0$, $\bgamma_p = \0$, $\btheta = \0$, $p_f = 0$, and $\lambda = 0$, using
arguments similar to \eqref{eq: inf-sup app epsilon}--\eqref{pf-eps-inf-sup}.
Therefore the solution of \eqref{eq: continuous formulation 2} is unique.
\end{proof}

\begin{corollary}\label{cor:init-data}
The solution of \eqref{eq: continuous formulation 2}
satisfies $\bu_f(0) = \bu_{f,0}$, $p_f(0) = p_{f,0}$, $\bu_p(0) =
\bu_{p,0}$, $\lambda(0) = \lambda_0$,
and $\btheta(0) = \btheta_0$.
\end{corollary}
\begin{proof}
Since $\bu_f \in \L^\infty(0,T;\bV_f)$, we can define $\bu_f(0) := \lim_{t \to 0^+} \bu_f(t)$.
Let $\ov\bu_f := \bu_f(0) - \bu_{f,0}$, with a similar definition and notation for
the rest of the variables.  Taking $t \to 0^+$ in all equations without
time derivatives in \eqref{eq: continuous formulation 2} and using
that the initial data $(\bp_0, \br_0)$ satisfies the same equations at
$t=0$, cf. \eqref{init-data}, and that $\ov\bsi_p = \0$ and $\ov p_p = 0$, we obtain
\begin{subequations}
\begin{align}
&\ds (2\mu \bD(\ov\bu_f),\bD(\bv_f))_{\Omega_f}
-(\nabla\cdot \bv_f, \ov p_f)_{\Omega_f} 
+\langle \bv_f\cdot \bn_f,  \ov \lambda \rangle_{\Gamma_{fp}}  \nonumber \\[1ex]
& \qquad +\sum_{j=1}^{n-1}{\langle \mu \alpha_{\BJS}\sqrt{\bK_j^{-1}}(\ov \bu_f - \ov\btheta)\cdot \bt_{f,j},
  \bv_f \cdot \bt_{f,j}\rangle_{\Gamma_{fp}}} = 0 , \label{init-1}  \\[1ex]
&\ds (\nabla\cdot \ov\bu_f, w_f)_{\Omega_f} = 0, \label{init-2} \\[1ex]
&\ds (\mu \bK^{-1}\ov\bu_p,\bv_p)_{\Omega_p}
+\langle \bv_p\cdot \bn_p, \ov\lambda \rangle_{\Gamma_{fp}} = 0, \label{init-6} \\[1ex]
&\ds \langle \ov\bu_f \cdot \bn_f  + \ov\btheta\cdot \bn_p  + \ov\bu_p\cdot \bn_p ,\xi \rangle_{\Gamma_{fp}} =0,
\label{init-8} \\[1ex]
&\ds \langle \bphi\cdot \bn_p, \ov\lambda \rangle_{\Gamma_{fp}}
-\sum_{j=1}^{n-1}{\langle \mu \alpha_{\BJS}\sqrt{\bK_j^{-1}}(\ov\bu_f-\ov\btheta)
  \cdot \bt_{f,j}, \bphi \cdot \bt_{f,j}\rangle_{\Gamma_{fp}}} = 0. \label{init-9}
\end{align}
\end{subequations}
Taking $(\bv_f,w_f,\bv_p,\xi,\bphi) = (\ov\bu_f,\ov
p_f,\ov\bu_p,\ov\lambda,\ov\btheta)$ and combining the equations results in
$$
\|\ov\bu_f\|_{\bH^1(\Omega_f)}^2 + \|\ov\bu_p\|_{\L^2(\Omega_p)}^2 + |\ov\bu_f - \ov\btheta|_{a_{\BJS}}^2 \le 0,
$$
which implies $\ov\bu_f = \0$,  $\ov\bu_p = \0$ and $\ov\btheta \cdot \bt_{f,j}=0$. Then \eqref{init-8} implies that $\langle \ov\btheta\cdot \bn_p  ,\xi \rangle_{\Gamma_{fp}} =0$ for all $\xi \in \H^{1/2}(\Gamma_{fp})$. We note that $\bn_p$ may be discontinuous on
$\Gamma_{fp}$, resulting in $\ov\btheta\cdot \bn_p \in \L^2(\Gamma_{fp})$. However,
since $\H^{1/2}(\Gamma_{fp})$ is dense in $\L^2(\Gamma_{fp})$, we obtain $\ov\btheta\cdot \bn_p=0$, thus $\ov\btheta = \0$.
Using the inf-sup condition
\eqref{eq: continuous inf-sup 3}, together with \eqref{init-1} and \eqref{init-6},
we conclude that $\ov p_f = 0$ and $\ov\lambda = 0$.
\end{proof}

\begin{remark}\label{rem:non-diff-eq}
As we noted in Remark~\ref{rem:time-diff}, the time differentiated
equation \eqref{weak-form-3} can be used to recover the
non-differentiated equation \eqref{non-diff-eq}. In particular,
recalling the initial data construction \eqref{sigma_p0}, let
$$
\forall \, t \in [0,T], \quad  \bbeta_p(t) = \bbeta_{p,0} + \int_0^t \bu_s(s) \, ds,
\quad \brho_p(t) = \brho_{p,0} + \int_0^t \bgamma_p(s) \, ds, \quad
\bpsi(t) = \bpsi_0 + \int_0^t \btheta(s) \, ds.
$$
Then \eqref{non-diff-eq} follows from integrating \eqref{weak-form-3} from $0$ to 
$t \in (0,T]$ and using the first equation in \eqref{sigma_p0}.
\end{remark}

%%%%%%%%%%%%%%%%%%%%%%%%%%%%%%%%%%%%%%%%%%%%%%%%%%%%%%%%%%%%%%
%%%%%%%%%%%%%%%%%%%%%%%%%%%%%%%%%%%%%%%%%%%%%%%%%%%%%%%%%%%%%%
%%%%%%%%%%%%%%%%%%%%%%%%%%%%%%%%%%%%%%%%%%%%%%%%%%%%%%%%%%%%%%
\section{Semi-discrete formulation}\label{sec:semi-discrete}
  
In this section we introduce the semi-discrete continuous-in-time
approximation of \eqref{eq: continuous formulation 3}. We assume for simplicity that
$\Omega_f$ and $\Omega_p$ are polygonal domains. Let
$\mathcal{T}_{h_f}^f$ and $\mathcal{T}_{h_p}^p$ be shape-regular \cite{ciarlet1978}
affine finite element partitions of $\Omega_f$ and $\Omega_p$,
respectively, which may be non-matching along the
interface $\Gamma_{fp}$. Here $h_f$ and $h_p$ are the maximum element diameters in $\Omega_f$ and $\Omega_p$, respectively.
Let $(\bV_{fh}, \W_{fh}) \subset (\bV_{f}, \W_{f})$ be any stable
Stokes finite element pair, such as Taylor-Hood or the MINI elements
\cite{brezzi1991mixed}, and let $(\bV_{ph}, \W_{ph}) \subset (\bV_{p}, \W_{p})$ be any
stable Darcy mixed finite element pair, such as the Raviart-Thomas (RT) or the
Brezzi-Douglas-Marini (BDM) elements \cite{brezzi1991mixed}. Let
$(\bbX_{ph}, \bV_{sh}, \bbQ_{ph}) \subset (\bbX_{p}, \bV_{s}, \bbQ_{p})$ by any stable
finite element triple for mixed elasticity with weak stress symmetry, such as
the spaces developed in \cite{arnold2015,arnold2007mixed,brezzi2008mixed}. We note that these spaces satisfy
\begin{equation}\label{div-prop}
  \nabla\cdot\bV_{ph} = \W_{ph}, \quad \nabla\cdot \bbX_{ph} = \bV_{sh}.
\end{equation}
For the Lagrange multipliers, we choose
non-conforming approximations:
\begin{equation}\label{Lagr-mult}
\Lambda_{ph}  :=  \bV_{ph} \cdot \bn_p \, |_{\Gamma_{fp}}, \quad 
\bLambda_{sh}  :=  \bbX_{ph} \bn_p \, |_{\Gamma_{fp}} \quad \text{with norms} \quad
\|\xi\|_{\Lambda_{ph}}  := \|\xi\|_{\L^2(\Gamma_{fp})}, \quad
\|\bphi\|_{\bLambda_{sh}}  := \|\bphi\|_{\bL^2(\Gamma_{fp})}.
\end{equation}

The semi-discrete continuous-in-time problem is: Given
$\f_f: [0,T] \rightarrow \bV_f'$,
$\f_p: [0,T] \rightarrow \bV_s'$,
$q_f: [0,T] \rightarrow \W_f'$, 
$q_p: [0,T] \rightarrow \W_p'$,
and $(\bsi_{ph,0},p_{ph,0}) \in \bbX_{ph} \times \W_{ph}$, find $(\bu_\fh, p_\fh,
\bsi_\ph, \bu_\sh, \bgamma_\ph,
\bu_\ph, p_\ph,
\lambda_h, \linebreak \btheta_h): [0,T] \rightarrow \bV_\fh \times \W_\fh
\times \bbX_\ph \times \bV_\sh \times \bbQ_\ph
\times \bV_\ph \times \W_\ph
\times \Lambda_\ph \times \bLambda_\sh $ such that
$(\bsi_\ph(0),p_\ph(0)) = (\bsi_{ph,0},p_{ph,0})$ and,
for a.e. $t\in (0,T)$ and for all $\bv_\fh \in \bV_\fh$, $w_\fh \in \W_\fh$,
$\btau_\ph \in \bbX_\ph$, $\bv_\sh \in \bV_\sh$, $\bchi_\ph \in \bbQ_\ph$,
$\bv_\ph \in \bV_\ph$, $w_\ph \in \W_\ph$,
$\xi_h \in \Lambda_\ph$, and $\bphi_h \in \bLambda_\sh$,
\begin{subequations}\label{eq: semidiscrete formulation 1}
\begin{align}
&\ds (2\mu \bD(\bu_\fh),\bD(\bv_\fh))_{\Omega_f}
-(\nabla\cdot \bv_\fh, p_\fh)_{\Omega_f} 
+\langle \bv_\fh\cdot \bn_f,  \lambda_h \rangle_{\Gamma_{fp}}  \nonumber \\[1ex]
& \qquad
+\sum_{j=1}^{n-1}{\langle \mu \alpha_{\BJS}\sqrt{\bK_j^{-1}}(\bu_\fh-\btheta_h)\cdot \bt_{f,j},
  \bv_\fh \cdot \bt_{f,j}\rangle_{\Gamma_{fp}}}
=(\f_f,\bv_\fh)_{\Omega_f}, \label{sd-1}  \\[1ex]
&\ds (\nabla\cdot \bu_\fh, w_\fh)_{\Omega_f}
=(q_f, w_\fh)_{\Omega_f}, \label{sd-2} \\[1ex]
&\ds (\partial_t A(\bsi_\ph+\alpha  p_\ph\bI),\btau_\ph)_{\Omega_p}
+(\nabla \cdot\btau_\ph,\bu_\sh)_{\Omega_p}
+(\btau_\ph,\bgamma_\ph)_{\Omega_p}
-\langle \btau_\ph \bn_p, \btheta_h \rangle_{\Gamma_{fp}}
=0, \label{sd-3}  \\[1ex]
&\ds (\nabla\cdot \bsi_\ph, \bv_\sh)_{\Omega_p} = -(\f_p, \bv_\sh)_{\Omega_p}, \label{sd-4}  \\[1ex]
&\ds (\bsi_\ph,\bchi_\ph)_{\Omega_p}  =0, \label{sd-5} \\[1ex]
&\ds (\mu \bK^{-1}\bu_\ph,\bv_\ph)_{\Omega_p}
-(\nabla \cdot \bv_\ph,p_\ph)_{\Omega_p}
+\langle \bv_\ph\cdot \bn_p, \lambda_h \rangle_{\Gamma_{fp}}
=0, \label{sd-6}  \\[1ex]
&\ds (s_0 \partial_t p_\ph, w_\ph)_{\Omega_p}
+\alpha( \partial_t A(\bsi_\ph+\alpha p_\ph \bI), w_\ph \bI)_{\Omega_p}
+(\nabla \cdot \bu_\ph, w_\ph)_{\Omega_p}
=(q_p, w_\ph)_{\Omega_p}, \label{sd-7}  \\[1ex]
&\ds \langle \bu_\fh \cdot \bn_f  + \btheta_h\cdot \bn_p
+ \bu_\ph\cdot \bn_p ,\xi_h \rangle_{\Gamma_{fp}} =0, \label{sd-8} \\[1ex]
&\ds \langle \bphi_h\cdot \bn_p, \lambda_h \rangle_{\Gamma_{fp}}
-\sum_{j=1}^{n-1}{\langle \mu \alpha_{\BJS}\sqrt{\bK_j^{-1}}(\bu_\fh-\btheta_h)
  \cdot \bt_{f,j}, \bphi_h \cdot \bt_{f,j}\rangle_{\Gamma_{fp}}}
+\langle \bsi_\ph \bn_p, \bphi_h \rangle_{\Gamma_{fp}} =0. \label{sd-9} 
\end{align}
\end{subequations}

\begin{remark}\label{rem: sd eq hold}
  We note that, since $\H^{1/2}(\Gamma_{fp})$ is dense in $\L^2(\Gamma_{fp})$,
  the continuous variational equations 
  \eqref{weak-form-8} and \eqref{weak-form-9} hold for test functions in
  $\L^2(\Gamma_{fp})$, assuming that the solution is smooth enough.
  In particular, they hold for $\xi_h \in \Lambda_{ph}$ and
  $\bphi_h \in \bLambda_{sh}$, respectively. 
\end{remark}

The formulation \eqref{eq: semidiscrete formulation 1} can be equivalently
written as
\begin{equation}\label{eq: semidiscrete formulation 2}
\begin{array}{l}  
\ds  a_f(\bu_{fh},\bv_{fh})
+a_p(\bu_{ph},\bv_{ph})
+a_{\BJS}(\bu_{fh},\btheta_h;\bv_{fh},\bphi_h)
+b_n^p(\bsi_{ph}, \bphi_h)
+b_p(\bv_{ph}, p_{ph})
\\[1ex]
\ds \quad
+b_f(\bv_{fh},p_{fh})
+b_s(\btau_{ph},\bu_{sh})
+b_{\sk}(\btau_{ph}, \bgamma_{ph})
+b_{\Gamma}(\bv_{fh},\bv_{ph},\bphi_h;\lambda_h)
+a_p^p(\partial_t p_{ph}, w_{ph})
\\[1ex]
\ds \quad
+a_e(\partial_t \bsi_{ph}, \partial_t p_{ph}; \btau_{ph}, w_{ph})
-b_n^p(\btau_{ph}, \btheta_h)
-b_p(\bu_{ph}, w_{ph}) 
= (\f_f,\bv_{fh})_{\Omega_f}
+(q_p,w_{ph})_{\Omega_p}, \\[1ex]
\ds -b_f(\bu_{fh}, w_{fh})
-b_s(\bsi_{ph},\bv_{sh})
-b_{\sk}(\bsi_{ph},\bchi_{ph})
-b_{\Gamma}(\bu_{fh}, \bu_{ph}, \btheta_h ;\xi_h)
=(q_f,w_{fh})_{\Omega_f}
+(\f_p,\bv_{sh})_{\Omega_p}.
\end{array}
\end{equation}
We group the spaces and test functions as in the continuous case:
\begin{gather*}
\ds \bQ_h := \bV_{fh} \times \bLambda_{sh} \times \bV_{ph} \times \bbX_{ph} \times \W_{ph},\quad
\ds \bS_h := \W_{fh} \times \bV_{sh} \times \bbQ_{ph} \times \Lambda_{ph}, \\[1ex]
\bp_h := (\bu_{fh}, \btheta_h, \bu_{ph}, \bsi_{ph}, p_{ph})\in \bQ_h,\quad 
\br_h := (p_{fh}, \bu_{sh}, \bgamma_{ph}, \lambda_h) \in \bS_h,\\[1ex]
\bq_h := (\bv_{fh}, \bphi_h, \bv_{ph}, \btau_{ph}, w_{ph})\in \bQ_h,\quad 
\bs_h := (w_{fh}, \bv_{sh}, \bchi_{ph}, \xi_h)\in \bS_h,
\end{gather*}
where the spaces $\bQ_h$ and $\bS_h$ are endowed with the norms, respectively,
\begin{gather*}
\|\bq_h\|_{\bQ_h}  =  \|\bv_{fh}\|_{\bV_{f}}
+\|\bphi_h\|_{\bLambda_{sh}}+ \|\bv_{ph}\|_{\bV_{p}}+\|\btau_{ph}\|_{\bbX_{p}}+\|w_{ph}\|_{\W_{p}},\\[1ex]
\|\bs_h\|_{\bS_h}  = \|w_{fh}\|_{\W_{f}}+\|\bv_{sh}\|_{\bV_{s}}+\|\bchi_{ph}\|_{\bbQ_{p}}
+\|\xi_h\|_{\Lambda_{ph}}. 
\end{gather*}
Hence, we can write \eqref{eq: semidiscrete formulation 2} in an
operator notation as a degenerate evolution problem in a mixed form:
\begin{equation}\label{eq: semidiscrete formulation 3}
    \arraycolsep=1.7pt
  \begin{array}{rcl}  
\ds \partial_t\,\cE_1\,\bp_h(t)
+ \cA\,\bp_h(t) + \cB'\,\br_h(t) & = & \bF(t) \qin \bQ_h', \\[1ex]
\ds - \cB\,\bp_h(t) & = & \bG(t) \qin  \bS_h'.
  \end{array}
\end{equation}

Next, we state the discrete inf-sup conditions. 

\begin{lemma}\label{lem: discrete inf-sup}
There exist positive constants $\beta_{h,1}$, $\beta_{h,2}$, and $\beta_{h,3}$ independent of $h_f$ and $h_p$ such that
\begin{align}
& \ds \beta_{h,1} (\|\bv_{sh}\|_{\bV_{s}}+\|\bchi_{ph}\|_{\bbQ_{p}})\leq 
\underset{\btau_{ph} \in \bbX_{ph} \, \text{s.t. } \btau_{ph}\bn_p=\0 \text{ on } \Gamma_{fp}}{\sup}
\frac{b_s(\btau_{ph},\bv_{sh})+b_{\sk}(\btau_{ph},\bchi_{ph})}{\|\btau_{ph}\|_{\bbX_{p}}}, \nonumber \\[1ex]
& \qquad\qquad\qquad\qquad \forall \, \bv_{sh} \in \bV_{sh}, \, \bchi_{ph} \in \bbQ_{ph},
\label{eq: discrete inf-sup 1} \\[1ex]
& \ds \beta_{h,2}(\|w_{fh}\|_{\W_{f}} + \|w_{ph}\|_{\W_{p}} + \|\xi_h\|_{\Lambda_{ph}} ) \hspace{6cm}  
\nonumber \\[1ex]
& \ds \qquad \leq  
\underset{(\bv_{fh}, \bv_{ph}) \in \bV_{fh}\times\bV_{ph}}{\sup}
\frac{b_f(\bv_{fh},w_{fh})+b_p(\bv_{ph},w_{ph})+
b_{\Gamma}(\bv_{fh}, \bv_{ph}, \0; \xi_h)}{\|(\bv_{fh}, \bv_{ph})\|_{\bV_{f}\times\bV_{p}}}, \nonumber \\[1ex]
& \qquad\qquad\qquad\qquad
\forall \, w_{fh}\in \W_{fh}, w_{ph}\in \W_{ph}, \xi_h \in \Lambda_{ph},
\label{eq: discrete inf-sup 2} \\[1ex]
& \ds \beta_{h,3}\|\bphi_h\|_{\bLambda_{sh}} \leq   \underset{\btau_{ph} \in \bbX_{ph} \,
  {\rm s.t.}\, \nabla \cdot \btau_{ph}=\0 }{\sup} \frac{b_n^p(\btau_{ph}, \bphi_p)}
{\|\btau_{ph}\|_{\bbX_{p}}}, \qquad \forall \, \bphi_h \in \bLambda_{sh}. \label{eq: discrete inf-sup 3}
\end{align}
\end{lemma}

\begin{proof}
Inequality \eqref{eq: discrete inf-sup 1} can be shown using the
argument in \cite[Theorem 4.1]{msmfe-simpl}. Inequality \eqref{eq: discrete inf-sup 2} is proved in \cite[Theorem 5.2]{aeny2019}.
Inequality \eqref{eq: discrete inf-sup 3} can be derived as
in \cite[Lemma 5.1]{aeny2019}.
\end{proof}

We next discuss the construction of compatible discrete initial data
$(\bp_{h,0}, \br_{h,0})$ based on a modification of the step-by-step
procedure for the continuous initial data.

1. Let
$P_h^{\bLambda_s}: \bLambda_s \rightarrow \bLambda_{sh}$ 
be the $\L^2$-projection operator, satisfying, for all $\bphi \in \bL^2(\Gamma_{fp})$,
\begin{align}\label{eq: interpolation 0}
&\ds
\langle \bphi - P_h^{\bLambda_s} \bphi, \bphi_h \rangle_{\Gamma_{fp}} = 0 \qquad \forall \, \bphi_h \in \bLambda_{sh}. 
\end{align}
Define 
\begin{equation}\label{eq: dis ini theta}
\btheta_{h,0}=P_h^{\bLambda_s} \, \btheta_0.
\end{equation}

2. Define $(\bu_{fh,0}, p_{fh,0}) \in \bV_{fh} \times \W_{fh}$ and $(\bu_{ph,0}, p_{ph,0},\lambda_{h,0})\in \bV_{ph}\times \W_{ph}\times \Lambda_{ph}$ by solving a coupled Stokes-Darcy problem:
for all $\bv_{fh} \in \bV_{fh}$, $w_{fh} \in \W_{fh}$, $\bv_{ph} \in \bV_{ph}$,
$w_{ph} \in \W_{ph}$,
$\xi_h \in \Lambda_{\ph}$,
\begin{align}\label{eq: dis ini S-D}
& \ds a_f(\bu_{fh,0}, \bv_{fh}) + b_f(\bv_{fh},p_{fh,0}) 
+ \sum_{j=1}^{n-1}\langle \mu \alpha_{\BJS} \sqrt{\bK_j^{-1}} (\bu_{fh,0}-\btheta_{h,0})
\cdot \bt_{f,j}, \bv_{fh} \cdot \bt_{f,j} \rangle_{\Gamma_{fp}}
+ \langle \bv_{fh} \cdot \bn_f, \lambda_{h,0} \rangle_{\Gamma_{fp}} \nonumber\\
& \ds  \quad = a_f(\bu_{f,0}, \bv_{fh}) + b_f(\bv_{fh},p_{f,0}) 
+ \sum_{j=1}^{n-1}\langle \mu \alpha_{\BJS} \sqrt{\bK_j^{-1}} (\bu_{f,0}-\btheta_{0})
\cdot \bt_{f,j}, \bv_{fh} \cdot \bt_{f,j} \rangle_{\Gamma_{fp}}
+ \langle \bv_{fh} \cdot \bn_f, \lambda_{0} \rangle_{\Gamma_{fp}} \nonumber \\
& \ds \quad = (\f_f(0),\bv_{fh})_{\Omega_f}, \nonumber\\[1ex]
& \ds -b_f(\bu_{fh,0}, w_{fh}) = -b_f(\bu_{f,0}, w_{fh}) = (q_f(0),w_{fh}), \nonumber\\[1ex]
& \ds a_p(\bu_{ph,0}, \bv_{ph}) + b_p(\bv_{ph}, p_{ph,0})
+ \langle \bv_{ph} \cdot \bn_p, \lambda_{h,0} \rangle_{\Gamma_{fp}}
= a_p(\bu_{p,0}, \bv_{ph}) + b_p(\bv_{ph}, p_{p,0})
+ \langle \bv_{ph} \cdot \bn_p, \lambda_0 \rangle_{\Gamma_{fp}} = 0, \nonumber\\[1ex]
 & \ds -b_p(\bu_{ph,0}, w_{ph}) = -b_p(\bu_{p,0}, w_{ph})
 = -\mu^{-1}(\nabla \cdot (\bK \nabla p_{p,0}), w_{ph})_{\Omega_p}, \nonumber\\[1ex]
& -\langle \bu_{ph,0}\cdot \bn_p + \bu_{fh,0} \cdot \bn_f
 + \btheta_{h,0} \cdot \bn_p, \xi_h \rangle_{\Gamma_{fp}}
 = -\langle \bu_{p,0}\cdot \bn_p + \bu_{f,0} \cdot \bn_f
 + \btheta_{0} \cdot \bn_p, \xi_h \rangle_{\Gamma_{fp}} = 0.
  \end{align}
This is a well-posed problem due to the inf-sup condition \eqref{eq:
  discrete inf-sup 3}, using the theory of saddle point problems
\cite{brezzi1991mixed}, see \cite{lsy2003,ervin2009}.

3. Define $\ds (\bsi_{ph,0}, \bbeta_{ph,0}, \brho_{ph,0}, \bpsi_{h,0}) \in \bbX_{ph}
\times \bV_{sh} \times \bbQ_{ph} \times \bLambda_{sh} $ such that, for all
$\btau_{ph} \in \bbX_{ph}$, $\bv_{sh} \in \bV_{sh}$, $\bchi_{ph} \in \bbQ_{ph}$,
$\bphi_h \in \bLambda_{sh}$, 
\begin{align}\label{eq: dis ini sigma}
& \ds (A \bsi_{ph,0}, \btau_{ph})_{\Omega_p}
+ b_s( \btau_{ph}, \bbeta_{ph,0})
+ b_{\sk}(\btau_{ph}, \brho_{ph,0})
- b_n^p(\btau_{ph}, \bpsi_{h,0})
+ (A \alpha p_{ph,0} \bI,\btau_{ph})_{\Omega_p} \nonumber\\[1ex] 
& \qquad =(A \bsi_{p,0}, \btau_{ph})_{\Omega_p}
+ b_s( \btau_{ph}, \bbeta_{p,0})
+ b_{\sk}(\btau_{ph}, \brho_{p,0})
- b_n^p(\btau_{ph}, \bpsi_0)
+ (A \alpha p_{p,0} \bI,\btau_{ph})_{\Omega_p} = 0, \nonumber \\[1ex] 
&\ds  - b_s ( \bsi_{ph,0}, \bv_{sh} ) = - b_s ( \bsi_{p,0}, \bv_{sh} )
= (\f_p(0),\bv_{sh})_{\Omega_p}, \nonumber \\[1ex] 
&\ds  - b_{\sk} ( \bsi_{ph,0}, \bchi_{ph} ) = - b_{\sk} ( \bsi_{p,0}, \bchi_{ph} ) = 0, \nonumber \\
&   \ds b_n^p(\bsi_{ph,0}, \bphi_h)
- \sum_{j=1}^{n-1}\langle \mu \alpha_{\BJS} \sqrt{\bK_j^{-1}} (\bu_{fh,0}-\btheta_{h,0})
\cdot \bt_{f,j}, \bphi_h \cdot \bt_{f,j} \rangle_{\Gamma_{fp}}
+ \langle \bphi_h \cdot \bn_p, \lambda_{h,0} \rangle_{\Gamma_{fp}} \nonumber \\
& \ds \qquad = b_n^p(\bsi_{p,0}, \bphi_h)
- \sum_{j=1}^{n-1}\langle \mu \alpha_{\BJS} \sqrt{\bK_j^{-1}} (\bu_{f,0}-\btheta_{0})
\cdot \bt_{f,j}, \bphi_h \cdot \bt_{f,j} \rangle_{\Gamma_{fp}}
+ \langle \bphi_h \cdot \bn_p, \lambda_0 \rangle_{\Gamma_{fp}} = 0. 
\end{align}
It can be shown that the above problem is well-posed using the finite
element theory for elasticity with weak stress symmetry \cite{arnold2015,arnold2007mixed} and
the inf-sup condition \eqref{eq: discrete inf-sup 3} for the Lagrange multiplier
$\bpsi_{h,0}$.

4. Define $(\wh{\bsi}_{ph,0}, \bu_{sh,0}, \bgamma_{ph,0}) \in \bbX_{ph} \times \bV_{sh} \times \bbQ_{ph}$
such that, for all $\btau_{ph} \in \bbX_{ph}$, $\bv_{sh} \in \bV_{sh}$, $\bchi_{ph} \in \bbQ_{ph}$,
\begin{align}\label{eq: dis ini us}
& \ds (A \wh\bsi_{ph,0}, \btau_{ph})_{\Omega_p}  + b_s( \btau_{ph}, \bu_{sh,0})
  + b_{\sk}(\btau_{ph}, \bgamma_{ph,0}) = b_n^p(\btau_{ph}, \btheta_{h,0}), \nonumber \\[1ex]
& \ds  -b_s ( \wh\bsi_{ph,0}, \bv_{sh} ) = 0, \nonumber \\[1ex]
& \ds  - b_{\sk} ( \wh\bsi_{ph,0}, \bchi_{ph} ) = 0.
\end{align}
This is a well posed discrete mixed elasticity problem \cite{arnold2015,arnold2007mixed}.

We then define $\bp_{h,0} = (\bu_{fh,0},\btheta_{h,0},\bu_{ph,0},\bsi_{ph,0},p_{p,0})$
and $\br_{h,0} = (p_{fh,0},\bu_{sh,0},\bgamma_{ph,0},\lambda_{h,0})$. This construction
guarantees that the discrete initial data is compatible in 
the sense of Lemma \ref{lem: initial condition}:
\begin{equation}\label{eq: discrete initial condition}
  \arraycolsep=1.7pt
  \begin{array}{rcl}    
\cA \bp_{h,0} + \cB'\br_{h,0} & = & \overline{\bF}_0  \qin  \bQ_h', \\[1ex]
    - \cB \bp_{h,0} & = & \bG(0) \qin \bS_h',
  \end{array}
\end{equation}
where $\overline{\bF}_0  = (\f_f(0),0,0,\overline{g}_{\btau_p},\overline{g}_{w_p})^\rt \in \bQ_{2}'$, 
with suitable $\overline{g}_{\btau_p} \in \bbX_{p,2}'$ and $\overline{g}_{w_p} \in \W_{p,2}'$.
Furthermore, it provides compatible initial
data for the non-differentiated elasticity variables
$(\bbeta_{ph,0},\brho_{ph,0},\bpsi_{h,0})$ in the sense of the first equation in \eqref{sigma_p0}.

The well-posedness of the problem \eqref{eq: semidiscrete formulation
  3} follows from similar arguments to the proof of Theorem \ref{thm:
  well-posedness continuous}.

\begin{theorem}\label{thm:well-posed-discr}
For each $\f_f \in \W^{1,1}(0,T;\bV_f')$, $\f_p\in
\W^{1,1}(0,T;\bV_s')$, $q_f\in \W^{1,1}(0,T;\W_f')$, and $q_p\in
\W^{1,1}(0,T; \W_p')$, and initial data $(\bp_{h,0}, \br_{h,0})$
satisfying \eqref{eq: discrete initial condition}, there exists a
unique solution of \eqref{eq: semidiscrete formulation 1} $(\bu_\fh,
p_\fh,\bsi_\ph, \bu_\sh, \bgamma_\ph,  \bu_\ph, p_\ph,\lambda_h, \btheta_h)\in \L^\infty(0,T;\bV_\fh)\times \L^\infty(0,T;\W_\fh)\times
\W^{1,\infty}(0,T; \bbL^2(\Omega_p))\cap \L^{\infty} \linebreak  (0,T; 
\bbX_\ph)\times \L^\infty(0,T;\bV_\sh)\times
\L^\infty(0,T; \bbQ_\ph)\times \L^\infty(0,T;\bV_\ph)\times
\W^{1,\infty}(0,T;\W_\ph) \times \L^\infty(0,T;\Lambda_\ph)\times
\L^\infty(0,T; \linebreak  \bLambda_\sh)$ with $(\bu_\fh(0), p_\fh(0),\bsi_\ph(0), 
\bu_\ph(0), p_\ph(0),\lambda_h(0), \btheta_h(0)) =
(\bu_{\fh,0}, p_{\fh,0},\bsi_{\ph,0}, \bu_{\ph,0}, p_{\ph,0}, \lambda_{h,0}, \linebreak  \btheta_{h,0})$.
\end{theorem}
\begin{proof}
  With the discrete inf-sup conditions
  \eqref{eq: discrete inf-sup 1}--\eqref{eq: discrete inf-sup 3}
  and the discrete initial data construction described in
  \eqref{eq: interpolation 0}--\eqref{eq: dis ini sigma},
the proof is similar to the
proofs of Theorem~\ref{thm: well-posedness continuous} and
Corollary~\ref{cor:init-data}, with two differences due to
non-conforming choices of the Lagrange multiplier spaces equipped with
$\L^2$-norms. The first is in the continuity of the bilinear forms
$b_n^p(\btau_{ph}, \bphi_h)$, cf. \eqref{eq: continuous continuity 0},
and $b_{\Gamma}(\bv_{fh},\bv_{ph},\bphi_h;\xi_h)$, cf. \eqref{eq: continuous continuity 3}.
In particular, using the discrete trace-inverse inequality for piecewise polynomial
functions, $\|\varphi\|_{L^2(\Gamma_{fp})} \le C h_{p,\min}^{-1/2}\|\varphi\|_{L^2(\Omega_p)}$, where $h_{p,\min}$ is the minimum element diameter in $\mathcal{T}_{h_p}^p$, we have
$$
b_n^p(\btau_{ph}, \bphi_h) \le C h_{p,\min}^{-1/2}\|\btau_{ph}\|_{\bbL^2(\Omega_p)}\|\bphi_h\|_{\bL^2(\Gamma_{fp})}
$$
and
$$
b_{\Gamma}(\bv_{fh},\bv_{ph},\bphi_h;\xi_h)
\le C (\|\bv_{fh}\|_{\bH^1(\Omega_f)} + h_{p,\min}^{-1/2}\|\bv_{ph}\|_{\bL^2(\Omega_p)}
+ \|\bphi_h\|_{\bL^2(\Gamma_{fp})})\|\xi_h\|_{\L^2(\Gamma_{fp})}.
$$
Therefore these bilinear forms are continuous for any given mesh. Second,
the operators $L_{\lambda}$ and
$R_{\btheta}$ from Lemma \ref{lem: R-operators 2} are now defined as
$L_{\lambda}: \Lambda_{ph} \rightarrow \Lambda_{ph}', \ (L_{\lambda}\,
\lambda_h, \xi_h):=\langle \lambda_h, \xi_h \rangle_{\Gamma_{fp}}$ and
$R_{\btheta}: \bLambda_{sh} \rightarrow \bLambda_{sh}',
\ (R_{\btheta}\,\btheta_h, \bphi_h):=\langle \btheta_h, \bphi_h
\rangle_{\Gamma_{fp}}$. The fact that $L_{\lambda}$ and
$R_{\btheta}$ are continuous and coercive follows
immediately from their definitions, since $(L_{\lambda}\, \xi_h,
\xi_h)=\Vert \xi \Vert^2_{\Lambda_{ph}}$ and $(R_{\btheta}\,\bphi_h,
\bphi_h)=\Vert \bphi_h \Vert^2_{\bLambda_{sh}}$.
We note that the proof of Corollary~\ref{cor:init-data}
works in the discrete case due to the choice of the discrete initial data
as the elliptic projection of the continuous initial data, cf. 
\eqref{eq: dis ini S-D} and \eqref{eq: dis ini sigma}.
\end{proof}

\begin{remark}
As in the continuous case, we can recover the non-differentiated elasticity variables with
$$
\forall \, t \in [0,T], \quad  \bbeta_\ph(t) = \bbeta_{ph,0} + \int_0^t \bu_\sh(s) \, ds,
\ \  \brho_\ph(t) = \brho_{ph,0} + \int_0^t \bgamma_\ph(s) \, ds, \ \ 
\bpsi_h(t) = \bpsi_{h,0} + \int_0^t \btheta_h(s) \, ds.
$$
Then \eqref{non-diff-eq} holds discretely, which follows from
integrating the third equation in \eqref{eq: semidiscrete formulation 1} from $0$ to 
$t \in (0,T]$ and using the discrete version of the
first equation in \eqref{sigma_p0}.
\end{remark}

%%%%%%%%%%%%%%%%%%%%%%%%%%%%%%%%%%%%%%%%%%%%%%%%%%%%%%%%%%%%%%
%%%%%%%%%%%%%%%%%%%%%%%%%%%%%%%%%%%%%%%%%%%%%%%%%%%%%%%%%%%%%%
%%%%%%%%%%%%%%%%%%%%%%%%%%%%%%%%%%%%%%%%%%%%%%%%%%%%%%%%%%%%%%
\section{Stability analysis}\label{sec:stability}
In this section we establish a stability bound for the solution of
semi-discrete continuous-in-time formulation \eqref{eq: semidiscrete
  formulation 3}. We emphasize that the stability constant is independent
of $s_0$ and $a_{\min}$, indicating robustness of the method in the limits of
small storativity and almost incompressible media, which are known to
cause locking in numerical methods for the Biot system \cite{ Yi-Biot-locking}.
Furthermore, since we do not utilize Gronwall's inequality,
we obtain long-time stability for our method.

\begin{theorem}\label{thm: stability analysis}
Assuming sufficient regularity of the data, for the solution to the semi-discrete problem \eqref{eq: semidiscrete formulation 1}, there exists a constant $C$ independent of $h_f$, $h_p$, $s_0$ and $a_{\min}$ such that
\begin{align}\label{eq: stability analysis}
& \ds 
\Vert \bu_{fh} \Vert_{\L^{\infty}(0,T;\bV_f)}
+\Vert \bu_{fh} \Vert_{\L^2(0,T;\bV_f)}
+\vert \bu_{fh}-\btheta_h \vert_{\L^{\infty}(0,T;a_{\BJS})}
+\vert \bu_{fh}-\btheta_h \vert_{\L^2(0,T;a_{\BJS})}
\nonumber \\[1ex]
&\ds \quad 
+\Vert p_{fh} \Vert_{\L^{\infty}(0,T;\W_f)}
+\Vert p_{fh}\Vert_{\L^2(0,T;\W_f)}
+\Vert A^{1/2}\bsi_{ph}\Vert_{\L^{\infty}(0,T;\bbL^2(\Omega_p))}
+\Vert \nabla \cdot \bsi_{ph}\Vert_{\L^{\infty}(0,T;\bL^2(\Omega_p))}
\nonumber \\[1ex]
&\ds \quad 
+\Vert A^{1/2} \partial_t (\bsi_{ph}+\alpha p_{ph}\bI) \Vert_{\L^2(0,T;\bbL^2(\Omega_p))}
+\Vert \nabla \cdot \bsi_{ph}\Vert_{\L^2(0,T;\bL^2(\Omega_p))}
+\Vert \bu_{sh}\Vert_{\L^2(0,T;\bV_s)}
\nonumber \\[1ex]
&\ds \quad 
+\Vert \bgamma_{ph}\Vert_{\L^2(0,T;\bbQ_p)}
+\Vert \bu_{ph} \Vert_{\L^{\infty}(0,T;\bL^2(\Omega_p))}
+\Vert \bu_{ph} \Vert_{\L^2(0,T;\bV_p)}
+\Vert p_{ph} \Vert_{\L^{\infty}(0,T;\W_p)}
+\Vert p_{ph} \Vert_{\L^2(0,T;\W_p)}
\nonumber \\[1ex]
&\ds \quad
+\sqrt{s_0}\Vert \partial_t p_{ph} \Vert_{\L^2(0,T;\W_p)}
+\Vert\lambda_h\Vert_{\L^{\infty}(0,T;\Lambda_{ph})}
+\Vert\lambda_h\Vert_{\L^2(0,T;\Lambda_{ph})}
+\Vert\btheta_h\Vert_{\L^2(0,T;\bLambda_{sh})}
\nonumber \\[1ex]
&\ds \leq 
C\Big(
\Vert\f_f\Vert_{\H^1(0,T;\bL^2(\Omega_f))}
+\Vert\f_p\Vert_{\H^1(0,T;\bL^2(\Omega_p))}
+\Vert q_f\Vert_{\H^1(0,T;\L^2(\Omega_f))}
+\Vert q_p\Vert_{\H^1(0,T;\L^2(\Omega_p))}
\nonumber \\[1ex]
&\ds \qquad\quad
+ \|p_{p,0}\|_{\H^1(\Omega_p)} + \|\nabla \cdot (\bK\nabla p_{p,0})\|_{\L^2(\Omega_p)}\Big).
\end{align}
\end{theorem}
\begin{proof}
By taking 
$(\bv_{fh},w_{fh},\btau_{ph},\bv_{sh},\bchi_{ph},\bv_{ph},w_{ph},\xi_h,\bphi_h)=(\bu_{fh},p_{fh},\bsi_{ph},\bu_{sh},\bgamma_{ph},\bu_{ph},p_{ph},\lambda_h,\btheta_h)$ in \eqref{eq: semidiscrete formulation 1} and adding up all the equations, we get
\begin{align}\label{eq: stability 1}
&\ds a_f(\bu_{fh}, \bu_{fh})
+a_{\BJS}(\bu_{fh}, \btheta_h; \bu_{fh}, \btheta_h)
+a_e(\partial_t \bsi_{ph}, \partial_t p_{ph}; \bsi_{ph}, p_{ph})
+a_p(\bu_{ph},\bu_{ph})
+a_p^p(\partial_t p_{ph}, p_{ph})
\nonumber \\[1ex]
&\ds \quad =(\f_f,\bu_{fh})_{\Omega_f}+(q_f, p_{fh})_{\Omega_f}+(\f_p, \bu_{sh})_{\Omega_p}+(q_p,p_{ph})_{\Omega_p}.
\end{align}
Using the algebraic identity
$\int_S v \, \partial_t v = \frac{1}{2} \partial_t \Vert v \Vert^2_{\L^2(S)}$,
and employing the coercivity properties of $a_f$ and $a_p$,
and the semi-positive definiteness of $a_{\BJS}$, cf. \eqref{eq: continuous monotonicity 1},
we obtain
\begin{align*}
& \ds
2\mu C_K^2 \Vert \bu_{fh} \Vert ^2_{\bV_f}
+ \mu \alpha_{\BJS} k_{\max}^{-1/2}
\vert \bu_{fh}-\btheta_h \vert ^2_{a_\BJS}
+\frac{1}{2}\partial_t\Vert A^{1/2}(\bsi_{ph}+\alpha p_{ph}\bI)\Vert^2_{\bbL^2(\Omega_p)}
\nonumber \\[1ex]
&\ds
+\mu k_{\max}^{-1} \Vert \bu_{ph} \Vert ^2_{\bL^2(\Omega_p)}
+\frac{1}{2}s_0 \partial_t\Vert p_{ph} \Vert ^2_{\W_p} 
\leq (\f_f,\bu_{fh})_{\Omega_f}
+(q_f, p_{fh})_{\Omega_f}
+(\f_p, \bu_{sh})_{\Omega_p}
+(q_p,p_{ph})_{\Omega_p}.
\end{align*}
Integrating from $0$ to any $t\in (0,T]$ and applying the Cauchy-Schwarz
and Young's inequalities, we get
\begin{align}\label{eq: stability 3}
& \ds \int_{0}^{t}\big(
2\mu C_K^2 \Vert \bu_{fh} \Vert ^2_{\bV_f}
+\mu \alpha_{\BJS} k_{\max}^{-1/2}
\vert \bu_{fh}-\btheta_h \vert ^2_{a_\BJS}
+\mu k_{\max}^{-1} \Vert \bu_{ph} \Vert ^2_{\bL^2(\Omega_p)}\big)ds
 \nonumber \\[1ex]
&\ds 
+\frac{1}{2}\Vert A^{1/2}(\bsi_{ph}+\alpha p_{ph}\bI)(t)\Vert^2_{\bbL^2(\Omega_p)}
-\frac{1}{2}\Vert A^{1/2}(\bsi_{ph}+\alpha p_{ph}\bI)(0)\Vert^2_{\bbL^2(\Omega_p)}
+\frac{1}{2}s_0 \Vert p_{ph}(t) \Vert ^2_{\W_p}-\frac{1}{2}s_0 \Vert p_{ph}(0) \Vert ^2_{\W_p}  \nonumber \\[1ex]
&\ds \leq
\frac{\epsilon}{2}\int_{0}^{t}(\Vert \bu_{fh}\Vert^2_{\bL^2(\Omega_f)}
+\Vert p_{fh}\Vert^2_{\W_f}
+\Vert \bu_{sh}\Vert^2_{\bL^2(\Omega_p)}
+\Vert p_{ph}\Vert^2_{\W_p}) \,ds \nonumber \\[1ex]
&\ds \quad +\frac{1}{2\epsilon}\int_{0}^{t}(\Vert \f_f\Vert^2_{\bL^2(\Omega_f)}
+\Vert q_f\Vert^2_{\L^2(\Omega_f)}
+\Vert \f_p\Vert^2_{\bL^2(\Omega_p)}
+\Vert q_p\Vert^2_{\L^2(\Omega_p)}) \,ds. 
\end{align}
From the discrete inf-sup conditions
\eqref{eq: discrete inf-sup 1}--\eqref{eq: discrete inf-sup 3} and \eqref{sd-1},
\eqref{sd-3}, and \eqref{sd-6}, we have
\begin{align}\label{eq: stability 4}
& \ds \Vert p_{fh} \Vert_{\W_f} 
+\Vert p_{ph} \Vert_{\W_p} 
+\Vert \lambda_h \Vert_{\Lambda_{ph}} \nonumber \\[1ex]
& \ds \quad \leq C
\underset{(\bv_{fh}, \bv_{ph}) \in \bV_{fh}\times\bV_{ph}}{\sup}\frac
{b_f(\bv_{fh},p_{fh})+b_p(\bv_{ph},p_{ph})+b_{\Gamma}(\bv_{fh}, \bv_{ph}, \0; \lambda_h)}
{\|(\bv_{fh}, \bv_{ph})\|_{\bV_{f}\times\bV_{p}}}\nonumber \\[1ex]
&\ds \quad 
=C\underset{(\bv_{fh}, \bv_{ph}) \in \bV_{fh}\times\bV_{ph}}{\sup}
\frac
{-a_f(\bu_{fh}, \bv_{fh})
-a_{\BJS}(\bu_{fh},\btheta_h;\bv_{fh},\0)
+(\f_f,\bv_{fh})_{\Omega_f}-a_p(\bu_{ph}, \bv_{ph})}
{\|\bv_{fh}\|_{\bV_f}+\|\bv_{ph}\|_{\bV_p}}
\nonumber \\[1ex]
&\ds \quad
\leq C(\Vert \bu_{fh} \Vert_{\bV_f}
+ \vert \bu_{fh}-\btheta_h \vert_{a_{\BJS}}
+ \Vert \f_f \Vert_{\bL^2(\Omega_f)}
+ \Vert \bu_{ph} \Vert_{\bL^2(\Omega_p)}),
\end{align}
\begin{align}\label{eq: stability 4.5}
& \ds 
\|\bu_{sh}\|_{\bV_s}+\|\bgamma_{ph}\|_{\bbQ_p}
\leq C
\underset{\btau_{ph} \in \bbX_{ph} \, \text{s.t.} \, \btau_{ph}\bn_p=\0 \text{ on } \Gamma_{fp}}{\sup}
\frac{b_s(\btau_{ph},\bu_{sh})+b_{\sk}(\btau_{ph},\bgamma_{ph})}{\|\btau_{ph}\|_{\bbX_{p}}}\nonumber \\[1ex]
&\ds \quad
= C \underset{\btau_{ph} \in \bbX_{ph} \, \text{s.t.} \, \btau_{ph}\bn_p=\0 \text{ on } \Gamma_{fp}}{\sup}
\frac{
-(A\partial_t(\bsi_{ph}+\alpha p_{ph}\bI), \btau_{ph})
+b_n^p( \btau_{ph}, \btheta_h) }
{\|\btau_{ph}\|_{\bbX_p}} \nonumber \\[1ex]
&\ds \quad
\leq C \Vert A^{1/2} \partial_t (\bsi_{ph}+\alpha p_{ph} \bI) \Vert_{\bbL^2(\Omega_p)}, 
\end{align}
\begin{align}\label{eq: stability 5}
&\ds \|\btheta_h\|_{\bLambda_{sh}} 
\leq C  \underset{\btau_{ph} \in \bbX_{ph} \, \text{s.t.}\, \nabla \cdot \btau_{ph}=\0 }{\sup} \frac{b_n^p(\btau_{ph}, \btheta_h) }{\|\btau_{ph}\|_{\bbX_{p}}}\nonumber \\[1ex]
&\ds \quad
= C\underset{\btau_{ph} \in \bbX_{ph} \, \text{s.t.}\, \nabla \cdot \btau_{ph}=\0 }{\sup}
\frac{
(A\partial_t(\bsi_{ph}+\alpha p_{ph}\bI), \btau_{ph})
+b_{\sk}(\btau_{ph}, \bgamma_{ph})
+ b_s(\btau_{ph}, \bu_{sh})
}{\|\btau_{ph}\|_{\bbX_p}} \nonumber \\[1ex]
&\ds \quad
\leq C( \Vert A^{1/2} \partial_t (\bsi_{ph}+\alpha p_{ph} \bI) \Vert_{\bbL^2(\Omega_p)}
+ \Vert \bgamma_{ph} \Vert_{\bbQ_p}).
\end{align}
Combining \eqref{eq: stability 3} with  \eqref{eq: stability 4}--\eqref{eq: stability 5},
and choosing $\epsilon$ small enough, results in
\begin{align}\label{eq: stability 6}
& \ds \int_{0}^{t}
\Big(\Vert \bu_{fh} \Vert ^2_{\bV_f}
+ \vert \bu_{fh}-\btheta_h \vert ^2_{a_{\BJS}}
+\Vert p_{fh} \Vert_{\W_f}^2
+\Vert \bu_{sh}\Vert_{\bV_s}^2
+\Vert \bgamma_{ph}\Vert_{\bbQ_p}^2
+ \Vert \bu_{ph} \Vert ^2_{\bL^2(\Omega_p)}
+\Vert p_{ph} \Vert_{\W_p}^2
\nonumber \\[1ex]
&\ds 
\qquad + \Vert \lambda_h \Vert_{\Lambda_{ph}}^2 
+ \Vert \btheta_h \Vert_{\bLambda_{sh}}^2
\Big)ds
+\Vert A^{1/2}(\bsi_{ph}+\alpha p_{ph}\bI)(t)\Vert^2_{\bbL^2(\Omega_p)}
+s_0 \Vert p_{ph}(t) \Vert ^2_{\W_p}
\nonumber \\[1ex]
&\ds \quad \leq  
C\bigg(
\int_0^t \big(\Vert A^{1/2} \partial_t (\bsi_{ph}+\alpha p_{ph}\bI)\Vert_{\bbL^2(\Omega_p)}^2
+ \Vert \f_f\Vert^2_{\bL^2(\Omega_f)}
+\Vert q_f\Vert^2_{\L^2(\Omega_f)}
+\Vert \f_p\Vert^2_{\bL^2(\Omega_p)}
+\Vert q_p\Vert^2_{\L^2(\Omega_p)}\big)ds
\nonumber \\[1ex]
&\ds \qquad 
+ \Vert A^{1/2}(\bsi_{ph}+\alpha p_{ph}\bI)(0)\Vert^2_{\bbL^2(\Omega_p)} 
+ s_0 \Vert p_{ph}(0) \Vert ^2_{\L^2(\Omega_p)}
 \bigg).
\end{align}
To get a bound for $\Vert A^{1/2} \partial_t (\bsi_{ph}+\alpha
p_{ph}\bI)\Vert_{\L^2(0,t;\bbL^2(\Omega_p))}^2$, we differentiate in time 
\eqref{sd-1}, \eqref{sd-4}, \eqref{sd-5}, \eqref{sd-6}, and \eqref{sd-9}, take
$(\bv_{fh},w_{fh},\btau_{ph},\bv_{sh},\bchi_{ph},\bv_{ph},w_{ph},\xi_h,
\bphi_h) =(\bu_{fh},\partial_t p_{fh},\partial_t\bsi_{ph},\bu_{sh},\bgamma_{ph},\bu_{ph}, \partial_t
p_{ph}, \linebreak \partial_t \lambda_h,
\btheta_h)$ in \eqref{eq: semidiscrete formulation 1},
and add all equations, to obtain
\begin{align}\label{eq: stability 7}
&\ds 
\frac{1}{2}\partial_t a_f(\bu_{fh},\bu_{fh})
+\frac{1}{2}\partial_t a_{\BJS}(\bu_{fh}, \btheta_h; \bu_{fh}, \btheta_h) 
+ \Vert A^{1/2} \partial_t (\bsi_{ph}+\alpha p_{ph}\bI)\Vert_{\bbL^2(\Omega_p)}^2
+\frac{1}{2}\partial_t a_p(\bu_{ph},\bu_{ph})
 \nonumber \\[1ex]
& \ds
+ s_0 \Vert \partial_t p_{ph}\Vert_{\W_p}^2
=(\partial_t\f_f,\bu_{fh})_{\Omega_f}
+(q_f, \partial_t p_{fh})_{\Omega_f}
+(\partial_t \f_p, \bu_{sh})_{\Omega_p}
+(q_p, \partial_t p_{ph})_{\Omega_p}.
\end{align}
We next integrate \eqref{eq: stability 7} in time from $0$ to an arbitrary $t\in (0,T]$ and
use integration by parts in time for the last two terms:
\begin{align*}
\ds \int_0^t (q_f, \partial_t p_{fh})_{\Omega_f} ds
+  \int_0^t (q_p, \partial_t p_{ph})_{\Omega_p} ds 
& =(q_f, p_{fh})_{\Omega_f} \Big|_0^t -\int_0^t(\partial_t q_f, p_{fh})_{\Omega_f} \,ds \\
& \quad + (q_p, p_{ph})_{\Omega_p} \Big|_0^t -\int_0^t(\partial_t q_p, p_{ph})_{\Omega_p} \,ds.
\end{align*}
Making use of the continuity of $a_f$, $a_p$ and $a_{\BJS}$,
cf. \eqref{eq: continuous continuity 0}, the coercivity of
$a_f$ and $a_p$, the semi-positive definiteness of $a_{\BJS}$,
cf. \eqref{eq: continuous monotonicity 1}, and the Cauchy-Schwarz and
Young's inequalities, we get
\begin{align}\label{eq: stability 8}
&\ds \mu C_K^2 \Vert \bu_{fh}(t) \Vert ^2_{\bV_f}
+\frac{1}{2}\mu \alpha_{\BJS} k_{\max}^{-1/2}
\vert (\bu_{fh}-\btheta_h)(t) \vert ^2_{a_{\BJS}}
+\frac{1}{2}\mu k_{\max}^{-1} \Vert \bu_{ph}(t) \Vert ^2_{\bL^2(\Omega_p)}
 \nonumber \\[1ex]
&\ds \quad\qquad + \int_{0}^{t}(
\Vert A^{1/2}\partial_t(\bsi_{ph}+\alpha p_{ph}\bI)\Vert^2_{\bbL^2(\Omega_p)}+
s_0 \Vert \partial_t p_{ph} \Vert ^2_{\W_p}) \, ds \nonumber \\[1ex]
&\ds \leq 
\frac{\epsilon}{2}\Big(\int_{0}^{t}
(\Vert \bu_{fh}\Vert^2_{\bL^2(\Omega_f)}
+ \Vert p_{fh}\Vert^2_{\W_f}
+\Vert \bu_{sh}\Vert^2_{\bV_s}
+\Vert p_{ph}\Vert^2_{\W_p}) \, ds
+ \Vert p_{fh}(t)\Vert^2_{\W_f}
+ \Vert p_{ph}(t)\Vert^2_{\W_p} 
\Big)\nonumber \\[1ex]
&\ds \quad
+\frac{1}{2\epsilon}\Big(\int_{0}^{t}(\Vert \partial_t \f_f\Vert^2_{\bL^2(\Omega_f)}
+\Vert \partial_t q_f\Vert^2_{\L^2(\Omega_f)}
+\Vert \partial_t \f_p\Vert^2_{\bL^2(\Omega_p)}
+\Vert \partial_t q_p\Vert^2_{\L^2(\Omega_p)}) \, ds 
+\Vert q_f(t)\Vert^2_{\L^2(\Omega_f)}
\nonumber \\[1ex]
&\ds \quad
+ \Vert q_p(t)\Vert^2_{\L^2(\Omega_p)}
\Big)
+\mu\Vert \bu_{fh}(0) \Vert ^2_{\bH^1(\Omega_f)}
+\frac{1}{2}\mu \alpha_{\BJS} k_{\min}^{-1/2}
\vert (\bu_{fh}-\btheta_h)(0) \vert ^2_{a_{\BJS}}
+\frac{1}{2}\Vert p_{fh}(0)\Vert^2_{\W_f}
\nonumber \\[1ex]
&\ds \quad
+\frac{1}{2} \mu k_{\min}^{-1} \Vert \bu_{ph}(0) \Vert ^2_{\bL^2(\Omega_p)}
+\frac{1}{2}\Vert p_{ph}(0)\Vert^2_{\W_p}
+\frac{1}{2}\Vert q_f(0)\Vert^2_{\L^2(\Omega_f)}
+\frac{1}{2}\Vert q_p(0)\Vert^2_{\L^2(\Omega_p)}.
\end{align}
We note that the first four terms on the right hand side are controlled in \eqref{eq: stability 6}, while the terms $\Vert
p_{fh}(t)\Vert_{\W_f}$ and $\Vert
p_{ph}(t)\Vert_{\W_p}$ are controlled in the inf-sup bound
\eqref{eq: stability 4}. Thus, combining \eqref{eq: stability 4}, \eqref{eq: stability 6} and \eqref{eq: stability 8}, and taking $\epsilon$ small enough, we obtain
\begin{align}\label{eq: stability 9}
& \ds \int_{0}^{t}
\Big(\Vert \bu_{fh} \Vert ^2_{\bV_f}
+ \vert \bu_{fh}-\btheta_h \vert ^2_{a_{\BJS}}
+\Vert p_{fh} \Vert_{\W_f}^2
+\Vert A^{1/2}\partial_t(\bsi_{ph}+\alpha p_{ph}\bI)\Vert^2_{\bbL^2(\Omega_p)}
+\Vert \bu_{sh}\Vert_{\bV_s}^2
+\Vert \bgamma_{ph}\Vert_{\bbQ_p}^2
\nonumber \\[1ex]
&\ds 
+ \Vert \bu_{ph} \Vert ^2_{\bL^2(\Omega_p)}
+\Vert p_{ph} \Vert_{\W_p}^2
+ s_0 \Vert \partial_t p_{ph} \Vert ^2_{\W_p}
+\Vert \lambda_h \Vert_{\Lambda_{ph}}^2 
+ \Vert \btheta_h \Vert_{\bLambda_{sh}}^2
\Big)ds
+ \Vert \bu_{fh}(t) \Vert ^2_{\bV_f}
+ \vert (\bu_{fh}-\btheta_h)(t) \vert ^2_{a_{\BJS}}
\nonumber \\[1ex]
&\ds 
+ \Vert p_{fh}(t) \Vert ^2_{\W_f}
+\Vert A^{1/2}(\bsi_{ph}+\alpha p_{ph}\bI)(t)\Vert^2_{\bbL^2(\Omega_p)} 
+ \Vert \bu_{ph}(t) \Vert ^2_{\bL^2(\Omega_p)}
+\Vert p_{ph}(t) \Vert^2_{\W_p} 
+\Vert \lambda_h(t) \Vert^2_{\Lambda_{ph}}
\nonumber \\[1ex]
&\ds \leq  
C\bigg(\int_{0}^{t}\big(
\Vert \f_f\Vert^2_{\bL^2(\Omega_f)}
+\Vert \f_p\Vert^2_{\bL^2(\Omega_p)}
+\Vert q_f\Vert^2_{\L^2(\Omega_f)}
+\Vert q_p\Vert^2_{\L^2(\Omega_p)}\big)ds
+ \|\f_f(t)\|^2_{\L^2(\Omega_f)}
\nonumber \\[1ex]
&\ds 
+ \int_{0}^{t}(\Vert \partial_t \f_f\Vert^2_{\bL^2(\Omega_f)}
+\Vert \partial_t \f_p\Vert^2_{\bL^2(\Omega_p)}
+\Vert \partial_t q_f\Vert^2_{\L^2(\Omega_f)}
+\Vert \partial_t q_p\Vert^2_{\L^2(\Omega_p)}) \, ds \,
+  \|q_f(t)\|^2_{\L^2(\Omega_f)}+ \|q_p(t)\|^2_{\L^2(\Omega_p)}
\nonumber \\[1ex]
&\ds 
+ \Vert \bu_{fh}(0) \Vert ^2_{\bV_f} 
+ \vert (\bu_{fh}-\btheta_h)(0) \vert ^2_{a_{\BJS}} 
+ \Vert p_{fh}(0) \Vert ^2_{\W_f}
+ \Vert A^{1/2}\bsi_{ph}(0)\Vert^2_{\bbL^2(\Omega_p)} 
+ \Vert \bu_{ph}(0) \Vert ^2_{\bL^2(\Omega_p)}
+ \Vert p_{ph}(0) \Vert ^2_{\W_p}
\nonumber \\[1ex]
&\ds 
+ \|q_f(0)\|^2_{\L^2(\Omega_f)} 
+ \|q_p(0)\|^2_{\L^2(\Omega_p)} 
\bigg).
\end{align}
We remark that in the above bound we have obtained control on
$\|p_{ph}(t)\|_{\L^2(\Omega_p)}$ independent of $s_0$.
To bound the initial data terms above, we recall that
$(\bu_\fh(0), p_\fh(0), \bsi_\ph(0), \bu_\ph(0), p_\ph(0), \lambda_h(0), \btheta_h(0))
= (\bu_{\fh,0}, p_{\fh,0},\bsi_{\ph,0},   \bu_{\ph,0}, p_{\ph,0}, \lambda_{h,0}, \btheta_{h,0})$
and the construction of the discrete initial data
\eqref{eq: dis ini S-D}--\eqref{eq: dis ini sigma}. Combining the two systems and
using the steady-state version of the arguments presented in
\eqref{eq: stability 1}--\eqref{eq: stability 4}, we obtain
\begin{align}
& \|\bu_{fh}(0)\|_{\bV_f} + \|p_\fh(0)\|_{\W_f} + \|A^{1/2}\bsi_\ph(0)\|_{\bbL^2(\Omega_p)}
+ \|\bu_\ph(0)\|_{\bL^2(\Omega_p)} + \|p_\ph(0)\|_{\W_p} 
+ |(\bu_{fh}-\btheta_h)(0)|_{a_{\BJS}}\nonumber \\
& \quad
\le C(\|\nabla \cdot (\bK\nabla p_{p,0})\|_{\L^2(\Omega_p)}
+ \| \f_f(0) \|_{\bL^2(\Omega_f)}
+ \| q_f(0) \|_{\L^2(\Omega_f)} 
+ \| \f_p(0) \|_{\bL^2(\Omega_p)}.
  \label{init-data-bound}
\end{align}
We complete the argument by deriving bounds for $\Vert \nabla \cdot
\bu_{ph}\Vert_{\L^2(\Omega_p)}$ and $\Vert \nabla \cdot
\bsi_{ph}\Vert_{\bL^2(\Omega_p)}$. Due to \eqref{div-prop}, we can choose $w_{ph}=\nabla
\cdot \bu_{ph}$ in \eqref{sd-7}, obtaining
\begin{align*}
& \ds 
\Vert \nabla \cdot \bu_{ph}\Vert^2_{\L^2(\Omega_p)}
=
-(A\partial_t(\bsi_{ph}+\alpha p_{ph}\bI), \nabla \cdot \bu_{ph})_{\Omega_p}
-(s_0 \partial_t p_{ph}, \nabla \cdot \bu_{ph})_{\Omega_p} 
+(q_p,\nabla \cdot \bu_{ph})_{\Omega_p}
\nonumber \\[1ex]
&\ds \quad \leq 
(
a_{\max}^{1/2} \Vert A^{1/2} \partial_t (\bsi_{ph}+\alpha p_{ph}\bI) \Vert_{\bbL^2(\Omega_p)}
+s_0\Vert \partial_t p_{ph} \Vert_{\L^2(\Omega_p)}
+ \Vert q_p \Vert_{\L^2(\Omega_p)} )\Vert \nabla \cdot \bu_{ph} \Vert_{\L^2(\Omega_p)},
\end{align*}
therefore
\begin{equation}\label{eq: bound for div up}
\ds \int_0^t\Vert \nabla \cdot \bu_{ph} \Vert^2_{\L^2(\Omega_p)}ds 
 \leq 
C\int_0^t(
\Vert A^{1/2} \partial_t (\bsi_{ph}+\alpha p_{ph} \bI)\Vert^2_{\bbL^2(\Omega_p)}
+ s_0 \Vert \partial_t p_{ph} \Vert^2_{\L^2(\Omega_p)}
+\Vert q_p \Vert^2_{\L^2(\Omega_p)}
)ds.
\end{equation}
Similarly, the choice of $\bv_{sh}=\nabla \cdot\bsi_{ph}$ in \eqref{sd-4} gives
\begin{equation}\label{eq: bound for div sigma}
\ds \Vert\nabla \cdot\bsi_{ph}\Vert_{\bL^2(\Omega_p)}
\leq \Vert \f_p \Vert_{\bL^2(\Omega_p)} \qan
\ds \int_0^t\Vert \nabla \cdot \bsi_{ph} \Vert^2_{\bL^2(\Omega_p)}ds 
\leq 
\int_0^t\Vert \f_p \Vert^2_{\bL^2(\Omega_p)} ds.
\end{equation}

Combining \eqref{eq: stability 9}--\eqref{eq: bound for div sigma},
we conclude \eqref{eq: stability analysis}, where we also use
\begin{equation*}
\Vert A^{1/2}\bsi_{ph}(t) \Vert_{\bbL^2(\Omega_p)} \leq C(\Vert A^{1/2}(\bsi_{ph}+\alpha p_{ph} \bI)(t) \Vert_{\bbL^2(\Omega_p)} + \Vert p_{ph}(t) \Vert_{\L^2(\Omega_p)}).
\end{equation*}
\end{proof}

%%%%%%%%%%%%%%%%%%%%%%%%%%%%%%%%%%%%%%%%%%%%%%%%%%%%%%%%%%%%%%
%%%%%%%%%%%%%%%%%%%%%%%%%%%%%%%%%%%%%%%%%%%%%%%%%%%%%%%%%%%%%%
%%%%%%%%%%%%%%%%%%%%%%%%%%%%%%%%%%%%%%%%%%%%%%%%%%%%%%%%%%%%%%
\section{Error analysis}\label{sec:error}

In this section we derive an a priori error estimate for the
semi-discrete formulation \eqref{eq: semidiscrete formulation 1}.  We
assume that the finite element spaces contain polynomials of degrees
$s_{\bu_f}$ and $s_{p_f}$ for $\bV_{fh}$ and $\W_{fh}$, $s_{\bu_p}$
and $s_{p_p}$ for $\bV_{ph}$ and $\W_{ph}$, $s_{\bsi_p}$,
$s_{\bu_s}$, and $s_{\bgamma_p}$ for $\bbX_{ph}$, $\bV_{sh}$, and
$\bbQ_{ph}$, $s_{\btheta}$ and $s_{\lambda}$ for $\bLambda_{sh}$ and
$\Lambda_{ph}$. Next, we define interpolation operators into the
finite elements spaces that will be used in the error analysis.
 
We recall that $P_h^{\bLambda_s}: \bLambda_s \rightarrow
\bLambda_{sh}$ is the $\L^2$-projection operator, cf. \eqref{eq:
  interpolation 0}, and define $P_h^{\Lambda_p}: \Lambda_p
\rightarrow \Lambda_{ph}$ as the $\L^2$-projection operator,
satisfying, for any $\xi \in \L^2(\Gamma_{fp})$,
$\langle \xi - P_h^{\Lambda_p} \xi , \xi_h \rangle_{\Gamma_{fp}} = 0$
$\forall \, \xi_h \in \Lambda_{ph}$. Since the discrete Lagrange multiplier
spaces are chosen as
$\bLambda_{sh} = \bbX_{ph} \, \bn_p |_{\Gamma_{fp}}$ and
$\Lambda_{ph} = \bV_{ph} \cdot \bn_p |_{\Gamma_{fp}}$,  respectively, we have
\begin{align}\label{eq: interpolation 1}
  &\ds \langle \bphi - P_h^{\bLambda_s} \bphi, \btau_{ph} \,\bn_p \rangle_{\Gamma_{fp}} = 0,
  \quad \forall \, \btau_{ph} \in \bbX_{ph}, &\quad
  \langle \xi - P_h^{\Lambda_p} \xi, \bv_{ph} \cdot \bn_p \rangle_{\Gamma_{fp}} = 0,
  \quad \forall \, \bv_{ph} \in \bV_{ph}.
\end{align}
These operators have 
approximation properties \cite{ciarlet1978},
\begin{align}\label{eq: approx property 1}
&\ds \Vert \bphi - P_h^{\bLambda_s} \bphi \Vert_{\bL^2(\Gamma_{fp})} \leq Ch_p^{s_{\btheta}+1} \Vert \bphi \Vert_{\bH^{s_{\btheta}+1}(\Gamma_{fp})}, &\Vert \xi - P_h^{\Lambda_p} \xi \Vert_{\L^2(\Gamma_{fp})} \leq Ch_p^{s_{\lambda}+1} \Vert \xi \Vert_{\H^{s_{\lambda}+1}(\Gamma_{fp})}.
\end{align}
Similarly, we introduce $P_h^{\W_f}: \W_f \rightarrow \W_{fh}$,
$P_h^{\W_p}: \W_p \rightarrow \W_{ph}$, $P_h^{\bV_s}: \bV_s
\rightarrow \bV_{sh}$ and $P_h^{\bbQ_p}: \bbQ_p \rightarrow \bbQ_{ph}$
as $\L^2$-projection operators, satisfying
\begin{equation}\label{eq: interpolation 2}
\begin{array}{ll}
  \ds 
( w_f - P_h^{\W_f} w_f, w_{fh} )_{\Omega_f}  =  0,  \quad \forall  \, w_{fh} \in \W_{fh}, \quad
&( w_p - P_h^{\W_p} w_p , w_{ph} )_{\Omega_p}  =  0,  \quad \forall \, w_{ph} \in \W_{ph}, \\[1ex]
\ds
( \bv_s - P_h^{\bV_s} \bv_s,  \bv_{sh} )_{\Omega_p} = 0, \quad \forall \, \bv_{sh} \in \bV_{sh}, \quad
&( \bchi_p - P_h^{\bbQ_p} \bchi_p , \bchi_{ph} )_{\Omega_p} = 0, \quad \forall \, \bchi_{ph} \in \bbQ_{ph},
\end{array}
\end{equation}
with approximation properties \cite{ciarlet1978},
\begin{equation}\label{eq: approx property 2}
\begin{array}{ll}
\ds 
\Vert w_f - P_h^{\W_f} w_f \Vert_{\L^2(\Omega_f)} \leq Ch_f^{s_{p_f}+1} \Vert w_f \Vert_{\H^{s_{p_f}+1}(\Omega_f)}, 
& \Vert w_p - P_h^{\W_p} w_p \Vert_{\L^2(\Omega_p)} \leq Ch_p^{s_{p_p}+1} \Vert w_p \Vert_{\H^{s_{p_p}+1}(\Omega_p)}, \\[1ex]
\ds \Vert \bv_s - P_h^{\bV_s} \bv_s \Vert_{\bL^2(\Omega_p)} \leq Ch_p^{s_{\bu_s}+1} \Vert \bv_s \Vert_{\bH^{s_{\bu_s}+1}(\Omega_p)}, 
& \Vert \bchi_p - P_h^{\bbQ_p} \bchi_p \Vert_{\bbL^2(\Omega_p)} \leq Ch_p^{s_{\bgamma_p}+1} \Vert \bchi_p \Vert_{\bbH^{s_{\bgamma_p}+1}(\Omega_p)}.
\end{array}
\end{equation}

Next, we consider a Stokes-like projection operator $I_h^{\bV_f}:
\bV_f \rightarrow \bV_{fh}$, defined by solving the problem: find
$I_h^{\bV_f} \bv_f$ and $\wt{p}_{fh}\in \W_{fh}$ such that
\begin{equation} \label{eq: interpolation 3}
  \begin{array}{ll}
a_f(I_h^{\bV_f}\bv_f, \bv_{fh})
-b_f(\bv_{fh},\wt{p}_{fh})
=a_f(\bv_f, \bv_{fh}), 
& \quad  \forall \, \bv_{fh}\in \bV_{fh}, \\[1ex]
b_f(I_h^{\bV_f}\bv_f, w_{fh})=b_f(\bv_f, w_{fh}),
& \quad  \forall \, w_{fh}\in \W_{fh}. 
  \end{array}
\end{equation}
The operator $I_h^{\bV_f}$ satisfies the approximation property \cite{fern2011}:
\begin{equation}\label{eq: approx property 3}
\Vert \bv_f - I_h^{\bV_f}\bv_f \Vert_{\bH^1(\Omega_f)} \leq Ch_f^{s_{\bu_f}}\Vert \bv_f \Vert_{\bH^{s_{\bu_f}+1}(\Omega_f)}.
\end{equation}
Let $I_h^{\bV_p}$ be the mixed finite element interpolant onto
$\bV_{ph}$, which satisfies for all $\bv_p \in \bV_p \cap
\bH^1(\Omega_p)$,
\begin{equation}\label{eq: interpolation 4}
  \begin{array}{ll}
  ( \nabla \cdot I_h^{\bV_p}\bv_p, w_{ph} )_{\Omega_p}
  = (\nabla \cdot \bv_p, w_{ph})_{\Omega_p}, & \quad  \forall \, w_{ph} \in \W_{ph},\\[1ex]
  \langle I_h^{\bV_p}\bv_p \cdot \bn_p, \bv_{ph} \cdot \bn_p \rangle_{\Gamma_{fp}}
  = \langle \bv_p \cdot \bn_p, \bv_{ph} \cdot \bn_p \rangle_{\Gamma_{fp}},
  & \quad  \forall \, \bv_{ph} \in  \bV_{ph}, 
  \end{array}
\end{equation}
and
\begin{equation}\label{eq: approx property 4}
  \begin{array}{l}
\Vert \bv_p - I_h^{\bV_p}\bv_p \, \Vert_{\bL^2(\Omega_p)} \, \leq \, C h_p^{s_{\bu_p}+1} \Vert \, \bv_p \Vert_{\bH^{s_{\bu_p}+1}(\Omega_p)}, \\[1ex]
\Vert \, \nabla \cdot (\bv_p - I_h^{\bV_p}\bv_p) \, \Vert_{\L^2(\Omega_p)} \, \leq \, C h_p^{s_{\bu_p}+1} \Vert \, \nabla \cdot \bv_p \Vert_{\H^{s_{\bu_p}+1}(\Omega_p)}.
\end{array}
\end{equation}
For $\bbX_{ph}$, we consider the weakly symmetric elliptic projection introduced in
\cite{arnold2014} and extended in \cite{eldar_elastdd} to the case of
Neumann boundary condition: given $\bsi_p \in \bbX_p \cap \bbH^1(\Omega_p)$, find
$(\wt{\bsi}_{ph}, \wt{\bbeta}_{ph}, \wt{\brho}_{ph})
\in \bbX_{ph} \times \bV_{sh} \times \bbQ_{ph}$ such that
\begin{equation}\label{eq: interpolation 5}
  \begin{array}{ll}
    (\wt{\bsi}_{ph}, \btau_{ph}) + (\wt{\bbeta}_{ph}, \nabla \cdot \btau_{ph})
    + (\wt{\brho}_{ph}, \btau_{ph}) = (\bsi_p, \btau_{ph}),
    & \quad  \forall \, \btau_{ph} \, \in \, \bbX_{ph}^0, \\[1ex]
    (\nabla \cdot \wt{\bsi}_{ph}, \bv_{sh}) = (\nabla \cdot \bsi_p, \bv_{sh}),
    & \quad  \forall \, \bv_{sh} \, \in \, \bV_{sh}, \\[1ex]
    (\wt{\bsi}_{ph}, \bchi_{ph}) = (\bsi_p, \bchi_{ph}),
    & \quad  \forall \, \bchi_{ph} \, \in \, \bbQ_{ph}, \\[1ex]
    \langle  \wt{\bsi}_{ph} \bn_p, \btau_{ph} \bn_p \rangle_{\Gamma_{fp}}
    = \langle \bsi_p \bn_p, \btau_{ph} \bn_p \rangle_{\Gamma_{fp}},
    & \quad  \forall \, \btau_{ph} \, \in \, \bbX_{ph}^{\Gamma_{fp}}, 
  \end{array}
\end{equation}
where $\ds \bbX_{ph}^0=\{\btau_{ph}\in \bbX_{ph}: \btau_{ph} \bn_p =
\0 \, \text{ on } \, \Gamma_{fp} \}$, and $\bbX_{ph}^{\Gamma_{fp}}$ is
the complement of $\bbX_{ph}^0$ in $\bbX_{ph}$, which spans the degrees of
freedoms on $\Gamma_{fp}$. We define $I_h^{\bbX_p}\bsi_p :=
\wt{\bsi}_{ph}$, which satisfies
\begin{equation}\label{eq: approx property 5}
\begin{array}{l}
\Vert \bsi_p - I_h^{\bbX_p}\bsi_p \Vert_{\bbL^2(\Omega_p)} \leq h_p^{s_{\bsi_p}+1} \Vert \bsi_p \Vert_{\bbH^{s_{\bsi_p}+1}(\Omega_p)}, \\[1ex]
\Vert \nabla \cdot (\bsi_p - I_h^{\bbX_p}\bsi_p) \Vert_{\bL^2(\Omega_p)} \leq C h_p^{s_{\bsi_p}+1} \Vert \nabla \cdot \bsi_p \Vert_{\bH^{s_{\bsi_p}+1}(\Omega_p)}.
\end{array}
\end{equation}

We now establish the main result of this section.

\begin{theorem}\label{thm: error analysis}
  Assuming sufficient regularity of the solution to the continuous
  problem \eqref{weak-form},
  for the solution of the semi-discrete problem \eqref{eq: semidiscrete formulation 1},
    there exists a constant $C$ independent of $h$, $s_0$, and $a_{\min}$ such that
\begin{align}\label{eq: error analysis result}
&\ds 
\Vert \bu_f - \bu_{fh} \Vert_{\L^\infty(0,T; \bV_f)}
+ \Vert \bu_f - \bu_{fh} \Vert_{\L^2(0,T; \bV_f)}
+ \vert (\bu_f - \btheta) - (\bu_{fh} - \btheta_h) \vert_{\L^\infty(0,T; a_{\BJS})}
\nonumber \\[1ex]
&\ds \quad
+ \vert (\bu_f - \btheta) - (\bu_{fh} - \btheta_h) \vert_{\L^2(0,T; a_{\BJS})}
+ \Vert p_f - p_{fh} \Vert_{\L^\infty(0,T; \W_f)}
+ \Vert p_f - p_{fh} \Vert_{\L^2(0,T; \W_f)}
\nonumber \\[1ex]
&\ds \quad
+ \Vert A^{1/2}(\bsi_p - \bsi_{ph}) \Vert_{\L^\infty(0,T; \bbL^2(\Omega_p))}
+ \Vert \nabla \cdot (\bsi_p - \bsi_{ph})\Vert_{\L^\infty(0,T; \bL^2(\Omega_p))} 
+ \Vert \nabla \cdot (\bsi_p - \bsi_{ph})\Vert_{\L^2(0,T; \bL^2(\Omega_p))}
\nonumber \\[1ex]
&\ds \quad
+ \Vert A^{1/2}\partial_t ((\bsi_p+\alpha p_p \bI) - (\bsi_{ph}+\alpha p_{ph} \bI)) \Vert_{\L^2(0,T; \bbL^2(\Omega_p))}
+ \Vert \bu_s - \bu_{sh} \Vert_{\L^2(0,T; \bV_s)}
+ \Vert \bgamma_p - \bgamma_{ph} \Vert_{\L^2(0,T; \bbQ_p)}
\nonumber \\[1ex]
&\ds \quad
+\Vert \bu_p - \bu_{ph} \Vert_{\L^\infty(0,T; \bL^2(\Omega_p))}
+\Vert \bu_p - \bu_{ph} \Vert_{\L^2(0,T; \bV_p})
+ \Vert p_p - p_{ph} \Vert_{\L^\infty(0,T; \W_p)}
+ \Vert p_p - p_{ph} \Vert_{\L^2(0,T; \W_p)}
 \nonumber \\[1ex]
&\ds \quad
+ \sqrt{s_0}\Vert \partial_t(p_p - p_{ph}) \Vert_{\L^2(0,T; \W_p))}
+ \Vert \lambda - \lambda_h \Vert_{\L^\infty(0,T; \Lambda_{ph})}
+ \Vert \lambda - \lambda_h \Vert_{\L^2(0,T; \Lambda_{ph})}
+ \Vert \btheta - \btheta_h \Vert_{\L^2(0,T; \bLambda_{sh})} 
 \nonumber \\[1ex]
&\ds 
\leq 
C\sqrt{\exp(T)}\Big(
h_f^{s_{\bu_f}}
\Vert \bu_f \Vert_{\H^1(0,T; \bH^{s_{\bu_f}+1}(\Omega_f))}
+ h_f^{s_{p_f}+1}
\Vert p_f \Vert_{\H^1(0,T; \H^{s_{p_f}+1}(\Omega_f))}
+ h_p^{s_{\bsi_p}+1}
\Vert \bsi_p \Vert_{\H^1(0,T; \bbH^{s_{\bsi_p}+1}(\Omega_p))}
 \nonumber \\[1ex]
&\ds \quad
+ h_p^{s_{\bsi_p}+1}
\Vert \nabla\cdot\bsi_p \Vert_{\L^\infty(0,T; \bH^{s_{\bsi_p}+1}(\Omega_p))}
+ h_p^{s_{\bu_s}+1}
\Vert \bu_s \Vert_{\L^2(0,T; \bH^{s_{\bu_s}+1}(\Omega_p))}
+ h_p^{s_{\bgamma_p}+1}
\Vert \gamma_p \Vert_{\H^1(0,T; \bbH^{s_{\bgamma_p}+1}(\Omega_p))}
\nonumber \\[1ex]
&\ds \quad
+ h_p^{s_{\bu_p}+1}
(
\Vert \bu_p \Vert_{\H^1(0,T; \bH^{s_{\bu_p}+1}(\Omega_p))}
+ \Vert \nabla\cdot\bu_p \Vert_{\L^2(0,T; \H^{s_{\bu_p}+1}(\Omega_p))}
)
+ h_p^{s_{p_p}+1}
\Vert p_p \Vert_{\H^1(0,T; \H^{s_{p_p}+1}(\Omega_p))}
\nonumber \\[1ex]
&\ds \quad
+ h_p^{s_{\lambda}+1}
\Vert \lambda \Vert_{\H^1(0,T; \H^{s_{\lambda}+1}(\Gamma_{fp}))}
+ h_p^{s_{\btheta}+1}
\Vert \btheta \Vert_{\H^1(0,T; \bH^{s_{\btheta}+1}(\Gamma_{fp}))}
+ h_p^{s_{\bgamma_p}+1} \Vert \brho_{p}(0) \Vert_{\bbH^{s_{\bgamma_p}+1}(\Omega_p)}
\Big).
\end{align}

\end{theorem}
\begin{proof}
We introduce the error terms as the differences of the solutions to
\eqref{weak-form} and \eqref{eq: semidiscrete formulation 1} and decompose them
into approximation and discretization errors using the interpolation operators:
\begin{align}
&\ds e_{\bu_f}  :=  \bu_f - \bu_{fh}  =  (\bu_f - I_h^{\bV_f} \bu_f) + ( I_h^{\bV_f} \bu_f - \bu_{fh})  :=  e^I_{\bu_f} + e^h_{\bu_f}, \nonumber \\[1ex]
&\ds e_{p_f}  :=  p_f - p_{fh} = (p_f - P_h^{\W_f} p_f) + ( P_h^{\W_f} p_f - p_{fh}) := e^I_{p_f} + e^h_{p_f},  \nonumber \\[1ex]
&\ds e_{\bu_p} := \bu_p - \bu_{ph} = (\bu_p - I_h^{\bV_p} \bu_p) + ( I_h^{\bV_p} \bu_p - \bu_{ph}) := e^I_{\bu_p} + e^h_{\bu_p}, \nonumber \\[1ex]
&\ds e_{p_p}  :=  p_p - p_{ph} = (p_p - P_h^{\W_p} p_p) + ( P_h^{\W_p} p_p - p_{ph}) := e^I_{p_p} + e^h_{p_p},\nonumber \\[1ex]
&\ds e_{\bsi_p} := \bsi_p - \bsi_{ph} = (\bsi_p - I_h^{\bbX_p} \bsi_p) + ( I_h^{\bbX_p} \bsi_p - \bsi_{ph}) := e^I_{\bsi_p} + e^h_{\bsi_p},   \nonumber \\[1ex]
&\ds e_{\bu_s} := \bu_s - \bu_{sh} = (\bu_s - P_h^{\bV_s} \bu_s) + ( P_h^{\bV_s} \bu_s - \bu_{sh}) := e^I_{\bu_s} + e^h_{\bu_s},\nonumber \\[1ex]
&\ds e_{\bgamma_p} := \bgamma_p - \bgamma_{ph} = (\bgamma_p - P_h^{\bbQ_p} \bgamma_p) + ( P_h^{\bbQ_p} \bgamma_p - \bgamma_{ph}) := e^I_{\bgamma_p} + e^h_{\bgamma_p}, \nonumber \\[1ex]
&\ds e_{\btheta} := \btheta - \btheta_h = (\btheta - P_h^{\bLambda_s} \btheta) + ( P_h^{\bLambda_s} \btheta - \btheta_h) := e^I_{\btheta} + e^h_{\btheta}, \nonumber \\[1ex]
&\ds e_{\lambda} := \lambda - \lambda_h = (\lambda - P_h^{\bLambda_p} \lambda) + ( P_h^{\bLambda_p} \lambda - \lambda_h) := e^I_{\lambda} + e^h_{\lambda}.
\end{align}
We also define the approximation errors for non-differentiated variables:
\begin{equation*}
e_{\bbeta_p}^I = \bbeta_p - P_h^{\bV_s}\bbeta_{p}, \quad
e_{\brho_p}^I = \brho_p - P_h^{\bbQ_p}\brho_{p}, \quad
e_{\bpsi}^I = \bpsi - P_h^{\bLambda_s}\bpsi.
\end{equation*}
We form the error equations by subtracting the semi-discrete equations
\eqref{eq: semidiscrete formulation 1} from the continuous equations
\eqref{weak-form}:
\begin{subequations}\label{eq: error equation 1}
\begin{align}
&\ds a_f(e_{\bu_f},\bv_{fh})+b_f(\bv_{fh},e_{p_f})+b_{\Gamma}(\bv_{fh},\0,\0;e_{\lambda})+a_{\BJS}(e_{\bu_f},e_{\btheta};\bv_{fh},\0)=0, \label{eq: error equation 1a}
\\[1ex]
&\ds 
-b_f(e_{\bu_f}, w_{fh})=0, \label{eq: error equation 1b}\\[1ex]
&\ds 
a_e(\partial_t e_{\bsi_p}, \partial_t e_{p_p}; \btau_{ph}, 0) +b_s(\btau_{ph},e_{\bu_s})
+b_{\sk}(\btau_{ph}, e_{\bgamma_p})
-b_n^p(\btau_{ph}, e_{\btheta})=0,
\label{eq: error equation 1c}
\\[1ex]
&\ds 
-b_s(e_{\bsi_p},\bv_{sh})=0,
\label{eq: error equation 1d}
\\[1ex]
&\ds 
-b_{\sk}(e_{\bsi_p},\bchi_{ph})=0,
\label{eq: error equation 1e}
\\[1ex]
&\ds 
a_p(e_{\bu_p},\bv_{ph})+b_p(\bv_{ph}, e_{p_p})+b_{\Gamma}(\0,\bv_{ph},\0;e_{\lambda})=0,
\label{eq: error equation 1f}
\\[1ex]
&\ds 
a_p^p(\partial_t e_{p_p}, w_{ph})
+a_e(\partial_t e_{\bsi_p}, \partial_t e_{p_p}; \0, w_{ph}) 
-b_p(e_{\bu_p}, w_{ph}) =0,
\label{eq: error equation 1g}
\\[1ex]
&\ds 
-b_{\Gamma}(e_{\bu_f}, e_{\bu_p}, e_{\btheta} ;\xi_h) 
=0,
\label{eq: error equation 1h}
\\[1ex]
&\ds 
b_{\Gamma}(\0,\0,\bphi_h;e_{\lambda})
+a_{\BJS}(e_{\bu_f},e_{\btheta};\0,\bphi_h)
+b_n^p(e_{\bsi_p}, \bphi_h)
=0. \label{eq: error equation 1i}
\end{align}
\end{subequations}
Setting $\bv_{fh}=e^h_{\bu_f}, w_{fh}=e^h_{p_f}, \btau_{ph}=e^h_{\bsi_p}, \bv_{sh}=e^h_{\bu_s}, \bchi_{ph}=e^h_{\bgamma_p}, \bv_{ph}=e^h_{\bu_p}, w_{ph}=e^h_{p_p}, \xi_h=e^h_{\lambda}, \bphi_h=e^h_{\btheta}$, and summing the equations, we obtain
\begin{align}\label{eq: error equation 2}
&\ds a_f(e^I_{\bu_f},e^h_{\bu_f})
+a_f(e^h_{\bu_f},e^h_{\bu_f})
+a_{\BJS}(e^I_{\bu_f},e^I_{\btheta};e^h_{\bu_f},e^h_{\btheta}) 
+a_{\BJS}(e^h_{\bu_f},e^h_{\btheta};e^h_{\bu_f},e^h_{\btheta})
+a_e(\partial_t e^I_{\bsi_p}, \partial_t e^I_{p_p}; e^h_{\bsi_p}, e^h_{p_p})
 \nonumber \\[1ex]
&\ds
+a_e(\partial_t e^h_{\bsi_p}, \partial_t e^h_{p_p}; e^h_{\bsi_p}, e^h_{p_p})
+a_p(e^I_{\bu_p},e^h_{\bu_p})
+a_p(e^h_{\bu_p},e^h_{\bu_p})
+a_p^p(\partial_t e^I_{p_p}, e^h_{p_p})
+a_p^p(\partial_t e^h_{p_p}, e^h_{p_p})
 \nonumber \\[1ex]
&\ds
+b_n^p(e^I_{\bsi_p}, e^h_{\btheta})
+b_p(e^h_{\bu_p}, e^I_{p_p})
+b_f(e^h_{\bu_f},e^I_{p_f})
+b_s(e^h_{\bsi_p},e^I_{\bu_s})
+b_{\sk}(e^h_{\bsi_p}, e^I_{\bgamma_p})
+b_{\Gamma}(e^h_{\bu_f},e^h_{\bu_p},e^h_{\btheta};e^I_{\lambda})
\\[1ex]
&\ds
-b_n^p(e^h_{\bsi_p}, e^I_{\btheta})
-b_p(e^I_{\bu_p}, e^h_{p_p}) 
-b_f(e^I_{\bu_f}, e^h_{p_f})
-b_s(e^I_{\bsi_p},e^h_{\bu_s})
-b_{\sk}(e^I_{\bsi_p},e^h_{\bgamma_p})
-b_{\Gamma}(e^I_{\bu_f}, e^I_{\bu_p}, e^I_{\btheta} ;e^h_{\lambda}) \nonumber
=0. 
\end{align}
Due to \eqref{div-prop} and the properties of the projection operators
\eqref{eq: interpolation 1}, \eqref{eq: interpolation 2}, \eqref{eq: interpolation 3}, \eqref{eq: interpolation 4} and \eqref{eq: interpolation 5}, we have
\begin{gather*}
\ds b_n^p(e^h_{\bsi_p}, e^I_\btheta) =0, \quad \langle e^h_{\bu_p} \cdot \bn_p ,e^I_\lambda \rangle_{\Gamma_{fp}}
= 0, \qquad a_p^p(\partial_t e^I_{p_p},e^h_{p_p})_{\Omega_p} 
=0, \quad b_p(e^h_{\bu_p}, e^I_{p_p}) 
=0, \quad b_s(e^h_{\bsi_p}, e^I_{\bu_s})
= 0, \\[1ex]
\ds  
b_f(e^I_{\bu_f}, e^h_{p_f}) 
= 0, \qquad
b_p(e^I_{\bu_p},e^h_{p_p})
=0, \quad \langle e^I_{\bu_p} \cdot \bn_p ,e^h_\lambda \rangle_{\Gamma_{fp}}
= 0,\\[1ex]
\ds b_s(e^I_{\bsi_p}, e^h_{\bu_s})
=0, \quad b_{\sk}(e^I_{\bsi_p}, e^h_{\bgamma_p})
=0, \quad b_n^p(e^I_{\bsi_p}, e^h_\btheta) 
= 0.
\end{gather*}
With the use of the algebraic identity
$\int_S v \, \partial_t v = \frac{1}{2} \partial_t \Vert v \Vert^2_{\L^2(S)}$,
the error equation \eqref{eq: error equation 2} becomes
\begin{align}\label{eq: error equation 3}
&\ds 
a_f(e^h_{\bu_f},e^h_{\bu_f})
+a_{\BJS}(e^h_{\bu_f},e^h_{\btheta};e^h_{\bu_f},e^h_{\btheta})
+\frac{1}{2}\partial_t \Vert A^{1/2}(e^h_{\bsi_p} + \alpha_p\,e^h_{p_p}\,\bI) \Vert^2_{\bbL^2(\Omega_p)} 
+a_p(e^h_{\bu_p},e^h_{\bu_p})
+\frac{1}{2} s_0 \partial_t \Vert e^h_{p_p}\Vert^2_{\W_p}
\nonumber \\[1ex]
&\ds 
= -a_f(e^I_{\bu_f},e^h_{\bu_f})
-a_{\BJS}(e^I_{\bu_f},e^I_{\btheta};e^h_{\bu_f},e^h_{\btheta}) 
-a_e(\partial_t \,e^I_{\bsi_p}, \partial_t \, e^I_{p_p}; e^h_{\bsi_p}, e^h_{p_p})
-a_p(e^I_{\bu_p},e^h_{\bu_p})
\nonumber \\[1ex]
&\ds \quad
-b_f(e^h_{\bu_f}, e^I_{p_f})
-b_{\sk}(e^h_{\bsi_p}, e^I_{\bgamma_p})
- b_{\Gamma}(e^h_{\bu_f}, \0, e^h_{\btheta}; e^I_\lambda)
+ b_{\Gamma}(e^I_{\bu_f}, \0, e^I_{\btheta}; e^h_\lambda).
\end{align}
We proceed by integrating \eqref{eq: error equation 3} from
$0$ to $t\in (0,T]$, applying the coercivity properties of $a_f$ and
  $a_p$, the semi-positive definiteness of $a_{\BJS}$ \eqref{eq:
    continuous monotonicity 1}, the Cauchy-Schwarz inequality, the
  trace inequality \eqref{trace}, and Young's inequality, to
  get
\begin{align} \label{eq: error analysis 1}
&\ds 
\Vert e^h_{\bu_f} \Vert^2_{\L^2(0,t;\bV_f)}
+\vert e^h_{\bu_f}-e^h_{\btheta}\vert^2_{\L^2(0,t;a_{\BJS})}
+\Vert A^{1/2}(e^h_{\bsi_p} + \alpha e^h_{p_p} \bI)(t) \Vert^2_{\bbL^2(\Omega_p)}
+\Vert e^h_{\bu_p} \Vert^2_{\L^2(0,t;\bL^2(\Omega_p))}
+s_0\Vert e^h_{p_p}(t)\Vert^2_{\W_p}
\nonumber \\[1ex]
&\ds 
\leq 
\epsilon\Big(
\Vert e^h_{\bu_f} \Vert^2_{\L^2(0,t;\bV_f)}
+ \vert e^h_{\bu_f}-e^h_{\btheta} \vert^2_{\L^2(0,t;a_{\BJS})}
+ \Vert A^{1/2}(e^h_{\bsi_p}+\alpha \, e^h_{p_p}\, \bI)\Vert^2_{\L^2(0,t;\bbL^2(\Omega_p))}
+ \Vert A^{1/2} e^h_{\bsi_p} \Vert^2_{\L^2(0,t;\bbL^2(\Omega_p))}\nonumber \\[1ex]
&\ds  \quad
+ \Vert e^h_{\bu_p} \Vert^2_{\L^2(0,t;\bV_p)}
+ \Vert e^h_{\lambda} \Vert^2_{\L^2(0,t;\Lambda_{ph})}
+ \Vert e^h_{\btheta} \Vert^2_{\L^2(0,t;\bLambda_{sh})}
\Big)
+\frac{C}{\epsilon}\Big(
\Vert e^I_{\bu_f} \Vert^2_{\L^2(0,t;\bV_f)}
+ \vert e^I_{\bu_f}-e^I_{\btheta} \vert^2_{\L^2(0,t;a_{\BJS})}
\nonumber \\[1ex]
&\ds  \quad
+\Vert e^I_{p_f} \Vert^2_{\L^2(0,t;\W_f)}
+\Vert A^{1/2}\partial_t (e^I_{\bsi_p}+\alpha e^I_{p_p}\bI)\Vert^2_{\L^2(0,t;\bbL^2(\Omega_p))}
+\Vert e^I_{\bgamma_p} \Vert^2_{\L^2(0,t;\bbQ_p)}
+ \Vert e^I_{\bu_p}
\Vert^2_{\L^2(0,t;\bV_p)}
\nonumber \\[1ex]
&\ds  \quad
+ \Vert e^I_{\lambda} \Vert^2_{\L^2(0,t;\Lambda_{ph})}
+ \Vert e^I_{\btheta} \Vert^2_{\L^2(0,t;\bLambda_{sh})}
\Big)
+\Vert A^{1/2}(e^h_{\bsi_p} + \alpha e^h_{p_p} \bI)(0) \Vert^2_{\bbL^2(\Omega_p)}
+s_0 \Vert e^h_{p_p}(0)\Vert^2_{\W_p},
\end{align}
Here we also used that the extension of $A$ from $\bbS$ to $\bbM$ can be chosen as
the identity operator, therefore, cf. \cite{lee2016}, there exists $c > 0$ such that
\begin{equation}\label{eq: bsk bound}
\ds  b_{\sk}(e_{\bsi_p}^h, e_{\bgamma_p}^I)=\frac{1}{c}( e_{\bsi_p}^h, A e_{\bgamma_p}^I)_{\Omega_p}
  =\frac{1}{c}(A^{1/2} e_{\bsi_p}^h, A^{1/2} e_{\bgamma_p}^I)_{\Omega_p}
  \leq \frac{a_{\max}^{1/2}}{c}  \Vert A^{1/2} e_{\bsi_p}^h\Vert_{\bbL^2(\Omega_p)} \Vert  e_{\bgamma_p}^I\Vert_{\bbQ_p}.
\end{equation}

On the other hand, from the discrete inf-sup condition \eqref{eq: discrete inf-sup 1},
and using \eqref{eq: error equation 1a} and \eqref{eq: error equation 1f}, we have
\begin{align}\label{eq: error inf-sup 1}
&\ds 
\Vert e^h_{p_f} \Vert_{\W_f} 
+\Vert e^h_{p_p} \Vert_{\W_p} 
+\Vert e^h_{\lambda} \Vert_{\Lambda_{ph}} \nonumber \\[1ex]
&\ds \leq C
\underset{(\bv_{fh}, \bv_{ph}) \in \bV_{fh}\times\bV_{ph}}{\sup}\frac
{b_f(\bv_{fh},e^h_{p_f})+b_p(\bv_{ph},e^h_{p_p} )+b_{\Gamma}(\bv_{fh}, \bv_{ph}, \0; e^h_{\lambda})}
{\|(\bv_{fh}, \bv_{ph})\|_{\bV_{f}\times\bV_{p}}} \nonumber \\[1ex]
&\ds
=C\underset{(\bv_{fh}, \bv_{ph}) \in \bV_{fh}\times\bV_{ph}}{\sup}
\Bigg(\frac
{
-a_f(e^h_{\bu_f}, \bv_{fh})
-a_{\BJS}(e^h_{\bu_f},e^h_{\btheta};\bv_{fh},\0)
-a_f(e^I_{\bu_f}, \bv_{fh})
-a_{\BJS}(e^I_{\bu_f},e^I_{\btheta};\bv_{fh},\0)
}
{\|\bv_{fh}\|_{\bV_f}+\|\bv_{ph}\|_{\bV_p}} 
\nonumber \\[1ex]
&\ds  \qquad
+ \frac{
-a_p(e^h_{\bu_p}, \bv_{ph})
-a_p(e^I_{\bu_p}, \bv_{ph})
-b_f(\bv_{fh}, e_{p_f}^I)
-b_{\Gamma}(\bv_{fh}, \0, \0; e^I_{\lambda})
}
{\|\bv_{fh}\|_{\bV_f}+\|\bv_{ph}\|_{\bV_p}} \Bigg) \nonumber \\[1ex]
&\ds
\leq C(
\Vert e^h_{\bu_f} \Vert_{\bV_f}
+ \vert e^h_{\bu_f}-e^h_{\btheta} \vert_{a_{\BJS}}
+\Vert e^I_{\bu_f} \Vert_{\bV_f}
+ \vert e^I_{\bu_f}-e^I_{\btheta} \vert_{a_{\BJS}}
+ \Vert e^h_{\bu_p} \Vert_{\bL^2(\Omega_p)}
+ \Vert e^I_{\bu_p} \Vert_{\bL^2(\Omega_p)}
\nonumber \\[1ex]
&\ds \qquad 
+ \Vert e^I_{p_f} \Vert_{\W_f}
+ \Vert e^I_{\lambda} \Vert_{\Lambda_{ph}}
),
\end{align}
where we also used \eqref{div-prop}, \eqref{eq: interpolation 1} and
\eqref{eq: interpolation 2}. Similarly, the inf-sup condition
\eqref{eq: discrete inf-sup 2} and \eqref{eq: error equation 1c} give
\begin{align}\label{eq: error inf-sup 2}
&\ds 
\|e^h_{\bu_s}\|_{\bV_s}+\|e^h_{\bgamma_p}\|_{\bbQ_p}
\leq C
\underset{\btau_{ph} \in \bbX_{ph} \text{ s.t. } \btau_{ph}\bn_p=\0 \text{ on } \Gamma_{fp}}{\sup}
\frac{b_s(\btau_{ph},e^h_{\bu_s})+b_{\sk}(\btau_{ph},e^h_{\bgamma_p})}{\|\btau_{ph}\|_{\bbX_{p}}} \nonumber \\[1ex]
&\ds
= C \underset{\btau_{ph} \in \bbX_{ph} \text{ s.t. } \btau_{ph}\bn_p=\0 \text{ on } \Gamma_{fp}}{\sup} 
\Bigg(
\frac{
-a_e(\partial_t e^h_{\bsi_p}, \partial_t e^h_{p_p}; \btau_{ph},0)
+ b_n^p( \btau_{ph}, e^h_{\btheta}) 
}
{\|\btau_{ph}\|_{\bbX_p}} \nonumber \\[1ex]
&\ds \qquad
+\frac{
-a_e(\partial_t e^I_{\bsi_p},\partial_t e^I_{p_p}; \btau_{ph},0)
-b_{\sk}^p( \btau_{ph}, e^I_{\gamma_p}) 
}
{\|\btau_{ph}\|_{\bbX_p}}\Bigg) \nonumber \\[1ex]
&\ds
\leq C( 
\Vert A^{1/2} \partial_t (e^h_{\bsi_p}+\alpha e^h_{p_p} \bI) \Vert_{\bbL^2(\Omega_p)}
+\Vert A^{1/2} \partial_t (e^I_{\bsi_p}+\alpha e^I_{p_p} \bI) \Vert_{\bbL^2(\Omega_p)}
+\Vert e^I_{\bgamma_p}  \Vert_{\bbQ_p}),
\end{align}
where we also used \eqref{div-prop} and \eqref{eq: interpolation 2}. Finally,
using the inf-sup condition \eqref{eq: discrete inf-sup 3} and
\eqref{eq: error equation 1c}, we obtain
\begin{align}\label{eq: error inf-sup 3}
&\ds
\|e^h_{\btheta}\|_{\bLambda_{sh}} 
\leq C  \underset{\btau_{ph} \in \bbX_{ph} \, \text{s.t.}\, \nabla \cdot \btau_{ph}=\0 }
{\sup} \frac{b_n^p(\btau_{ph}, e^h_{\btheta}) }{\|\btau_{ph}\|_{\bbX_{p}}} \nonumber \\[1ex]
&\ds
= C\underset{\btau_{ph} \in \bbX_{ph} \, \text{s.t.}\, \nabla \cdot \btau_{ph}=\0 }{\sup}
\Bigg(
\frac{
a_e(\partial_t e^h_{\bsi_p}, \partial_t e^h_{p_p}; \btau_{ph},0)
+ b_s(\btau_{ph}, e^h_{\bu_s})
+ b_{\sk}(\btau_{ph}, e^h_{\bgamma_p})
}{\|\btau_{ph}\|_{\bbX_p}} \nonumber \\[1ex]
&\ds
\qquad
+ \frac{
a_e(\partial_t e^I_{\bsi_p},\partial_t e^I_{p_p};
\btau_{ph},0)
+ b_{\sk}(\btau_{ph}, e^I_{\bgamma_p})
}{\|\btau_{ph}\|_{\bbX_p}} \Bigg)
\nonumber \\[1ex]
&\ds
\leq C(  \Vert A^{1/2} \partial_t(e^h_{\bsi_p}+\alpha e^h_{p_p} \bI) \Vert_{\bbL^2(\Omega_p)}
+ \Vert e^h_{\bgamma_p} \Vert_{\bbQ_p}
+\Vert A^{1/2} \partial_t(e^I_{\bsi_p}+\alpha e^I_{p_p} \bI) \Vert_{\bbL^2(\Omega_p)}
+ \Vert e^I_{\bgamma_p} \Vert_{\bbQ_p}
),
\end{align}
where we also used \eqref{eq: interpolation 1}.

We next derive bounds for $\Vert \nabla \cdot e_{\bu_p}^h \Vert_{\L^2(\Omega_p)}$ and $\Vert \nabla \cdot e_{\bsi_p}^h \Vert_{\bL^2(\Omega_p)}$. Due to \eqref{div-prop}, we can choose $w_{ph}=\nabla \cdot e_{\bu_p}^h$ in \eqref{eq: error equation 1g}, obtaining
\begin{align}\label{eq: bound for error div up}
\ds 
\Vert \nabla \cdot e_{\bu_p}^h\Vert^2_{\L^2(\Omega_p)}
& =
-(s_0 \partial_t e_{p_p}^h, \nabla \cdot e_{\bu_p}^h)_{\Omega_p} 
-(A\partial_t(e_{\bsi_p}^h+\alpha e_{p_p}^h\bI), \nabla \cdot e_{\bu_p}^h)_{\Omega_p}
-(A\partial_t(e_{\bsi_p}^I+\alpha e_{p_p}^I\bI), \nabla \cdot e_{\bu_p}^h)_{\Omega_p}
\nonumber \\[1ex]
&\ds \leq 
(s_0\Vert \partial_t e_{p_p}^h \Vert_{\W_p}
+a_{\max}^{1/2} \Vert A^{1/2} \partial_t (e_{\bsi_p}^h+\alpha e_{p_p}^h\bI) \Vert_{\bbL^2(\Omega_p)}\nonumber \\[1ex]
&\ds \qquad 
+a_{\max}^{1/2} \Vert A^{1/2} \partial_t (e_{\bsi_p}^I+\alpha e_{p_p}^I \bI) \Vert_{\bbL^2(\Omega_p)} )\Vert \nabla \cdot e_{\bu_{p}}^h \Vert_{\L^2(\Omega_p)}.
\end{align}
Similarly, the choice of $\bv_{sh}=\nabla \cdot e_{\bsi_p}^h$ in \eqref{eq: error equation 1d} gives
\begin{equation}\label{eq: bound for error div sigma}
\ds \Vert\nabla \cdot e_{\bsi_p}^h(t)\Vert_{\bL^2(\Omega_p)}
=0 \qan
\ds \Vert \nabla \cdot e_{\bsi_p}^h \Vert_{\L^2(0,t;\bL^2(\Omega_p))}=0.
\end{equation}
Combining \eqref{eq: error analysis 1} with
\eqref{eq: error inf-sup 1}--\eqref{eq: bound for error div sigma}
and choosing $\epsilon$ small enough, results in
\begin{align} \label{eq: error analysis 2}
&\ds 
\Vert e^h_{\bu_f} \Vert^2_{\L^2(0,t;\bV_f)}
+\vert e^h_{\bu_f}-e^h_{\btheta}\vert^2_{\L^2(0,t;a_{\BJS})}
+\Vert e^h_{p_f} \Vert^2_{\L^2(0,t;\W_f)}
+\Vert A^{1/2}(e^h_{\bsi_p} + \alpha e^h_{p_p} \bI)(t) \Vert^2_{\bbL^2(\Omega_p)}
\nonumber \\[1ex]
&\ds  \quad
+\Vert \nabla \cdot e^h_{\bsi_p} \Vert^2_{\L^2(0,t;\bL^2(\Omega_p))}
+\Vert \nabla \cdot e^h_{\bsi_p}(t) \Vert^2_{\bL^2(\Omega_p)}
+\Vert e^h_{\bu_s} \Vert^2_{\L^2(0,t;\bV_s)} 
+\Vert e^h_{\bgamma_p} \Vert^2_{\L^2(0,t;\bbQ_p)}
+\Vert e^h_{\bu_p} \Vert^2_{\L^2(0,t;\bV_p)}
\nonumber \\[1ex]
&\ds  \quad
+\Vert e^h_{p_p} \Vert^2_{\L^2(0,t;\W_p)}
+s_0\Vert e^h_{p_p}(t)\Vert^2_{\W_p}
+\Vert e^h_{\lambda} \Vert^2_{\L^2(0,t;\Lambda_{ph})}
+\Vert e^h_{\btheta} \Vert^2_{\L^2(0,t;\bLambda_{sh})}
%+ \Vert e^h_{\bsi_p}(t)\Vert^2_{\bbX_p}
\nonumber \\[1ex]
&\ds 
\leq 
C(
\Vert A^{1/2}(e^h_{\bsi_p}+\alpha  e^h_{p_p} \bI)\Vert^2_{\L^2(0,t;\bbL^2(\Omega_p))}
+ \Vert A^{1/2} \partial_t (e^h_{\bsi_p}+\alpha e^h_{p_p} \bI) \Vert^2_{\L^2(0,t;\bbL^2(\Omega_p))}
+ s_0\Vert \partial_t e_{p_p}^h \Vert^2_{\L^2(0,t;\W_p)}
\nonumber \\[1ex]
&\ds  \quad
+ \Vert e^I_{\bu_f} \Vert^2_{\L^2(0,t;\bV_f)}
+ \vert e^I_{\bu_f}-e^I_{\btheta} \vert^2_{\L^2(0,t;a_{\BJS})}
+ \Vert e_{p_f}^I \Vert^2_{\L^2(0,t;\W_f)}
+\Vert A^{1/2} \partial_t (e^I_{\bsi_p}+\alpha e_{p_p}^I \bI)\Vert^2_{\L^2(0,t;\bbL^2(\Omega_p))}
\nonumber \\[1ex]
&\ds  \quad
+ \Vert e^I_{\bgamma_p} \Vert^2_{\L^2(0,t;\bbQ_p)}
+ \Vert e^I_{\bu_p} \Vert^2_{\L^2(0,t;\bV_p)}
%+\Vert \partial_t e^I_{p_p} \Vert^2_{\L^2(0,t;\W_p)}
+ \Vert e^I_{\lambda} \Vert^2_{\L^2(0,t;\Lambda_{ph})}
+ \Vert e^I_{\btheta} \Vert^2_{\L^2(0,t;\bLambda_{sh})}
\nonumber \\[1ex]
&\ds  \quad
+\Vert A^{1/2}(e^h_{\bsi_p} + \alpha e^h_{p_p} \bI)(0) \Vert^2_{\bbL^2(\Omega_p)}
+s_0 \Vert e^h_{p_p}(0)\Vert^2_{\L^2(\Omega_p)}),
\end{align}
where we also used 
\begin{equation}\label{eq: Absi bound}
\Vert A^{1/2} e_{\bsi_p}^h \Vert_{\L^2(0,t;\bbL^2(\Omega_p))} \leq C(\Vert A^{1/2} (e_{\bsi_p}^h+\alpha e_{p_p}^h\bI) \Vert_{\L^2(0,t;\bbL^2(\Omega_p))}+\Vert e_{p_p}^h\Vert_{\L^2(0,t;\W_p)}).
\end{equation}
In order to bound $\Vert A^{1/2} \partial_t (e^h_{\bsi_p}+\alpha e^h_{p_p}\bI) \Vert_{\L^2(0,t;\bbL^2(\Omega_p))}$ and $s_0\Vert \partial_t e_{p_p}^h \Vert_{\L^2(0,t;\W_p)}$, we differentiate in time \eqref{weak-form-1}, \eqref{weak-form-4}, \eqref{weak-form-5}, \eqref{weak-form-6}, and \eqref{weak-form-9} in the continuous equations and \eqref{sd-1}, \eqref{sd-4}, \eqref{sd-5}, \eqref{sd-6}, and \eqref{sd-9} in the semi-discrete equations, subtract the two systems, take 
$\ds (\bv_{fh},w_{fh},\btau_{ph},\bv_{sh},\bchi_{ph},\bv_{ph},w_{ph},\xi_h, \linebreak \bphi_h)
=(e^h_{\bu_f},\partial_t e^h_{p_f}, \partial_t e^h_{\bsi_p}, e^h_{\bu_s},  e^h_{\bgamma_p}, e^h_{\bu_p}, \partial_t e^h_{p_p}, \partial_t e^h_{\lambda}, e^h_{\btheta})$, and add all the equations together to obtain, in a way similar to \eqref{eq: error equation 3},
\begin{align}\label{eq: error equation 4}
&\ds 
\frac{1}{2}\partial_t a_f(e^h_{\bu_f},e^h_{\bu_f})
+\frac{1}{2}\partial_t a_{\BJS}(e^h_{\bu_f},e^h_{\btheta};e^h_{\bu_f},e^h_{\btheta})
+\Vert A^{1/2}\partial_t (e^h_{\bsi_p} + \alpha e^h_{p_p} \bI) \Vert^2_{\bbL^2(\Omega_p)}
\nonumber \\[1ex]
&\ds \quad
+\frac{1}{2}\partial_t a_p(e^h_{\bu_p},e^h_{\bu_p})
+ s_0 \Vert \partial_t e^h_{p_p}\Vert^2_{\W_p}
 \nonumber \\[1ex]
&\ds 
= -a_f(\partial_t e_{\bu_f}^I, e_{\bu_f}^h)-a_{\BJS}(\partial_t e^I_{\bu_f}, \partial_t e^I_{\btheta};e^h_{\bu_f},e^h_{\btheta}) 
-a_e(\partial_t e^I_{\bsi_p}, \partial_t e^I_{p_p}; \partial_t e^h_{\bsi_p}, \partial_t e^h_{p_p})
-a_p(\partial_t e_{\bu_p}^I, e_{\bu_p}^h)
\nonumber \\[1ex]
&\ds  \quad
-b_f(e_{\bu_f}^h, \partial_t e_{p_f}^I)
-b_{\sk}(\partial_t e_{\bsi_p}^h, e_{\bgamma_p}^I)
- b_{\Gamma}(e^h_{\bu_f},\0, e^h_{\btheta};\partial_t e^I_\lambda)
+ b_{\Gamma}(e^I_{\bu_f},\0, e^I_{\btheta};\partial_t e^h_\lambda).
\end{align}
Using integration by parts in time, we obtain
\begin{align*}
&\ds \int_0^t b_{\sk}(\partial_t e_{\bsi_p}^h, e_{\bgamma_p}^I) ds =
b_{\sk}( e_{\bsi_p}^h, e_{\bgamma_p}^I)\Big|_0^t 
- \int_0^t b_{\sk}( e_{\bsi_p}^h, \partial_t e_{\bgamma_p}^I) ds,  \\
&\ds \int_0^t \langle e^I_{\bu_f} \cdot \bn_f, \partial_t e^h_\lambda \rangle_{\Gamma_{fp}} ds =
\langle e^I_{\bu_f} \cdot \bn_f,
e^h_\lambda \rangle_{\Gamma_{fp}}\Big|_0^t 
- \int_0^t \langle \partial_t e^I_{\bu_f} \cdot \bn_f, e^h_\lambda \rangle_{\Gamma_{fp}} ds,  \\
&\ds \int_0^t \langle e^I_\btheta \cdot \bn_p, \partial_t e^h_\lambda \rangle_{\Gamma_{fp}} ds =
\langle e^I_\btheta \cdot \bn_p, e^h_\lambda \rangle_{\Gamma_{fp}}\Big|_0^t 
- \int_0^t \langle \partial_t e^I_\btheta \cdot \bn_p, e^h_\lambda \rangle_{\Gamma_{fp}} ds.
\end{align*}
We integrate \eqref{eq: error equation 4} over $(0,t)$ and apply the coercivity properties of $a_f$ and $a_p$,
the semi-positive definiteness of $a_{\BJS}$ \eqref{eq: continuous monotonicity 1}, the Cauchy-Schwarz inequality, the trace inequality \eqref{trace}, and Young's inequality,
to obtain
\begin{align}\label{eq: error analysis 3}
&\ds 
\Vert e^h_{\bu_f}(t) \Vert^2_{\bV_f}
+ \vert (e^h_{\bu_f}-e^h_{\btheta})(t) \vert^2_{a_{\BJS}}
+ \Vert A^{1/2}\partial_t (e^h_{\bsi_p}+\alpha e^h_{p_p}\bI) \Vert^2_{\L^2(0,t; \bbL^2(\Omega_p))}
\nonumber \\[1ex]
&\ds \quad
+ \Vert e^h_{\bu_p}(t) \Vert^2_{\bL^2(\Omega_p)}
+ s_0\Vert \partial_t e^h_{p_p} \Vert^2_{\L^2(0,t; W_p)}
 \nonumber \\[1ex]
&\ds 
\leq
\epsilon \Big(
\Vert e^h_{\bu_f}
\Vert^2_{\L^2(0,t;\bV_f)}
+\vert e^h_{\bu_f}-e^h_{\btheta} \vert^2_{\L^2(0,t;a_{\BJS})}
+ \Vert A^{1/2}\partial_t(e^h_{\bsi_p}+\alpha e^h_{p_p}\bI) \Vert^2_{\L^2(0,t;\bbL^2(\Omega_p))}
\nonumber \\[1ex]
&\ds \quad
+ \Vert A^{1/2} (e^h_{\bsi_p}+\alpha e^h_{p_p}\bI) \Vert^2_{\L^2(0,t;\bbL^2(\Omega_p))}
+ \Vert A^{1/2} (e^h_{\bsi_p}+\alpha e^h_{p_p}\bI)(t) \Vert^2_{\bbL^2(\Omega_p)}
+ \Vert e^h_{\bu_p} \Vert^2_{\L^2(0,t;\bV_p)}
\nonumber \\[1ex]
&\ds \quad
+ \Vert e^h_{p_p} \Vert^2_{\L^2(0,t;\W_p)}
+ \Vert e^h_{p_p}(t) \Vert^2_{\W_p}
+ \Vert e^h_{\lambda}(t) \Vert^2_{\Lambda_{ph}}
+ \Vert e^h_{\lambda} \Vert^2_{\L^2(0,t;\Lambda_{ph})}
+ \Vert e^h_{\btheta} \Vert^2_{\L^2(0,t;\bLambda_{sh})}
\Big) 
\nonumber \\[1ex]
&\ds \quad
+\frac{C}{\epsilon}\Big( 
\Vert \partial_t e^I_{\bu_f} \Vert^2_{\L^2(0,t;\bV_f)}
+\vert \partial_t (e^I_{\bu_f}-e^I_{\btheta}) \vert^2_{\L^2(0,t;a_{\BJS})}
+ \Vert \partial_t e^I_{p_f} \Vert^2_{\L^2(0,t;\W_f)}
\nonumber \\[1ex]
&\ds \quad
+ \Vert A^{1/2}\partial_t (e^I_{\bsi_p}+\alpha e_{p_p}^I) \Vert^2_{\L^2(0,t;\bbL^2(\Omega_p))}
+ \Vert \partial_t e^I_{\bgamma_p} \Vert^2_{\L^2(0,t;\bbQ_p)}
+ \Vert e^I_{\bgamma_p}(t) \Vert^2_{\bbQ_p}
\nonumber \\[1ex]
&\ds \quad
+ \Vert \partial_t e^I_{\bu_p} \Vert^2_{\L^2(0,t;\bV_p)}
+ \Vert \partial_t e^I_{\lambda} \Vert^2_{\L^2(0,t;\Lambda_{ph})} 
+ \Vert \partial_t e^I_{\btheta} \Vert^2_{\L^2(0,t;\bLambda_{sh})}
+ \Vert e^I_{\bu_f}(t) \Vert^2_{\bV_f}
+ \Vert e^I_{\btheta}(t) \Vert^2_{\bLambda_s}
\Big) 
\nonumber \\[1ex]
&\ds \quad
+ \Vert e^h_{\bu_f}(0) \Vert^2_{\bV_f}
+ \vert (e^h_{\bu_f}-e^h_{\btheta})(0) \vert^2_{a_{\BJS}}
+ \Vert A^{1/2} e^h_{\bsi_p}(0) \Vert^2_{\bbL^2(\Omega_p)}
+ \Vert e^h_{\bu_p}(0) \Vert^2_{\bL^2(\Omega_p)}
+\Vert e^h_{\lambda}(0)\Vert^2_{
\Lambda_p}
\nonumber \\[1ex]
&\ds \quad
+\Vert e^I_{\bu_f}(0)\Vert^2_{\bV_f}
+\Vert e^I_{\bgamma_p}(0)\Vert^2_{\bbQ_p}
+\Vert e^I_{\btheta}(0)\Vert^2_{\bLambda_{sh}},
\end{align}
where we also used 
$\ds  b_{\sk}( e_{\bsi_p}^h, \partial_t e_{\bgamma_p}^I)\leq C \Vert A^{1/2} e_{\bsi_p}^h\Vert_{\bbL^2(\Omega_p)} \Vert \partial_t e_{\bgamma_p}^I\Vert_{\bbQ_p}$, cf. \eqref{eq: bsk bound}, and 
\begin{equation*}
\Vert A^{1/2} e_{\bsi_p}^h(t) \Vert_{\bbL^2(\Omega_p)} \leq C(\Vert A^{1/2} (e_{\bsi_p}^h+\alpha e_{p_p}^h\bI)(t) \Vert_{\bbL^2(\Omega_p)}+\Vert e_{p_p}^h(t)\Vert_{\W_p}).
\end{equation*}
In addition, the choice of $\bv_{sh}=\nabla \cdot \partial_t e_{\bsi_p}^h$ in the time differentiated version of \eqref{eq: error equation 1} gives
\begin{equation}\label{eq: bound for error div sigma 2}
\ds \Vert\nabla \cdot \partial_t e_{\bsi_p}^h(t)\Vert_{\bL^2(\Omega_p)}
=0 \qan
\ds \Vert \nabla \cdot \partial_t e_{\bsi_p}^h \Vert_{\L^2(0,t;\bL^2(\Omega_p))}=0.
\end{equation}
Thus, combining \eqref{eq: error analysis 3} with \eqref{eq: error inf-sup 1}, \eqref{eq: error analysis 2} and \eqref{eq: bound for error div sigma 2}, and taking $\epsilon$ small enough, we obtain
\begin{align}\label{eq: error analysis 5}
&\ds 
\Vert e^h_{\bu_f} \Vert^2_{\L^2(0,t;\bV_f)}
+ \Vert e^h_{\bu_f}(t) \Vert^2_{\bV_f}
+\vert e^h_{\bu_f}-e^h_{\btheta}\vert^2_{\L^2(0,t;a_{\BJS})}
+ \vert (e^h_{\bu_f}-e^h_{\btheta})(t) \vert^2_{a_{\BJS}}
+\Vert e^h_{p_f} \Vert^2_{\L^2(0,t;\W_f)}
+\Vert e^h_{p_f}(t) \Vert_{\W_f}
\nonumber \\[1ex]
&\ds 
+\Vert \nabla \cdot e^h_{\bsi_p}  \Vert^2_{\L^2(0,t;\bbL^2(\Omega_p))}
+\Vert \nabla \cdot e^h_{\bsi_p} (t) \Vert^2_{\bbL^2(\Omega_p)}
+\Vert \nabla \cdot \partial_t e^h_{\bsi_p}  \Vert^2_{\L^2(0,t;\bbL^2(\Omega_p))}
+\Vert \nabla \cdot \partial_t e^h_{\bsi_p} (t) \Vert^2_{\bbL^2(\Omega_p)}
\nonumber \\[1ex]
&\ds 
+\Vert A^{1/2}(e^h_{\bsi_p} + \alpha e^h_{p_p} \bI)(t) \Vert^2_{\bbL^2(\Omega_p)}
+ \Vert A^{1/2}\partial_t (e^h_{\bsi_p}+\alpha e^h_{p_p}\bI) \Vert^2_{\L^2(0,t; \bbL^2(\Omega_p))}
+\Vert e^h_{\bu_s} \Vert^2_{\L^2(0,t;\bV_s)}
\nonumber \\[1ex]
&\ds 
+\Vert e^h_{\bgamma_p} \Vert^2_{\L^2(0,t;\bbQ_p)}
+\Vert e^h_{\bu_p} \Vert^2_{\L^2(0,t;\bV_p)}
+ \Vert e^h_{\bu_p}(t) \Vert^2_{\bL^2(\Omega_p)}
+\Vert e^h_{p_p} \Vert^2_{\L^2(0,t;\W_p)}
+\Vert e^h_{p_p}(t) \Vert^2_{\W_p}
\nonumber \\[1ex]
&\ds 
+s_0\Vert \partial_t e^h_{p_p}
 \Vert^2_{\L^2(0,t;\W_p)}
+\Vert e^h_{\lambda} \Vert^2_{\L^2(0,t;\Lambda_{ph})}
+\Vert e^h_{\lambda}(t) \Vert_{\Lambda_{ph}}
+\Vert e^h_{\btheta} \Vert^2_{\L^2(0,t;\bLambda_{sh})}
\nonumber \\[1ex]
&\ds  \leq
C(
\Vert A^{1/2}(e^h_{\bsi_p}+\alpha e^h_{p_p}\bI)\Vert^2_{\L^2(0,t;\bbL^2(\Omega_p))}
+\Vert e^I_{\bu_f} \Vert^2_{\L^2(0,t;\bV_f)}
+\Vert \partial_t e^I_{\bu_f} \Vert^2_{\L^2(0,t;\bV_f)}
+ \Vert e^I_{\bu_f}(t) \Vert^2_{\bV_f}
\nonumber \\[1ex]
&\ds  \quad
+ \vert e^I_{\bu_f}-e^I_{\btheta} \vert^2_{\L^2(0,t;a_{\BJS})} 
+ \vert \partial_t (e^I_{\bu_f}-e^I_{\btheta}) \vert^2_{\L^2(0,t;a_{\BJS})}
+ \vert (e^I_{\bu_f}-e^I_{\btheta})(t) \vert_{a_{\BJS}}
+ \Vert e^I_{p_f}\Vert^2_{\L^2(0,t;\W_f)}
\nonumber \\[1ex]
&\ds  \quad
+ \Vert \partial_t e^I_{p_f}\Vert^2_{\L^2(0,t;\W_f)}
+ \Vert e^I_{p_f}(t)\Vert^2_{\W_f}
+ \Vert A^{1/2}\partial_t (e^I_{\bsi_p}+ \alpha e_{p_p}^I \bI)\Vert^2_{\L^2(0,t;\bbL^2(\Omega_p))}
+ \Vert e^I_{\bgamma_p}\Vert^2_{\L^2(0,t;\bbQ_p)}
\nonumber \\[1ex]
&\ds  \quad
+ \Vert \partial_t e^I_{\bgamma_p}\Vert^2_{\L^2(0,t;\bbQ_p)}
+ \Vert e^I_{\bgamma_p}(t)\Vert^2_{\bbQ_p}
+ \Vert e^I_{\bu_p} \Vert^2_{\L^2(0,t;\bV_p)}
+ \Vert \partial_t e^I_{\bu_p} \Vert^2_{\L^2(0,t;\bV_p)}
+ \Vert e^I_{\bu_p}(t) \Vert_{\bV_p}
\nonumber \\[1ex]
&\ds  \quad
+ \Vert e^I_{\lambda} \Vert^2_{\L^2(0,t;\Lambda_{ph})}
+ \Vert \partial_t e^I_{\lambda} \Vert^2_{\L^2(0,t;\Lambda_{ph})}
+ \Vert e^I_{\lambda}(t) \Vert_{\Lambda_{ph}}
+ \Vert e^I_{\btheta} \Vert^2_{\L^2(0,t;\bLambda_{sh})}
+ \Vert \partial_t e^I_{\btheta} \Vert^2_{\L^2(0,t;\bLambda_{sh})} 
+ \Vert e^I_{\btheta}(t) \Vert^2_{\bLambda_{sh}}
\nonumber \\[1ex]
&\ds \quad
+ \Vert e^h_{\bu_f}(0) \Vert^2_{\bV_f}
+ \vert (e^h_{\bu_f}-e^h_{\btheta})(0) \vert^2_{a_{\BJS}}
+\Vert A^{1/2}(e^h_{\bsi_p} + \alpha e^h_{p_p} \bI)(0) \Vert^2_{\bbL^2(\Omega_p)}
+s_0 \Vert e^h_{p_p}(0)\Vert^2_{\L^2(\Omega_p)})
\nonumber \\[1ex]
&\ds \quad
+\Vert A^{1/2} e^h_{\bsi_p}(0) \Vert^2_{\bbL^2(\Omega_p)}
+ \Vert e^h_{\bu_p}(0) \Vert^2_{\bL^2(\Omega_p)}
+\Vert e^h_{\lambda}(0)\Vert^2_{
\Lambda_{ph}}
+\Vert e^I_{\bu_f}(0)\Vert^2_{\bV_f}
+\Vert e^I_{\bgamma_p}(0)\Vert^2_{\bbQ_p}
+\Vert e^I_{\btheta}(0)\Vert^2_{\bLambda_{sh}}).
\end{align}
We remark that in the above bound we have obtained control on $\Vert e_{p_p}^h(t)\Vert_{\W_p}$ independent of $s_0$. 

We next establish a bound on the initial data terms above. We recall that $(\bu_f(0), p_f(0),\bsi_p(0),
\bu_p(0), \linebreak  p_p(0), \lambda(0), \btheta(0)) =
(\bu_{f,0}, p_{f,0},\bsi_{p,0}, \bu_{p,0}, p_{p,0},
\lambda_{0}, \btheta_{0})$, cf. Corollary \ref{cor:init-data},
and $(\bu_\fh(0), p_\fh(0),\bsi_\ph(0), 
\linebreak \bu_\ph(0),  p_\ph(0), \lambda_h(0), \btheta_h(0)) =
(\bu_{\fh,0}, p_{\fh,0},\bsi_{\ph,0}, \bu_{\ph,0}, p_{\ph,0}, \lambda_{h,0}, \btheta_{h,0})$, cf. Theorem \ref{thm:well-posed-discr}.
 We first note that, since  $\btheta_{h,0}=P_h^{\bLambda_s} \, \btheta_0$,
\begin{equation}
  e_\btheta^h(0) =\0.
\end{equation}
Next, similarly to \eqref{init-data-bound}, we obtain %
\begin{align}\label{eq: error initial}
& \Vert e^h_{\bu_f}(0) \Vert^2_{\bV_f}
+ \vert (e^h_{\bu_f}-e^h_{\btheta})(0) \vert^2_{a_{\BJS}}
+\Vert A^{1/2}e^h_{\bsi_p} (0) \Vert^2_{\bbL^2(\Omega_p)}
+ \Vert e^h_{\bu_p}(0) \Vert^2_{\bL^2(\Omega_p)}
+\Vert e^h_{p_p}(0)\Vert^2_{\W_p}
+\Vert e^h_{\lambda}(0)\Vert^2_{
\Lambda_{ph}} 
\nonumber \\[1ex]
&\ds
\qquad \leq C \big( 
\Vert e_{\bu_f}^I(0) \Vert_{\bV_f} 
+ \vert e^I_{\bu_f}(0) - e^I_{\btheta}(0)\vert^2_{a_{\BJS}}
+\Vert e_{p_f}^I(0) \Vert_{\W_f} 
+ \Vert e_{\bsi_p}^I(0) \Vert_{\bbX_p} 
+ \Vert e_{\brho_p}^I(0) \Vert_{\bbQ_p} 
\nonumber \\[1ex]
&\ds \quad \qquad
+ \Vert e_{\bu_p}^I(0) \Vert_{\bV_p} 
+ \Vert e_{p_p}^I(0) \Vert_{\W_p}
+ \Vert e_{\lambda}^I(0) \Vert_{\Lambda_p}
+ \Vert e_{\btheta}^I(0) \Vert_{\bLambda_{sh}} 
\big).
\end{align}
Combining \eqref{eq: error analysis 5}-\eqref{eq: error initial}, using Gronwall's inequality for $\Vert A^{1/2}(e^h_{\bsi_p}+\alpha e^h_{p_p} \bI) \Vert^2_{\L^2(0,t;\bbL^2(\Omega_p))}$, the triangle inequality, and the approximation properties \eqref{eq: approx property 1}, \eqref{eq: approx property 2}, \eqref{eq: approx property 3}, and \eqref{eq: approx property 4}, we obtain \eqref{eq: error analysis result}.
\end{proof}

\section{Numerical results}\label{sec:numerical}

\iftrue % Compiles the numerics section (followed by \fi)
%\iffalse % Does not compile the numerics section (followed by \fi)

In this section we present the results from a series of numerical
tests illustrating the performance of the proposed method. We employ
the backward Euler method for the time discretization. Let $\Delta t =
T/N$ be the time step, $t_n=n \Delta t$, $n=0, \cdots, N$. Let $d_t\,
u^n := (u^n - u^{n-1})/\Delta t$, where $u^n := u(t_n)$. The fully
discrete method reads: given $(\bp_h^0, \br_h^0)=(\bp_h(0), \br_h(0))$
satisfying \eqref{eq: discrete initial condition}, find $(\bp_h^n,
\br_h^n)\in \bQ_h \times \bS_h$, $n=1, \cdots, N$, such that for all
$(\bq_h, \bs_h)\in \bQ_h \times \bS_h$,
\begin{equation}\label{eq: discrete formulation}
  \arraycolsep=1.7pt
  \begin{array}{rcl}    
\ds d_t\,\cE_1\,(\bp_h^n)(\bq_h)
+ \cA\,(\bp_h^n)(\bq_h) + \cB'\,(\br_h^n)(\bq_h) & = & \bF(\bq_h), \\[1ex]
\ds - \cB\,(\bp_h^n)(\bs_h) & = & \bG(\bs_h).
  \end{array}
\end{equation}
Our implementation is on triangular grids, and it is based on the {\tt FreeFem++} finite element package \cite{freefem}. We use a monolithic scheme, in conjunction with the direct solver {\tt UMFPACK} \cite{umfpack}. We note that iterative solvers suitable for saddle point problems \cite{elman2014finite} could also be utilized. It is shown in \cite{byzz2015} that block-diagonal preconditioners based on split schemes involving individual physics solves can be very effective. Another alternative is non-overlapping domain decomposition methods, see \cite{manu} for a recent work on a fully mixed five-field formulation of the Biot system of poroelasticity. For spatial
discretization we use the MINI elements $\rP_1^b-\rP_1$ for the Stokes
spaces $(\bV_{fh}, \W_{fh})$, where $\rP_1^b$ stands for the space of
continuous piecewise linear polynomials enhanced elementwise by cubic
bubbles, the lowest order Raviart-Thomas elements $\RT_0-\rP_0$ for the
Darcy spaces $(\bV_{ph}, \W_{ph})$, and the $\BDM_1-\rP_0-\rP_1$
elements \cite{brezzi2008mixed} for the elasticity spaces $(\bbX_{ph},
\bV_{sh}, \bbQ_{ph})$. According to \eqref{Lagr-mult}, for the
Lagrange multiplier spaces we choose piecewise constants for
$\Lambda_h$ and discontinuous piecewise linears for $\bLambda_{sh}$. We note that the choice of the $\BDM_1-\rP_0-\rP_1$ spaces for elasticity fits in the framework of the multipoint stress mixed finite element
method \cite{msmfe-simpl}, where the stress and rotation variables can
be locally eliminated, resulting in a very efficient positive definite cell-centered scheme for the displacement.

We present two examples.  Example~1 is used to corroborate the rates
of convergence.  In Example 2 we present
simulations of the coupling of surface and subsurface hydrological
systems, focusing on the qualitative behavior of the solution.

\subsection{Example 1: convergence test}
For the convergence study we consider a test case with
domain $\Omega = (0,1)\times(-1,1)$ and a known analytical solution.
We associate the upper
half with the Stokes flow, while the lower half represents the flow in
the poroelastic structure governed by the Biot system. The physical
parameters are $\bK = \bI$, $\mu = 1$, $\alpha = 1$,
$\alpha_{\BJS} = 1$, $s_0 = 1$, $\lambda_p = 1$, and $\mu_p = 1$.
The solution in the Stokes region is
\begin{equation*}
\ds \bu_f = \pi \cos(\pi t)\begin{pmatrix}\ds -3x+\cos(y) \\[1ex]
\ds y+1 \end{pmatrix}, \quad p_f = e^t\sin(\pi x)\cos(\frac{\pi y}{2}) + 2\pi \cos(\pi t).
\end{equation*}
The Biot solution is chosen accordingly to satisfy the interface conditions at $y = 0$:
\begin{equation*}
\ds \bu_p = \pi e^t \begin{pmatrix} \ds -\cos(\pi x)\cos(\frac{\pi y}{2}) \\[1ex] \ds \frac12\sin(\pi x)\sin(\frac{\pi y}{2}) \end{pmatrix}, 
\quad p_p = e^t\sin(\pi x)\cos(\frac{\pi y}{2}), 
\quad \bbeta_p = \sin(\pi t) \begin{pmatrix} \ds -3x+\cos(y) \\[1ex] \ds y+1 \end{pmatrix}.
\end{equation*}
The right hand side functions $\f_f,\, q_f,\, \f_p$, and $q_p$
are computed using the above
solution. The model problem is complemented with 
Dirichlet boundary conditions and initial data obtained from the true solution.
The total simulation
time for this test case is $T=0.01$ and the time step is $\Delta t =
10^{-3}$. The time step is sufficiently small, so that the time
discretization error does not affect the spatial convergence rates.

In Table \ref{table: example 1-1}, we report errors on a sequence of
refined meshes, which are matching along the interface. We use the
notation $\|\cdot\|_{l^{\infty}(V)}$ and $\|\cdot\|_{l^2(V)}$ to
denote the time-discrete space-time errors. For all errors we
report the $\|\cdot\|_{l^2(V)}$ norms with the exception of the error
$e_{\bsi_p}$, for which we have a bound only in $l^\infty$ in time.
We observe at least $O(h)$ convergence for all norms, which is
consistent with the theoretical results stated in Theorem~\ref{thm:
  error analysis}. The observed $O(h^2)$ convergence for
$\|e_{\bsi_p}\|_{l^{\infty}(\bbL^2(\Omega_p))}$, $\|e_{\bgamma_p}\|_{l^2(\bbQ_p)}\|$,
and $\|e_{\btheta}\|_{l^2(\bLambda_{sh})}$ corresponds to the second order
of approximation in the spaces $\bbX_{ph}$, $\bbQ_{ph}$, and $\bLambda_{sh}$,
respectively,
and indicates that the convergence rates for these
variables are not affected by the lower rate for the rest of
the variables. Next, noting that the analysis in Theorem \ref{thm:
  error analysis} is not restricted to the case of matching grids, we
provide the convergence results obtained with non-matching grids along
the interface. The results in Table \ref{table: example 1-2} are
obtained by setting the ratio between the characteristic mesh sizes to be
$\ds h_{\text{Stokes}}=\frac{5}{8}h_{\text{Biot}}$. The results in Table
\ref{table: example 1-3} are with $\ds
h_{\text{Biot}}=\frac{5}{8}h_{\text{Stokes}}$. The convergence rates in both
tables agree with the statement of Theorem \ref{thm: error analysis}.

\begin{table}[H]
\begin{center}
\begin{tabular}{  c | c | c | c | c | c | c }
\hline
n & $\Vert e_{\bu_f} \Vert_{l^2(\bV_f)}$ & rate & $\Vert e_{p_f}\Vert_{l^2(\W_f)}$ & rate & $\Vert e_{\bsi_p}\Vert_{l^\infty(\bbL^2(\Omega_p))}$ & rate  \\ \hline
8& 7.731e-03 &  0.0   & 2.601e-03 &  0.0  &  7.454e-02 &  0.0 \\ \hline
16& 3.860e-03  & 1.0  & 8.319e-04 &  1.6  &  2.572e-02 &  1.5 \\ \hline
32& 1.929e-03  & 1.0  & 2.759e-04 &  1.6  &  8.775e-03 &  1.6 \\ \hline
64& 9.640e-04  & 1.0  & 9.419e-05 &  1.6  &  2.784e-03 &  1.7 \\ \hline
128& 4.819e-04 & 1.0  & 3.270e-05 &  1.5  &  8.224e-04 &  1.8 \\ \hline
\end{tabular}

\smallskip
\begin{tabular}{  c | c | c | c | c | c | c | c | c}
\hline
n & $\Vert e_{\nabla \cdot \bsi_p}\Vert_{l^2(\bL^2(\Omega_p))}$ & rate & $\Vert e_{\bu_s}\Vert_{l^2(\bV_s)}$ & rate & $\Vert e_{\bgamma_p}\Vert_{l^2(\bbQ_p)}$ & rate & $\Vert e_{\bu_p}\Vert_{l^2(\bL^2(\Omega_p))}$ & rate \\ \hline
8& 1.032e-01 & 0.0 &  7.141e-02 &  0.0 &  1.926e-01 &  0.0  & 1.046e-01 &  0.0   \\ \hline
16& 5.169e-02 & 1.0  & 3.550e-02 &  1.0  &  5.171e-02 &  1.9  & 5.224e-02 &  1.0   \\ \hline
32& 2.586e-02 & 1.0  & 1.773e-02 &  1.0  &  1.372e-02 &  1.9  & 2.612e-02 &  1.0   \\ \hline
64& 1.293e-02 & 1.0  & 8.862e-03 &  1.0  &  3.633e-03 &  1.9  & 1.306e-02 &  1.0   \\ \hline
128& 6.465e-03 & 1.0  & 4.431e-03 &  1.0  & 9.497e-04 &  1.9  & 6.532e-03 &  1.0   \\ \hline
\end{tabular}

\smallskip
\begin{tabular}{  c | c | c | c | c | c | c | c | c}
\hline
n & $\Vert e_{\nabla \cdot \bu_p}\Vert_{l^2(\bL^2(\Omega_p))}$ & rate & $\Vert e_{p_p}\Vert_{l^2(\W_p)}$ & rate
& $\Vert e_{\lambda}\Vert_{l^2(\Lambda_{ph})}$ & rate & $\Vert e_{\btheta}\Vert_{l^2(\bLambda_{sh})}$ & rate \\ \hline
8& 1.223e-01 & 0.0 & 1.033e-01 &  0.0 & 1.140e-01 &  0.0  &  3.232e-02  & 0.0   \\ \hline
16& 5.457e-02 & 1.2 & 5.172e-02 &  1.0   & 5.675e-02 &  1.0 &   6.446e-03 &   2.3   \\ \hline
32& 2.693e-02 & 1.0 & 2.587e-02 &  1.0   &  2.835e-02 &  1.0 &   1.238e-03 &  2.4   \\ \hline
64& 1.442e-02 & 0.9 & 1.294e-02 &  1.0   &  1.417e-02 &  1.0  &  2.328e-04  & 2.4   \\ \hline
128& 9.001e-03 & 0.7 &  6.468e-03 &  1.0   & 7.085e-03 &  1.0  &  4.442e-05  & 2.4   \\ \hline
\end{tabular}
\end{center}
\caption{{\sc Example 1}, Mesh sizes, errors and rates of convergences in matching grids.}\label{table: example 1-1}
\end{table}

\begin{table}[H]
\begin{center}
\begin{tabular}{  c | c | c | c | c | c | c }
\hline
n & $\Vert e_{\bu_f} \Vert_{l^2(\bV_f)}$ & rate & $\Vert e_{p_f}\Vert_{l^2(\W_f)}$ & rate & $\Vert e_{\bsi_p}\Vert_{l^\infty(\bbL^2(\Omega_p))}$ & rate  \\ \hline
8& 1.171e-02 &  0.0 & 8.326e-03  & 0.0 & 8.800e-02 &  0.0     \\ \hline
16& 5.725e-03 &  1.0  & 2.616e-03 &  1.7   & 3.220e-02 &  1.5     \\ \hline
32& 2.835e-03 &  1.0   & 9.239e-04 &  1.5   & 1.084e-02 &  1.6     \\ \hline
64& 1.411e-03 &  1.0   & 3.256e-04 &  1.5   &  3.262e-03 &  1.7    \\ \hline
128& 7.037e-04 &   1.0   & 1.152e-04 &  1.5   & 9.161e-04 &  1.8    \\ \hline
\end{tabular}

\begin{tabular}{  c | c | c | c | c | c | c | c | c}
\hline
n & $\Vert e_{\nabla \cdot \bsi_p}\Vert_{l^2(\bL^2(\Omega_p))}$ & rate & $\Vert e_{\bu_s}\Vert_{l^2(\bV_s)}$ & rate & $\Vert e_{\bgamma_p}\Vert_{l^2(\bbQ_p)}$ & rate & $\Vert e_{\bu_p}\Vert_{l^2(\bL^2(\Omega_p))}$ & rate \\ \hline
8& 1.032e-01  & 0.0 & 7.632e-02 &  0.0 & 2.255e-01 &  0.0  & 1.049e-01 &  0.0   \\ \hline
16& 5.170e-02 &  1.0  & 3.810e-02 &  1.0   & 6.617e-02 &  1.8  & 5.226e-02 &  1.0   \\ \hline
32& 2.587e-02 &  1.0  & 1.905e-02 &  1.0   & 1.955e-02 &  1.8  & 2.613e-02 &  1.0   \\ \hline
64& 1.293e-02 &  1.0  & 9.524e-03 &  1.0   & 5.773e-03 &  1.8  & 1.306e-02 &  1.0   \\ \hline
128& 6.467e-03 &  1.0  & 4.762e-03 &  1.0   & 1.638e-03  & 1.8  & 6.532e-03  & 1.0   \\ \hline
\end{tabular}

\begin{tabular}{  c | c | c | c | c | c | c | c | c}
\hline
n & $\Vert e_{\nabla \cdot \bu_p}\Vert_{l^2(\bL^2(\Omega_p))}$ & rate & $\Vert e_{p_p}\Vert_{l^2(\W_p)}$ & rate
& $\Vert e_{\lambda}\Vert_{l^2(\Lambda_{ph})}$ & rate & $\Vert e_{\btheta}\Vert_{l^2(\bLambda_{sh})}$ & rate \\ \hline
8&  1.323e-01 &  0.0 & 1.033e-01 &  0.0 &  1.141e-01 &  0.0  &  3.272e-02 &  0.0   \\ \hline
16& 5.742e-02 &  1.2  & 5.172e-02 &  1.0   & 5.675e-02  & 1.0  &  6.733e-03 &  2.3   \\ \hline
32& 2.738e-02 &  1.1  & 2.587e-02 &  1.0   & 2.835e-02 &  1.0  &  1.314e-03 &  2.4   \\ \hline
64& 1.448e-02 &  0.9  & 1.294e-02 &  1.0   & 1.417e-02 &  1.0  &  2.502e-04 &  2.4   \\ \hline
128& 9.007e-03 &  0.7  & 6.468e-03 &  1.0  & 7.085e-03  & 1.0  &  4.820e-05 &  2.4   \\ \hline
\end{tabular}
\end{center}
\caption{{\sc Example 1}, Mesh sizes, errors and rates of convergences with non-matching grids, using finer mesh in the Stokes region.}\label{table: example 1-2}
\end{table}

\begin{table}[H]
\begin{center}
\begin{tabular}{  c | c | c | c | c | c | c }
\hline
n & $\Vert e_{\bu_f} \Vert_{l^2(\bV_f)}$ & rate & $\Vert e_{p_f}\Vert_{l^2(\W_f)}$ & rate & $\Vert e_{\bsi_p}\Vert_{l^\infty(\bbL^2(\Omega_p))}$ & rate  \\ \hline
8& 7.203e-03 &  0.0 & 5.066e-03 &  0.0 & 1.661e-01 &  0.0     \\ \hline
16& 3.561e-03 &  1.0  & 1.404e-03 &  1.9   & 6.387e-02 &  1.4     \\ \hline
32& 1.768e-03 &  1.0   & 4.843e-04 &  1.5   & 2.298e-02 &  1.5     \\ \hline
64& 8.807e-04 &  1.0   & 1.697e-04 &  1.5   &  7.441e-03 &  1.6    \\ \hline
128& 4.396e-04 &  1.0   & 5.977e-05 &  1.5   & 2.178e-03  & 1.8    \\ \hline
\end{tabular}

\begin{tabular}{  c | c | c | c | c | c | c | c | c}
\hline
n & $\Vert e_{\nabla \cdot \bsi_p}\Vert_{l^2(\bL^2(\Omega_p))}$ & rate & $\Vert e_{\bu_s}\Vert_{l^2(\bV_s)}$ & rate & $\Vert e_{\bgamma_p}\Vert_{l^2(\bbQ_p)}$ & rate & $\Vert e_{\bu_p}\Vert_{l^2(\bL^2(\Omega_p))}$ & rate \\ \hline
8& 1.644e-01 &  0.0 & 1.230e-01 &  0.0 &  4.521e-01 &  0.0  & 1.698e-01 &  0.0   \\ \hline
16& 8.264e-02 &  1.0  & 6.100e-02 &  1.0   & 1.504e-01 &  1.6  & 8.374e-02 &  1.0   \\ \hline
32& 4.137e-02 &  1.0  & 3.048e-02 &  1.0   & 4.373e-02 &  1.8  & 4.180e-02 &  1.0   \\ \hline
64& 2.069e-02 &  1.0  & 1.524e-02 &  1.0   &  1.293e-02  & 1.8  & 2.090e-02 &  1.0   \\ \hline
128& 1.035e-02 &  1.0  & 7.619e-03 &  1.0   &  3.798e-03 &  1.8  & 1.045e-02 &  1.0   \\ \hline
\end{tabular}

\begin{tabular}{  c | c | c | c | c | c | c | c | c}
\hline
n & $\Vert e_{\nabla \cdot \bu_p}\Vert_{l^2(\bL^2(\Omega_p))}$ & rate & $\Vert e_{p_p}\Vert_{l^2(\W_p)}$ & rate
& $\Vert e_{\lambda}\Vert_{l^2(\Lambda_{ph})}$ & rate & $\Vert e_{\btheta}\Vert_{l^2(\bLambda_{sh})}$ & rate \\ \hline
8& 2.430e-01 &  0.0 & 1.649e-01  & 0.0 &  1.849e-01 &  0.0  &  9.021e-02 &  0.0   \\ \hline
16& 1.004e-01 &  1.3  & 8.270e-02  & 1.0   &  9.101e-02 &  1.0 &   1.977e-02 &  2.2   \\ \hline
32& 4.474e-02 &  1.2  & 4.138e-02 &  1.0   &  4.538e-02 &  1.0 &   3.990e-03 &  2.3   \\ \hline
64& 2.203e-02 &  1.0  & 2.070e-02 &  1.0   &  2.268e-02 &  1.0  &  7.683e-04 &  2.4   \\ \hline
128& 1.215e-02 &  0.9  & 1.035e-02  & 1.0   &  1.134e-02 &  1.0 &   1.461e-04 &  2.4   \\ \hline
\end{tabular}
\end{center}
\caption{{\sc Example 1}, Mesh sizes, errors and rates of convergences with non-matching grids, using finer mesh in the Biot region.}\label{table: example 1-3}
\end{table}

\subsection{Example 2: coupling of surface and subsurface hydrological systems}
In this example, we illustrate the behavior of the method for a
problem motivated by the coupling of surface and subsurface
hydrological systems and test its robustness with respect to physical
parameters. On the domain $\Omega =
(0,2)\times(-1,1)$, we associate the upper half with surface flow, such as
lake or river, modeled by the Stokes equations
while the lower half represents subsurface flow in a poroelastic aquifer,
governed by the Biot system. In each subdomain, we construct $64\times 64$ rectangular grid, which is then sub-divided into triangles, resulting in $8192$ finite elements in each region. The appropriate interface conditions are enforced along the interface $y = 0$.
We consider three cases with different values of $\bK$, $s_0$,
$\lambda_p$ and $\mu_p$, as described in Table~\ref{param},
\begin{table}[ht!]
\begin{center}
\begin{tabular}{  c | c | c | c | c }
\hline
 & $\bK$ & $s_0$ & $\lambda_p$ & $\mu_p$\\  \hline
Case 1 & $ \bI$ & $1$ & $1$ & $1$ \\ \hline
Case 2 & $10^{-4}\times \bI$ & $10^{-4}$ & $10^6$ & $1$ \\ \hline
Case 3 & $10^{-4}\times \bI$ & $10^{-4}$ & $10^6$ & $10^6$ \\ \hline
\end{tabular}
\caption{Set of parameters for the sensitivity analysis}
\label{param}
\end{center}
\end{table}
while we set the rest of the physical parameters to be $\mu = 1$,
$\alpha = 1$, and $\alpha_{\BJS} = 1$. In the discussion we will also refer
to the Young's modulus $E$ and the Poisson's ratio $\nu$, which are
related to the Lam\'e  coefficients via 
\begin{equation*}
\ds \nu = \frac{\lambda_p}{2(\lambda_p+\mu_p)}, \quad 
E=\frac{(3\lambda_p+2\mu_p)\mu_p}{\lambda_p+\mu_p}.
\end{equation*}
The body forces and
external source are zero, as well as the initial conditions. The flow
is driven by a parabolic fluid velocity on the left boundary of
fluid region. The boundary conditions are as follows:
\begin{gather*}
  \ds \bu_f = (-40y(y-1) \ \  0)^{\rt}  \qon  \Gamma_{f,{\rm left}},
  \quad \bu_f = \0  \qon  \Gamma_{f,{\rm top}}\cup\Gamma_{f,{\rm right}}, \\[1ex]
\ds p_p = 0 \qan \bsi_p\bn_p = \0   \qon  \Gamma_{p,{\rm bottom}}, \\[1ex]
\ds \bu_p\cdot\bn_p = 0 \qan \bu_{s} = \0  \qon \Gamma_{p,{\rm left}}\cup \Gamma_{p,right},
\end{gather*}
The simulation is run for a total time $T=3$ with a time step 
$\Delta t = 0.06$.

For each case, we present the plots of
computed velocities, first and second columns of stresses (top plots),
first column components of poroelastic stress (middle plots),
displacement and Darcy pressure (bottom plots) at final time $T=3$.

Case 1 focuses on the qualitative behavior of the solution. The
computed solution at the final time $T=3$ is shown in Figure~\ref{fig:
  Example 3-1}. On the top left, the arrows represent the velocity
vectors $\bu_f$ and $\bu_p + \partial_t\bbeta_p$ in the two regions,
while the color shows the vertical components of these vectors. The
other two plots on the top show the computed stress. The arrows in
both plots represent the second columns of the negative stresses
$-(\bsi_{f,12},\bsi_{f,22})^\rt$ and
$-(\bsi_{p,12},\bsi_{p,22})^\rt$. The colors show $-\bsi_{f,12}$ and
$-\bsi_{p,12}$ in the middle plot and $-\bsi_{f,22}$ and
$-\bsi_{p,22}$ in the right plot. Since the Stokes stress is much
larger than the poroelastic stress, the arrows in the fluid region are scaled
by a factor $1/5$ for visualization purpose and the color scale is
more suitable for the Stokes region. The poroelastic stresses
are presented separately in the middle row with their own color range.
The bottom plots show the displacement vector and its magnitude on the
left and the poroelastic pressure on the right.

From the velocity plot we observe that the fluid is driven into the
poroelastic medium due to zero pressure at the bottom, which simulates
gravity. The mass conservation $\ds \bu_{f}\cdot \bn_f + (\partial_t
\bbeta_p +\bu_{p})\cdot \bn_p = 0$ on the interface with $\bn_p =
(0,1)^\rt$ indicates continuity of second components of these two
velocity vectors, which is observed from the color plot of the
velocity. In addition, the conservation of momentum $\ds \bsi_f \bn_f
+\bsi_p \bn_p=0$ implies that $-\bsi_{f,12}=-\bsi_{p,12}$ and
$-\bsi_{f,22}=-\bsi_{p,22}$ on the interface. These conditions are
verified from the two stress color plots on the top row. We observe large
fluid stress near the top boundary, which is due to the no slip
condition there, as well as large fluid stress along the interface,
which is due to the slip with friction interface condition.  A
singularity in the left lower corner appears due to the mismatch in
inflow boundary conditions between the fluid and poroelastic
regions. The bottom plots show that the infiltration of fluid from the
Stokes region into the poroelastic region causes deformation of the
medium and larger Darcy pressure. Furthermore, comparing the right middle and bottom
plots, we note the match along the interface between $-\bsi_{p,22}$ and $p_p$, which
is consistent with the balance of force and momentum conservation
conditions $-(\bsi_f\bn_f)\cdot\bn_f = p_p$ and $\bsi_f\bn_f + \bsi_p\bn_p = 0$,
respectively.

\begin{figure}[ht!]
\begin{center}
\includegraphics[width=0.32\textwidth]{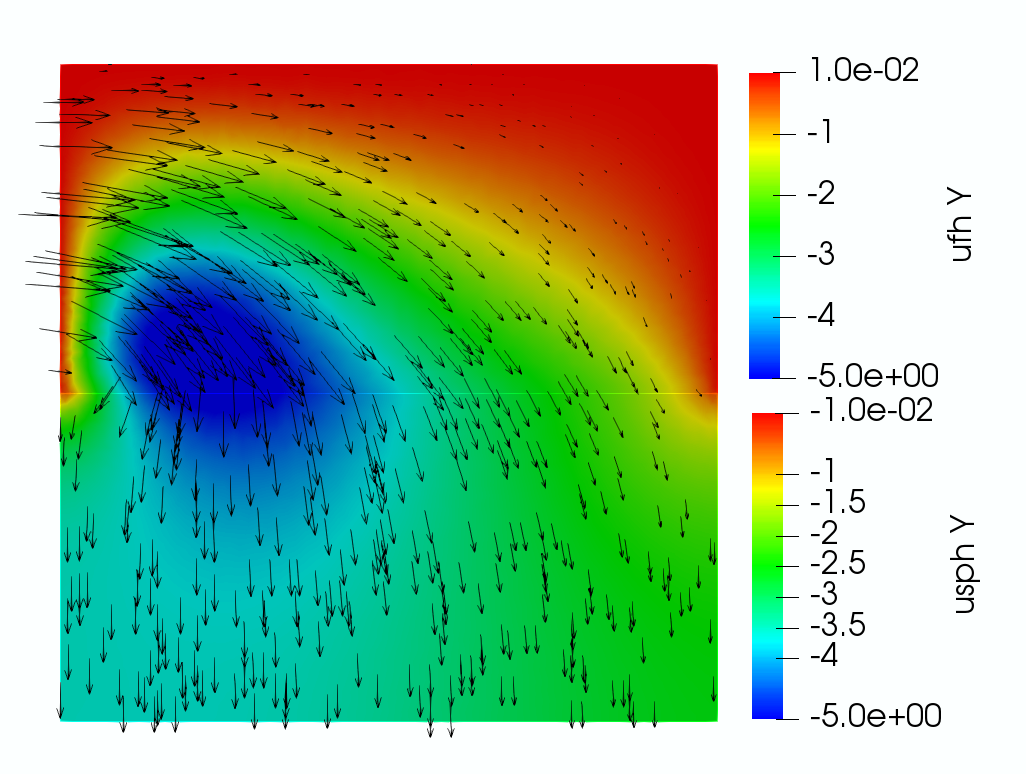}
\includegraphics[width=0.32\textwidth]{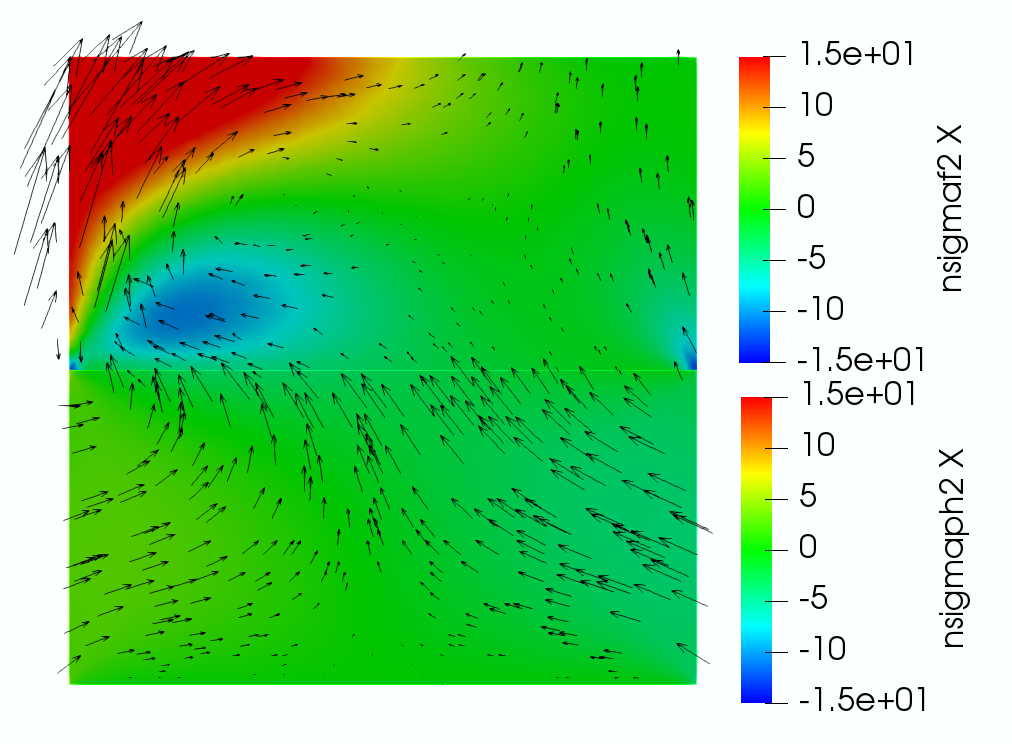}
\includegraphics[width=0.32\textwidth]{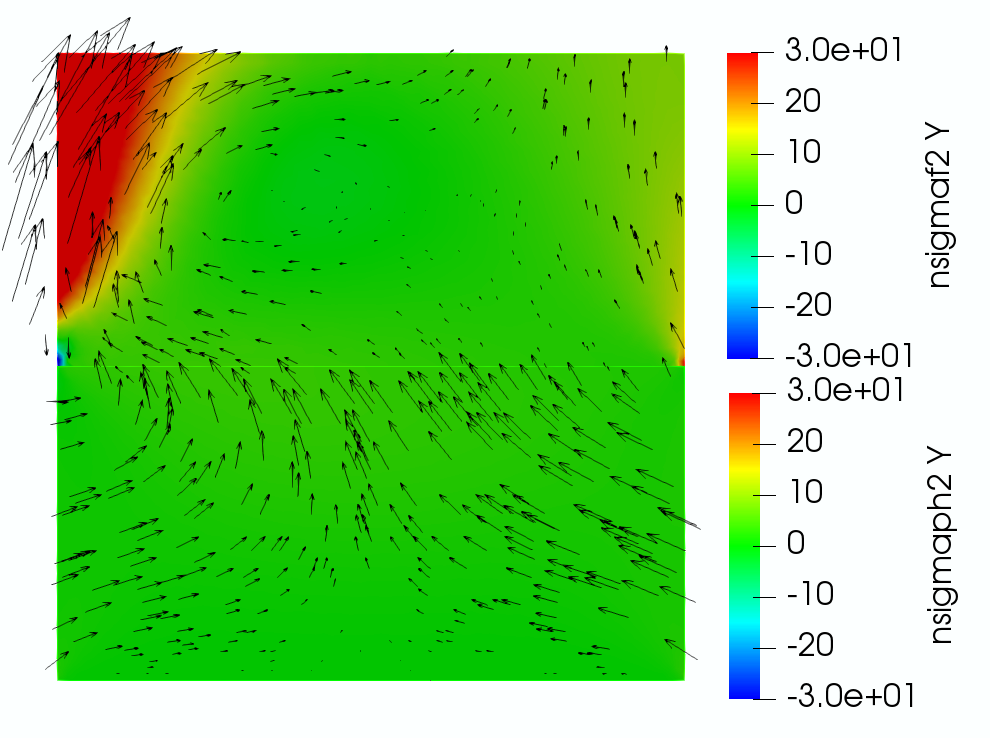}\\
\includegraphics[width=0.4\textwidth]{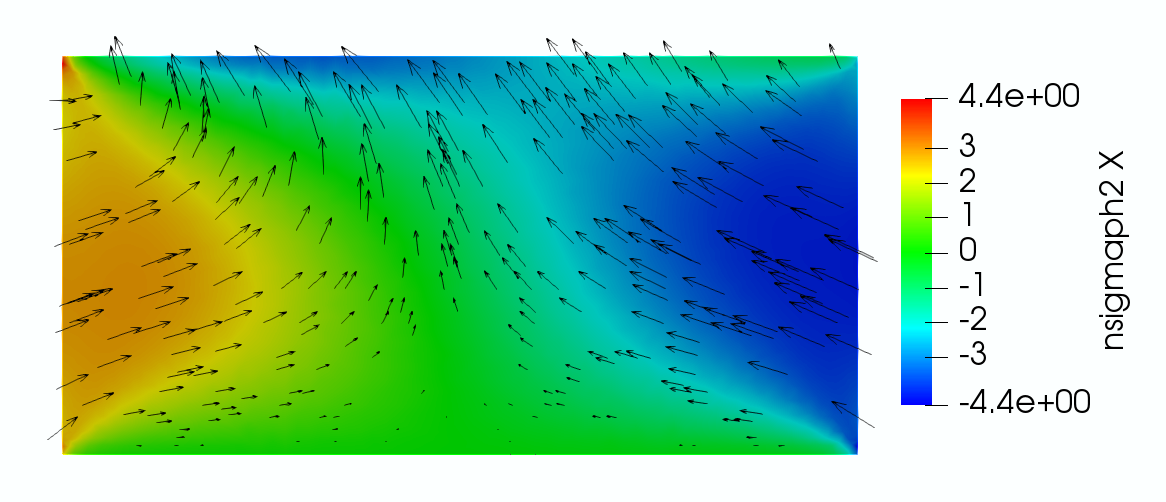}
\includegraphics[width=0.4\textwidth]{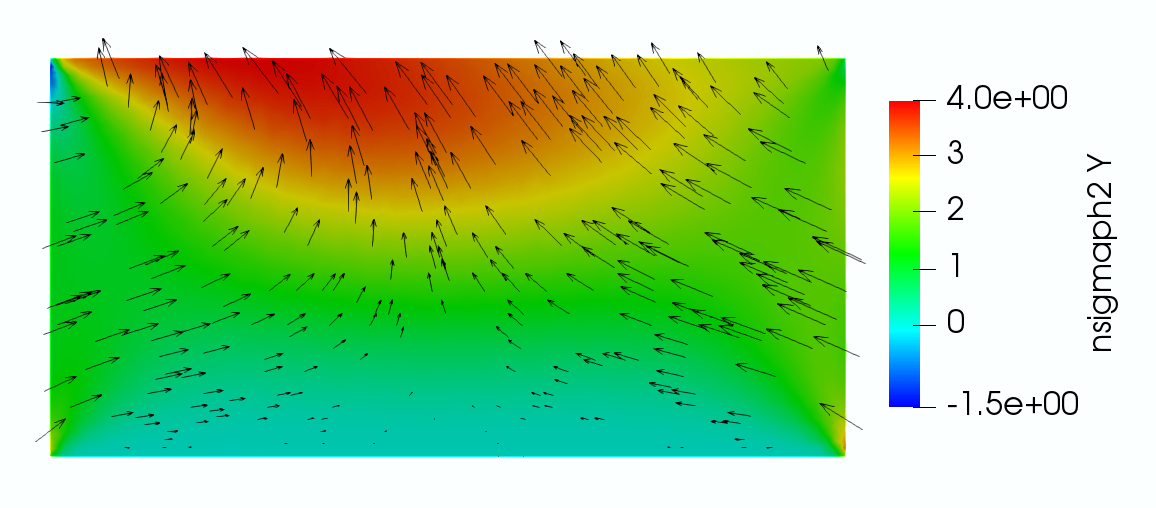}\\
\includegraphics[width=0.4\textwidth]{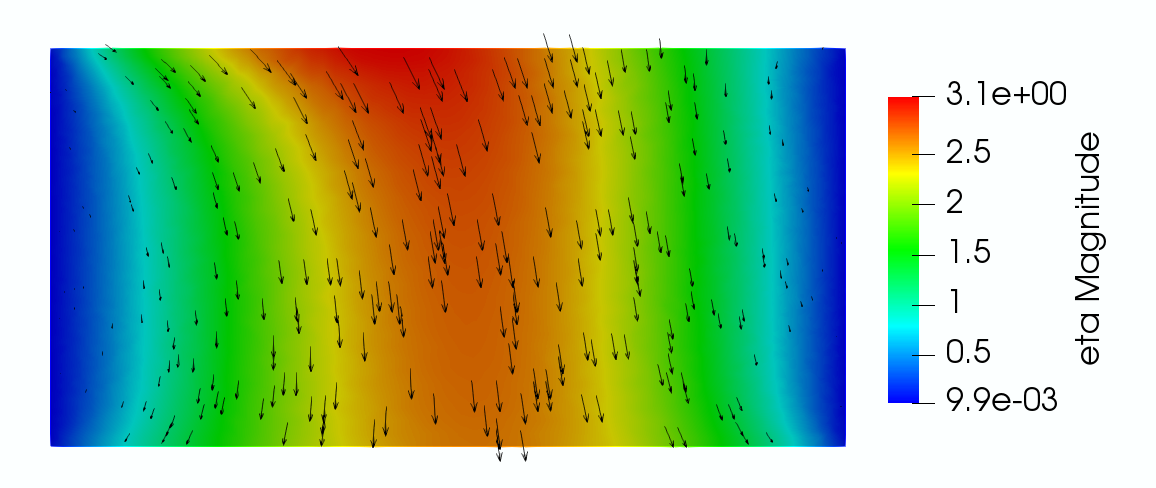}
\includegraphics[width=0.4\textwidth]{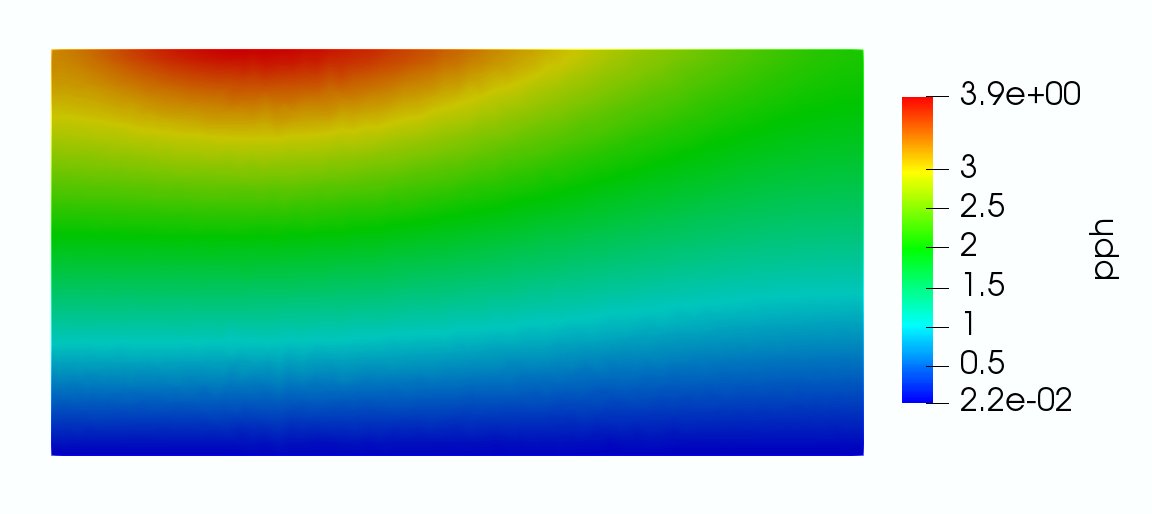}
\end{center}
\caption{Example 2, Case 1, $\bK=\bI$, $s_0=1$, $\lambda_p=1$,
  $\mu_p=1$.  Computed solution at final time $T=3$. Top left:
  velocities $\bu_f$ and $\bu_p + \partial_t\bbeta_p$ (arrows),
  $\bu_{f,2}$ and $\bu_{p,2} + \partial_t\bbeta_{p,2}$ (color). Top
  middle and right: stresses $-(\bsi_{f,12},\bsi_{f,22})^\rt$ and
  $-(\bsi_{p,12},\bsi_{p,22})^\rt$ (arrows); top middle:
  $-\bsi_{f,12}$ and $-\bsi_{p,12}$ (color); top right: $-\bsi_{f,22}$
  and $-\bsi_{p,22}$ (color). Middle: poroelastic stress
  $-(\bsi_{p,12},\bsi_{p,22})^\rt$ (arrows); middle left:
  $-\bsi_{p,12}$ (color); middle right: $-\bsi_{p,22}$ (color).
  Bottom left: displacement $\bbeta_p$ (arrows), $|\bbeta_p|$ (color).
  Bottom right: Darcy pressure $p_p$.}
\label{fig: Example 3-1}
\end{figure}

\begin{figure}[ht!]
\begin{center}
\includegraphics[width=0.32\textwidth]{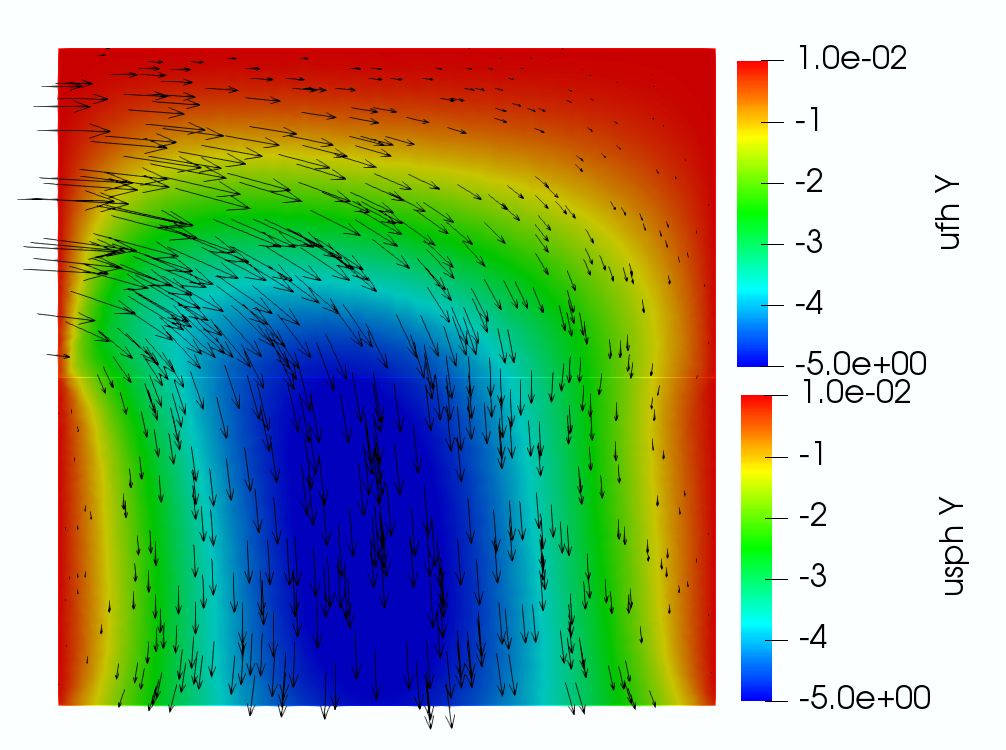}
\includegraphics[width=0.32\textwidth]{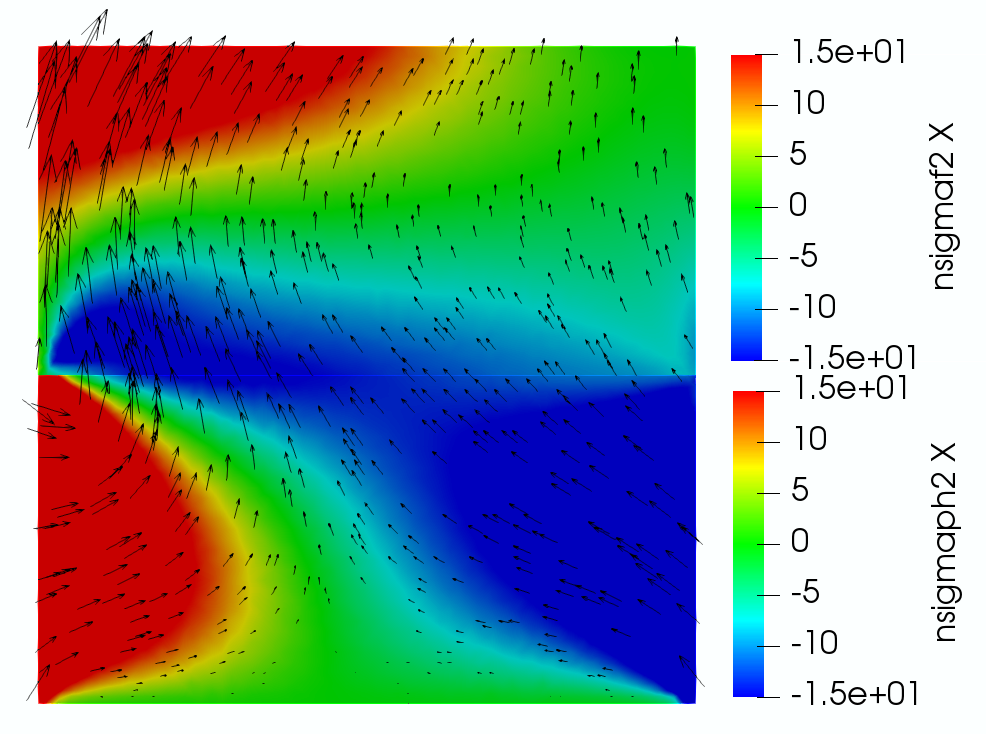}
\includegraphics[width=0.32\textwidth]{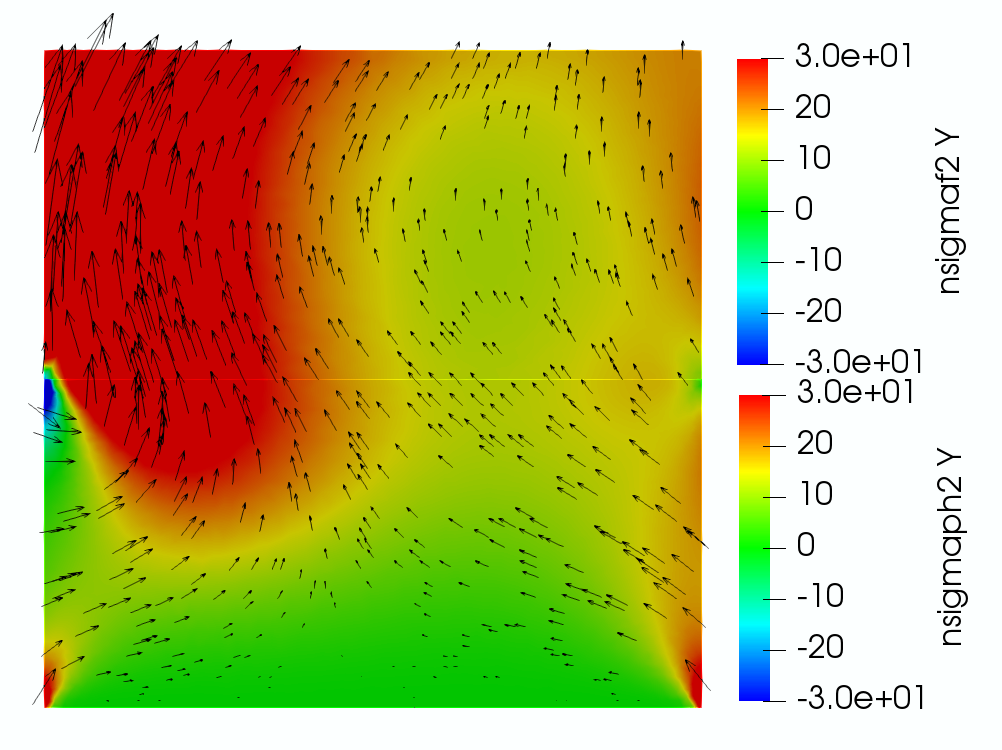}
\includegraphics[width=0.4\textwidth]{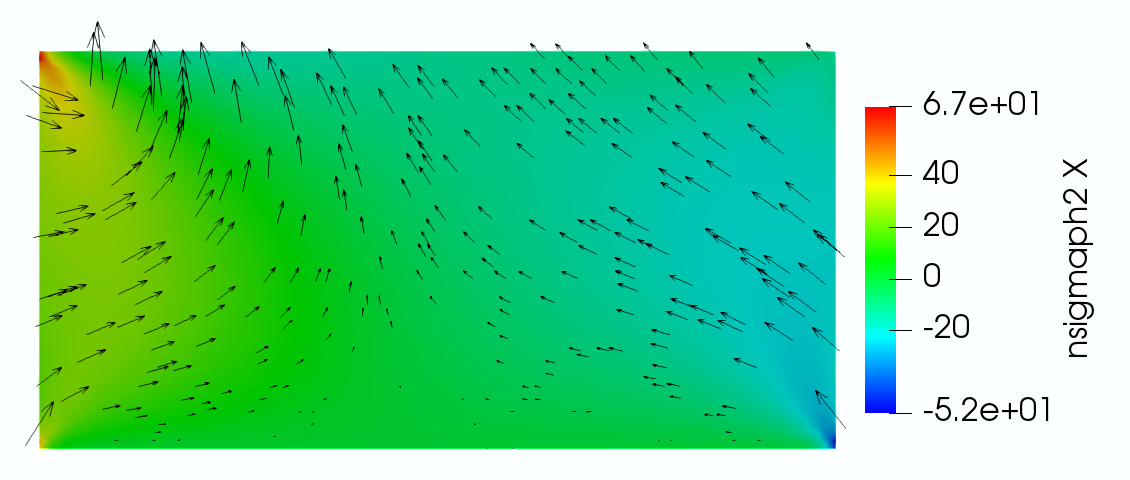}
\includegraphics[width=0.4\textwidth]{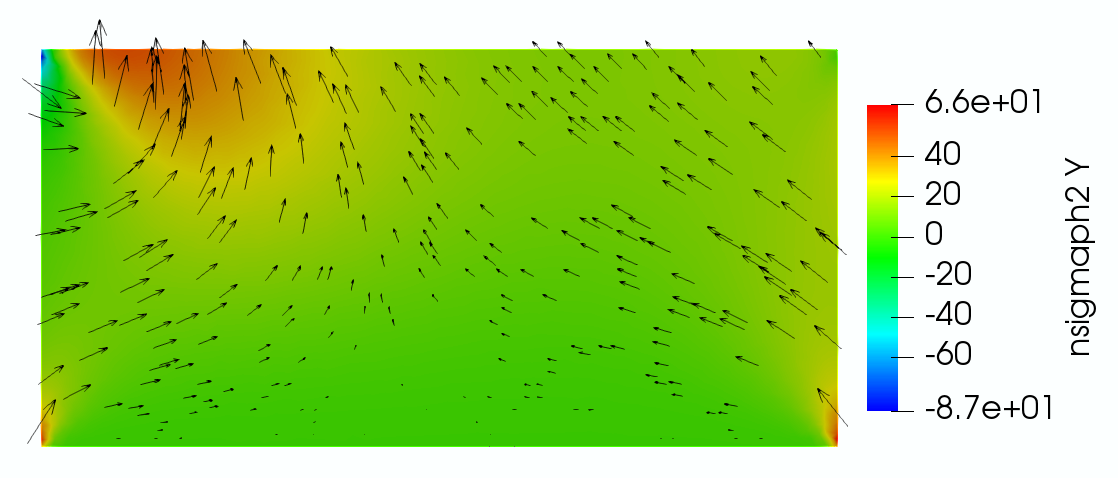}\\
\includegraphics[width=0.4\textwidth]{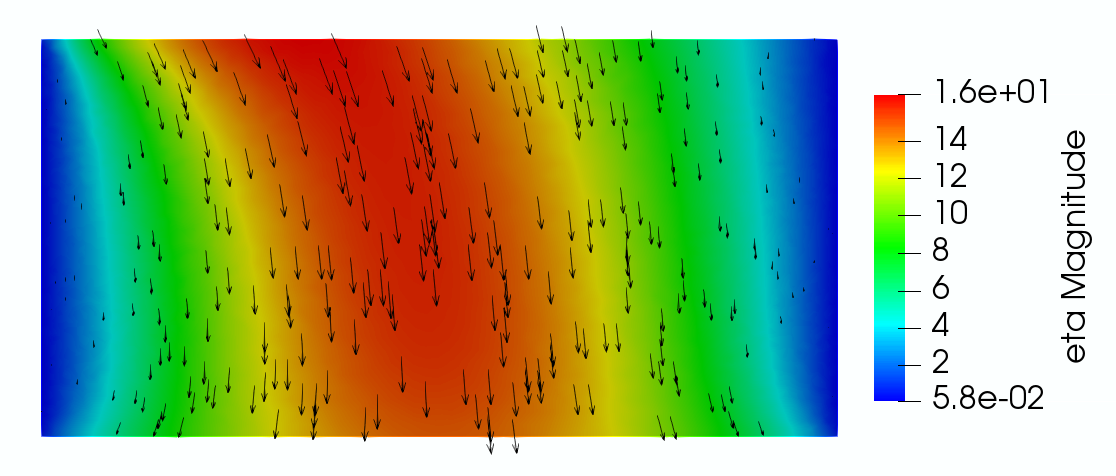}
\includegraphics[width=0.4\textwidth]{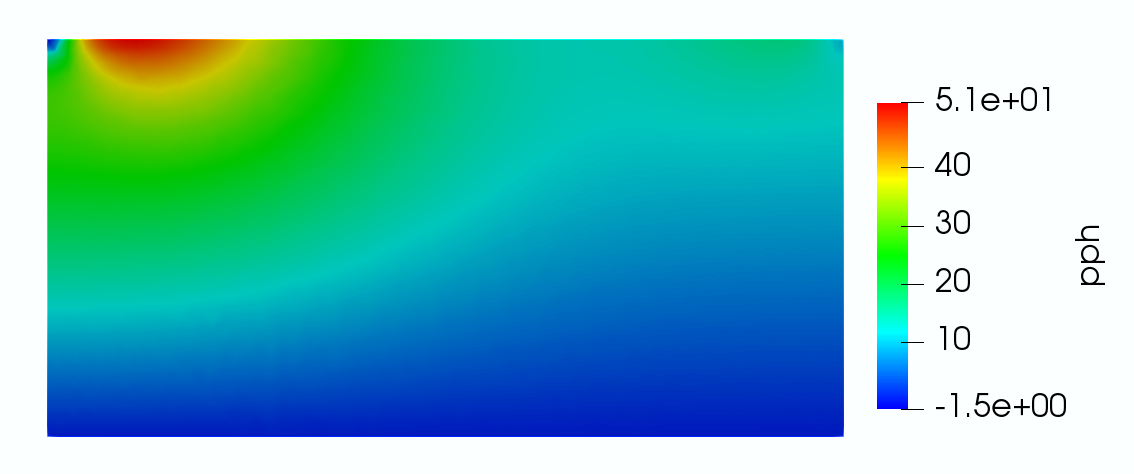}
\end{center}
\caption{Example 2, Case 2, $\bK=10^{-4} \times \bI$, $s_0=10^{-4}$,
  $\lambda_p=10^6$, $\mu_p=1$. Computed solution at final time $T=3$. Top left:
  velocities $\bu_f$ and $\bu_p + \partial_t\bbeta_p$ (arrows),
  $\bu_{f,2}$ and $\bu_{p,2} + \partial_t\bbeta_{p,2}$ (color). Top
  middle and right: stresses $-(\bsi_{f,12},\bsi_{f,22})^\rt$ and
  $-(\bsi_{p,12},\bsi_{p,22})^\rt$ (arrows); top middle:
  $-\bsi_{f,12}$ and $-\bsi_{p,12}$ (color); top right: $-\bsi_{f,22}$
  and $-\bsi_{p,22}$ (color). Middle: poroelastic stress
  $-(\bsi_{p,12},\bsi_{p,22})^\rt$ (arrows); middle left:
  $-\bsi_{p,12}$ (color); middle right: $-\bsi_{p,22}$ (color).
  Bottom left: displacement $\bbeta_p$ (arrows), $|\bbeta_p|$ (color).
  Bottom right: Darcy pressure $p_p$.}
  \label{fig: Example 3-2}
\end{figure}

\begin{figure}[ht!]
\begin{center}
\includegraphics[width=0.32\textwidth]{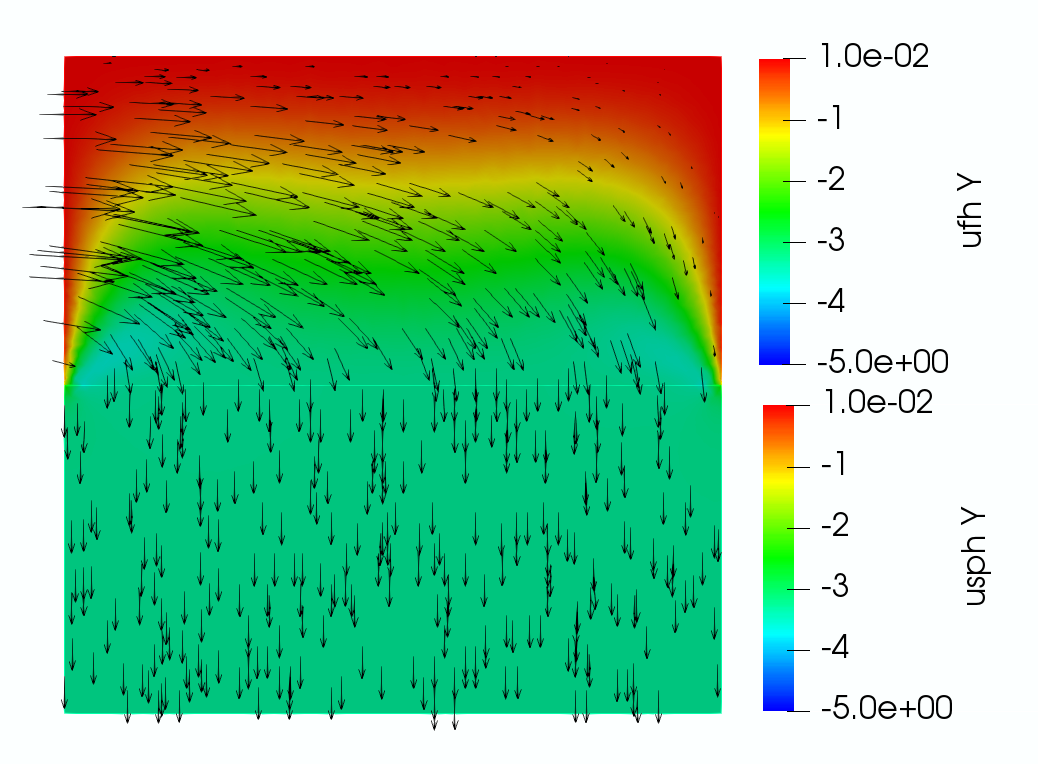}
\includegraphics[width=0.32\textwidth]{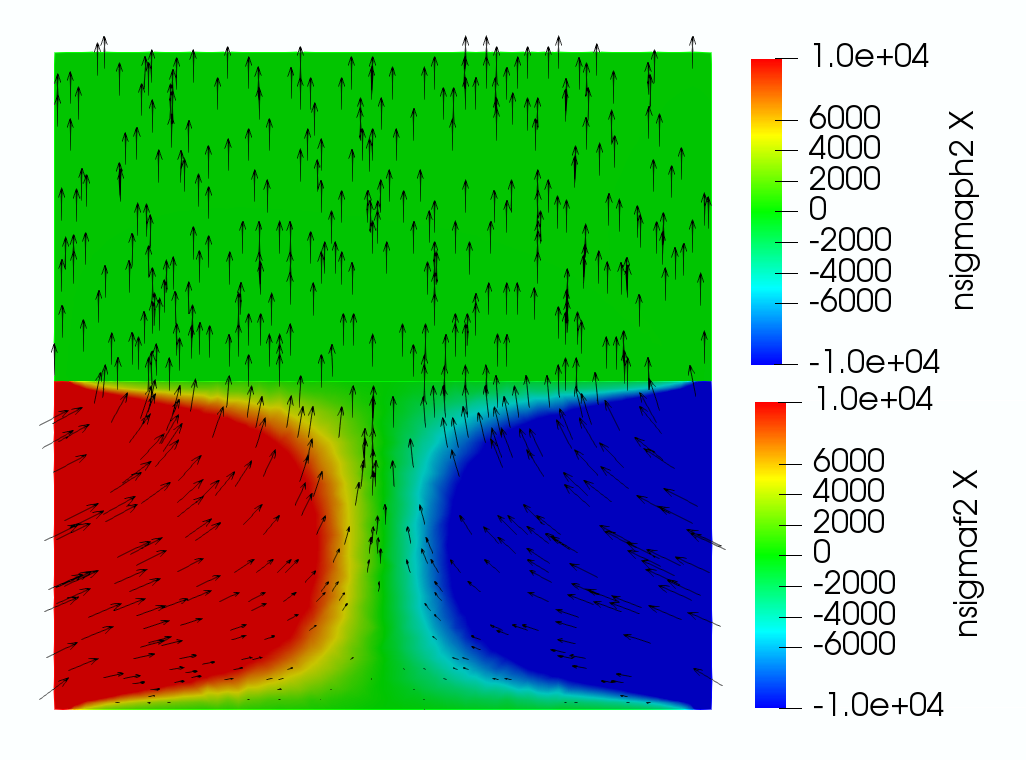}
\includegraphics[width=0.32\textwidth]{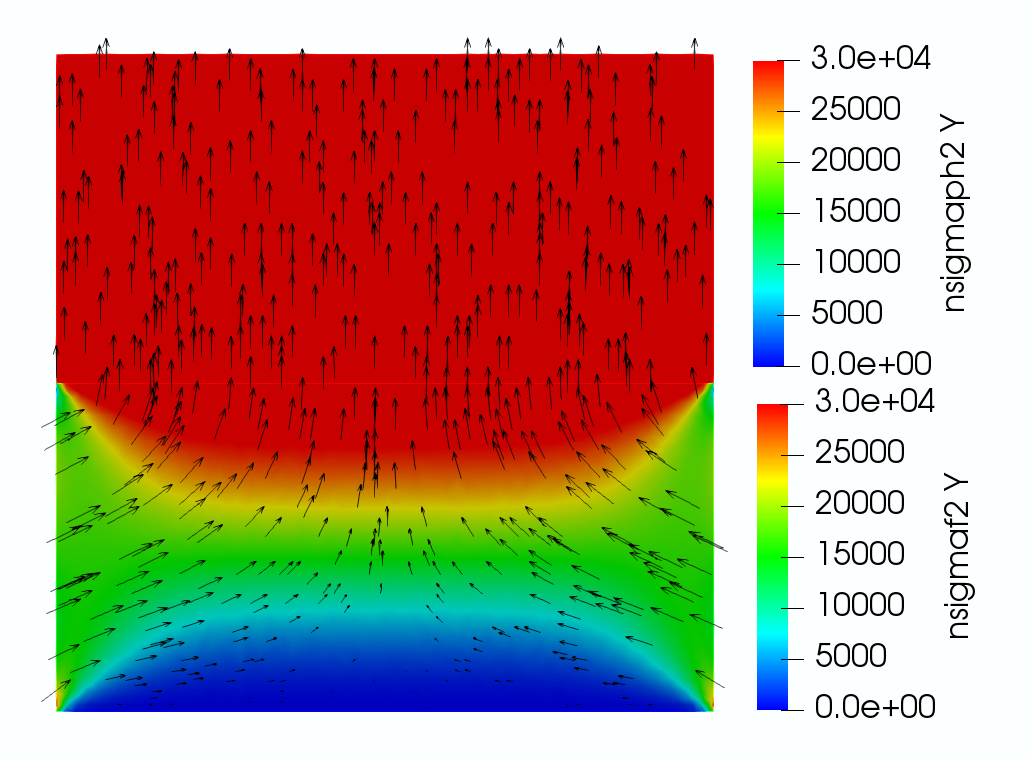}\\
\includegraphics[width=0.4\textwidth]{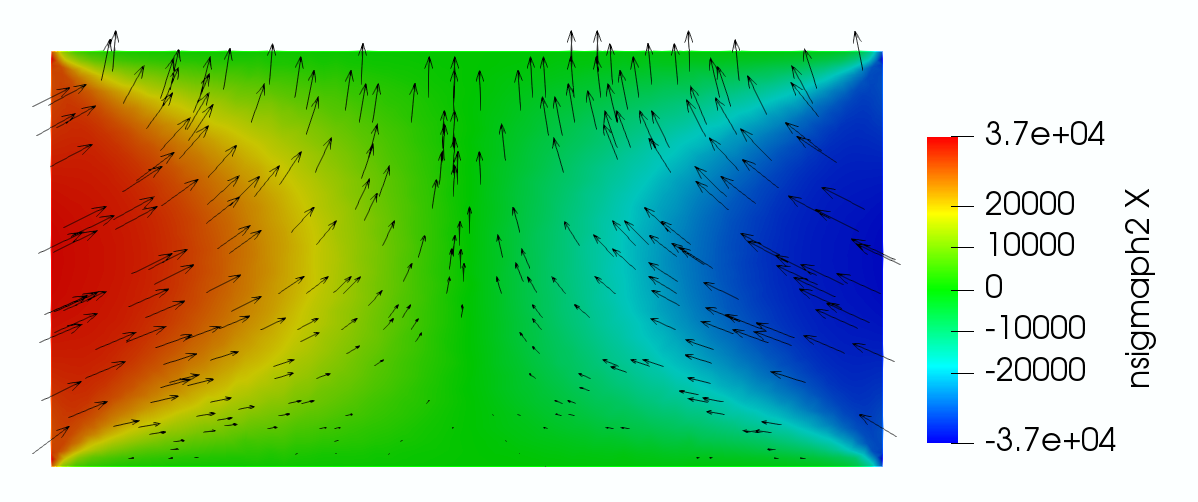}
\includegraphics[width=0.4\textwidth]{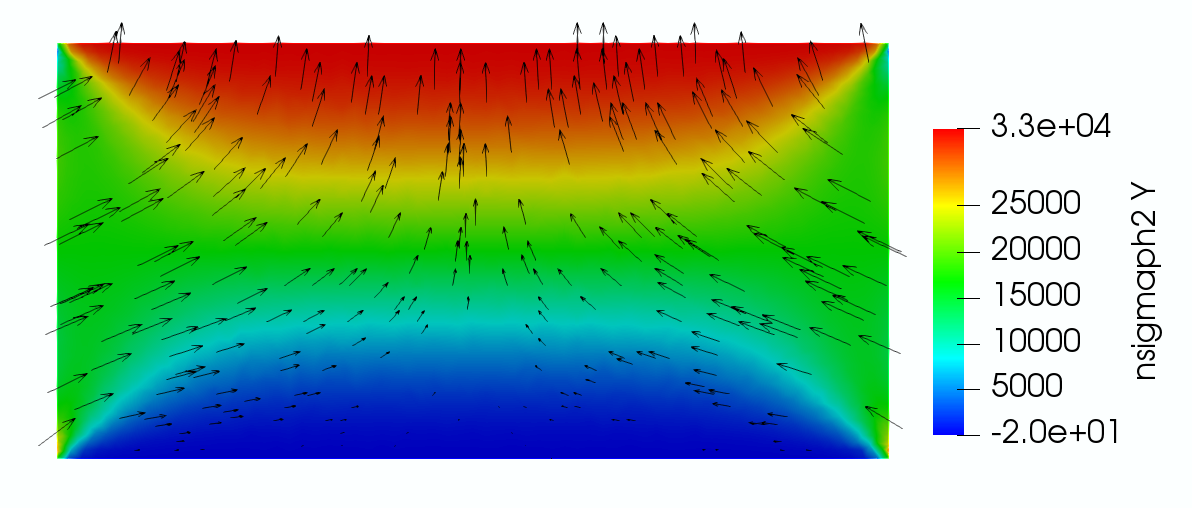}\\
\includegraphics[width=0.4\textwidth]{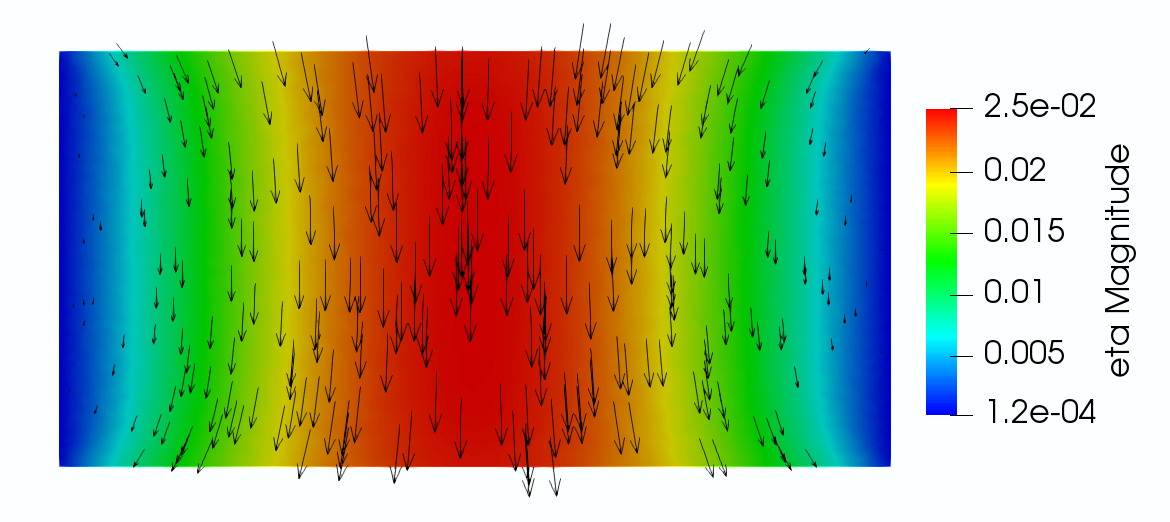}
\includegraphics[width=0.4\textwidth]{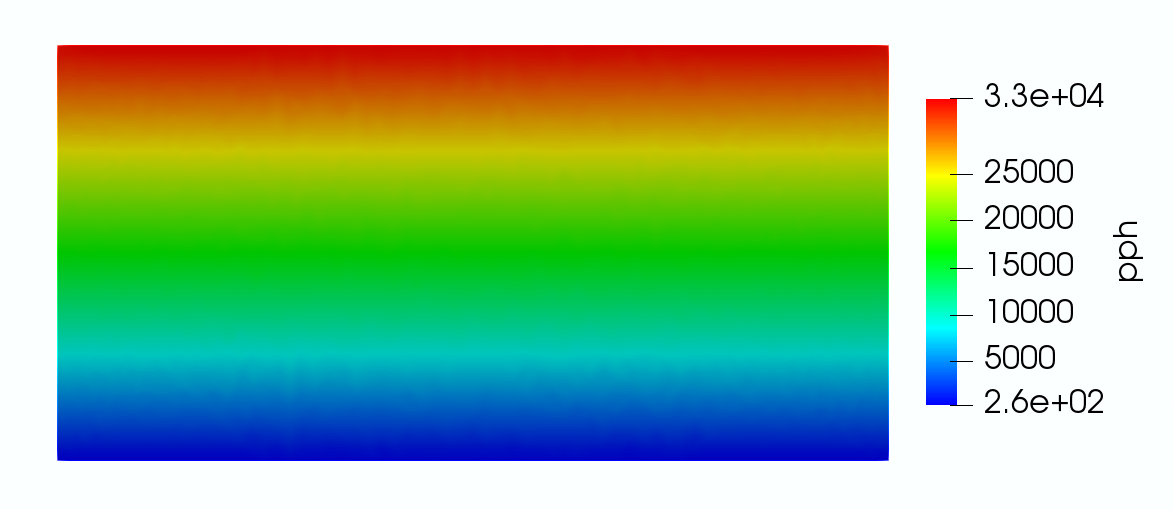}
\end{center}
\caption{Example 2, Case 3, $\bK=10^{-4} \times \bI$, $s_0=10^{-4}$,
  $\lambda_p=10^6$, $\mu_p=10^6$. Computed solution at final time $T=3$. Top left:
  velocities $\bu_f$ and $\bu_p + \partial_t\bbeta_p$ (arrows),
  $\bu_{f,2}$ and $\bu_{p,2} + \partial_t\bbeta_{p,2}$ (color). Top
  middle and right: stresses $-(\bsi_{f,12},\bsi_{f,22})^\rt$ and
  $-(\bsi_{p,12},\bsi_{p,22})^\rt$ (arrows); top middle:
  $-\bsi_{f,12}$ and $-\bsi_{p,12}$ (color); top right: $-\bsi_{f,22}$
  and $-\bsi_{p,22}$ (color). Middle: poroelastic stress
  $-(\bsi_{p,12},\bsi_{p,22})^\rt$ (arrows); middle left:
  $-\bsi_{p,12}$ (color); middle right: $-\bsi_{p,22}$ (color).
  Bottom left: displacement $\bbeta_p$ (arrows), $|\bbeta_p|$ (color).
  Bottom right: Darcy pressure $p_p$.}
  \label{fig: Example 3-3}
\end{figure}

In Case 2 we test the model for a problem that exhibits both locking
regimes for poroelasticity: 1) small permeability and storativity and
2) almost incompressible material \cite{ Yi-Biot-locking}. In
particular, we take $\bK=10^{-4} \times \bI$ and
$s_0=10^{-4}$. Furthermore, the choice $\lambda_p=10^6$, $\mu_p=1$
results in Poisson's ratio $\nu = 0.4999995$.  The computed solution
does not exhibit locking or oscillations. The behavior is
qualitatively similar to Case 1, with larger fluid and poroelastic
stresses and a Darcy pressure gradient.

In Case 3, the Lam\'e coefficient $\mu_p$ is increased from $1$ to
$10^6$, resulting in a much stiffer poroelastic medium, which is
typical in subsurface flow applications. The solution is again free
of locking effects or oscillations, but it differs significantly from
Case 2, including three orders of magnitude larger stresses and Darcy
pressure, as well as smaller displacement and Darcy velocity. 

\fi

\section{Conclusions}\label{sec:conclusions}

In this paper we developed and analyzed a new mixed elasticity
formulation for the Stokes--Biot problem, as well as its mixed finite
element approximation. We consider a five-field Biot formulation based
on a weakly symmetric stress--displacement--rotation elasticity
formulation and a mixed velocity--pressure Darcy formulation. The
classical velocity--pressure formulation is used for the Stokes
system. Suitable Lagrange multipliers are introduced to enforce weakly
the balance of force, slip with friction, and continuity of normal
flux on the interface. The advantages of the resulting mixed finite
element method, compared to previous works, include local momentum
conservation, accurate stress with continuous normal component, and
robustness with respect to the physical parameters. In particular, the
numerical results indicate locking-free and oscillation-free behavior
in the regimes of small storativity and permeability, as well as for
almost incompressible media.

\bibliographystyle{abbrv}
\bibliography{mixedelasticity}

\end{document}